%% file: TQFT.tex
\documentclass{amsart}
\usepackage{amssymb}
\usepackage{graphics}
\usepackage{amscd}
\usepackage{enumerate}
\usepackage{comment}
\usepackage[all]{xy}
\usepackage{textcomp}
\usepackage{comment}
\usepackage{url}
\usepackage{graphicx}
\usepackage{color}
\newtheorem*{thm*}{Theorem}
\newtheorem{thm}{Theorem}[section]
\newtheorem{cor}[thm]{Corollary}
\newtheorem{lem}[thm]{Lemma}
\newtheorem{prop}[thm]{Proposition}

\theoremstyle{definition}
\newtheorem{defn}[thm]{Definition}

\theoremstyle{remark}
\newtheorem{rem}[thm]{Remark}
\newtheorem{ex}[thm]{Example}
\numberwithin{equation}{section}

\newcommand{\eps}{\varepsilon}

\newcommand{\im}{\mathop{\textrm{Im}}\nolimits}
\newcommand{\ol}{\overline}

\def\a{\alpha}
\def\A{\mathcal{A}}
\def\AA{\mathbb{A}}

\def\g{\gamma}
\def\d{\delta}
\def\CC{\mathcal{C}}

\def\MM{\mathcal{M}}
\def\S{\Sigma}
\def\SS{\mathbb{S}}
\def\l{\lambda}

\def\W{\mathcal{W}}
\def\oW{\overline{W}}

\def\Z{\mathbb{Z}}
\def\PP{\mathbb{P}}

\def\R{\mathbb{R}}
\def\C{\mathbb{C}}

\def\F{\mathbb{F}}
\def\cF{\mathcal{F}}
\def\E{\mathcal{E}}
\def\M{\mathcal{M}}
\def\N{\mathbb{N}}
\def\D{\mathcal{D}}

\def\cR{\mathcal{R}}

\def\G{\mathcal{G}}

\def\BSut{\mathbf{BSut}}

\def\Vect{\mathbf{Vect}}

\def\spinc{\text{Spin}^c}

\def\G{\mathcal{G}}

\def\bP{\mathbb{P}}
\def\Cob{\mathbf{Cob}}
\def\Man{\mathbf{Man}}
\def\Id{\text{Id}}
\def\Diff{\text{Diff}}
\def\CC{\mathcal{C}}
\def\ub{\underline{b}}
\def\crit{\text{Crit}}

\def\ux{\underline{x}}
\def\MCG{\text{MCG}}
\def\End{\text{End}}
\def\ug{\underline{g}}
\def\uh{\underline{h}}
\def\gnf{$\text{GNF}^*$}
\def\HF{\mathit{HF}}

\def\JAlg{\mathbf{J \text{-} Alg}}

\begin{document}

\title{Defining and classifying TQFTs via surgery}%
\author{Andr\'as Juh\'asz}%
\address{Mathematical Institute, University of Oxford, Andrew Wiles Building,
Radcliffe Observatory Quarter, Woodstock Road, Oxford, OX2 6GG, UK}%
\email{juhasza@maths.ox.ac.uk}%


\begin{abstract}
We give a presentation of the $n$-dimensional oriented cobordism category $\Cob_n$
with generators corresponding to diffeomorphisms and surgeries along framed spheres,
and a complete set of relations.
Hence, given a functor $F$ from the category of smooth oriented manifolds and diffeomorphisms
to an arbitrary category $C$, and morphisms induced by surgeries along framed spheres,
we obtain a necessary and sufficient set of relations these have to satisfy to extend to
a functor from $\Cob_n$ to $C$.
If $C$ is symmetric and monoidal, then we also characterize when the extension is a TQFT.

This framework is well-suited to defining natural cobordism maps in Heegaard Floer homology.
It also allows us to give a short proof of the classical correspondence between (1+1)-dimensional
TQFTs and commutative Frobenius algebras.
Finally, we use it to classify (2+1)-dimensional TQFTs in terms of J-algebras,
a new algebraic structure that consists of a split graded involutive
nearly Frobenius algebra endowed with a certain mapping class group representation.
This solves a long-standing open problem. As a corollary, we obtain a structure theorem for
(2+1)-dimensional TQFTs that assign a vector space of the same dimension to every connected surface.
We also note that there are $2^{2^\omega}$ nonequivalent lax monoidal TQFTs over~$\C$ that do not extend to
(1+1+1)-dimensional ones.
\end{abstract}

\maketitle
\section{Introduction}
Let $\Man_n$ be the category whose objects are closed oriented $n$-manifolds and whose morphisms are
orientation preserving diffeomorphisms, and
let $\Cob_n$ be the category of closed oriented $n$-manifolds and equivalence classes of oriented cobordisms.
Furthermore, $\Cob_n'$ is the subcategory of~$\Cob_n$ that does not contain the empty $n$-manifold,
and such that each component of every cobordism has a non-empty incoming and outgoing end.
We denote by $\Cob_n^0$ the full subcategory of~$\Cob_n'$ consisting of connected objects (and hence
connected cobordisms).
Finally, $\BSut'$ is the category of balanced sutured manifolds and special
cobordisms that are trivial along the boundary; see \cite[Definition~5.1]{cobordism}.
We denote by~$\Vect$ the category of finite-dimensional vector spaces over some field~$\F$.

In physics, topological quantum field theories (in short, TQFTs) were introduced by Witten~\cite{Witten}.
Inspired by Segal's axioms proposed for conformal field theories~\cite{conformal}, they were first axiomatized by Atiyah~\cite{Atiyah}.
In the more recent terminology of Blanchet and Turaev~\cite{Blanchet},
an $(n+1)$-dimensional TQFT is a symmetric monoidal functor from the category~$\Cob_n$
to~$\Vect$; see Definition~\ref{def:tqft}.
More generally, the target category could be any symmetric monoidal category.
For the necessary category theoretical background, we refer the reader to the books
of Mac Lane~\cite{MacLane} and Kock~\cite{Kock}.
Throughout this paper, \emph{all manifolds are smooth and oriented and all diffeomorphisms are orientation preserving},
unless otherwise stated, though the methods easily generalize to unoriented manifolds.

It is a classical result that (1+1)-dimensional TQFTs correspond to commutative Frobenius algebras.
This statement dates back to the birth of the subject, but completely rigorous proofs are more recent;
see the book of Kock~\cite{Kock} that also discusses the history of this problem. Fully extended $(n+1)$-dimensional TQFTs constitute a
constrained subclass of $(n+1)$-dimensional TQFTs,
that assign invariants to all oriented manifolds with corners up to dimension $n+1$.
These were completely classified by Lurie~\cite{Lurie} via proving the ``cobordism hypothesis''
conjectured by Baez and Dolan. Based on the work of Reshetikhin and Turaev~\cite{TuTQFT},
Bartlett, Douglas, Schommer-Pries, and Vicary~\cite{BDSV, BDSV2} classified $3$-dimensional oriented
TQFTs extended down to $1$-manifolds, which are called (1+1+1)-dimensional or 1-2-3 TQFTs,
in terms of anomaly free modular tensor categories. This is a restricted subclass of all
lax monoidal (2+1)-dimensional TQFTs according to the following observation that we will
prove at the end of Section~\ref{sec:applications}.
(Recall that a TQFT~$F \colon \Cob_2 \to \Vect$ is lax monoidal if the comparison morphisms
$\Phi_{A,B} \colon F(A) \otimes F(B) \to F(A \sqcup B)$ are not necessarily invertible
for surfaces~$A$ and~$B$.)

\begin{prop} \label{prop:nonextending}
Over $\C$, there exist $2^{2^\omega}$ pairwise non-equivalent (2+1)-di\-men\-sional oriented lax monoidal TQFTs
that do not extend to (1+1+1)-dimensional TQFTs.
\end{prop}

Our first main result is a presentation of the $n$-dimensional oriented cobordism category
in terms of generators corresponding to diffeomorphisms and surgeries along framed spheres,
and a complete set of relations. We state the necessary definitions first.

\begin{defn} \label{def:sphere}
Let $M$ be an oriented $n$-manifold. For $k \in \{0,\dots,n\}$,
a \emph{framed $k$-sphere} in~$M$ is an orientation reversing embedding
$\SS \colon S^k \times D^{n-k} \hookrightarrow M$.
Then we can perform surgery on~$M$ along~$\SS$
by removing the interior of the image of~$\SS$
and gluing in~$D^{k+1} \times S^{n-k-1}$ via $\SS|_{S^k \times S^{n-k-1}}$;
after smoothing the corners we obtain the surgered manifold $M(\SS)$.
We consider two additional types of framed spheres, namely $\SS = 0$ and $\SS = \emptyset$.
For $\SS = 0$, which we think of as the
framed attaching sphere of a 0-handle, we let $M(0) = M \sqcup S^n$.
For $\SS = \emptyset$, we let $M(\emptyset) = M$. We write
\[
W(\SS) = (M \times I) \cup_\SS (D^{k+1} \times D^{n-k})
\]
for the \emph{trace of the surgery},
where $W(0) = (M \times I) \sqcup D^{n+1}$ and $W(\emptyset) = M \times I$.
Then $W(\SS)$ is a cobordism from $M$ to $M(\SS)$.
We denote by $a(\SS) = \SS(S^k \times \{0\}) \subset M$ the \emph{attaching sphere}
and by $b(\SS) = \{0\} \times S^{n-k-1} \subset M(\SS)$ the \emph{belt sphere} of the handle
attached along~$\SS$.


If $\SS \colon S^k \times D^{n-k} \hookrightarrow M$
is a framed $k$-sphere for $k < n$, then let~$\ol{\SS}$ be the framed sphere defined by
\[
\ol{\SS}(\underline{x},\underline{y}) = \SS \left(r_{k+1}(\underline{x}),r_{n-k}(\underline{y}) \right),
\]
where $\underline{x} \in S^k \subset \R^{k+1}$, $\underline{y} \in D^{n-k} \subset \R^{n-k}$,
and
\[
r_m(x_1,x_2,\dots,x_m) = (-x_1,x_2,\dots,x_m).
\]
\end{defn}

\begin{defn}
Let $\G_n$ be the directed graph obtained from
the category $\Man_n$ by adding an edge $e_{M,\SS}$ from $M$ to $M(\SS)$ for every pair $(M,\SS)$, where
$M$ is an oriented $n$-manifold and $\SS$ is a framed sphere inside~$M$. For clarity, we will sometimes
write~$e_d$ for the edge from~$M$ to~$N$ corresponding to a diffeomorphism $d \colon M \to N$.
Then $\Man_n$ is a subgraph of~$\G_n$. We denote by $\cF(\G_n)$ the free category generated by~$\G_n$.

Let $\G_n'$ be the subgraph of $\G_n$ obtained by removing the empty $n$-manifold,
and edges $e_{M,\SS}$ such that $\SS = 0$ or a framed  $n$-sphere.
Furthermore, $\G_n^0$ is the full subgraph of $\G_n'$ spanned by connected objects.
Finally, the vertices of $\G^s$ are balanced sutured manifolds, and the edges are diffeomorphisms
and surgeries along framed 0-, 1-, and 2-spheres in the interior of a balanced sutured manifold.
\end{defn}

\begin{defn} \label{def:relations}
We now define a set of relations~$\cR$ in $\cF(\G_n)$; these can be though of as 2-cells attached to~$\G_n$.
If $w$ and $w'$ are words consisting of composing arrows, then we write $w \sim w'$ if $w(w')^{-1} \in \cR$.
\begin{enumerate}
\item \label{it:isot}
Firstly, $e_{d \circ d'} \sim e_d \circ e_{d'}$ for diffeomorphisms~$d$ and~$d'$ that compose.
We have $e_{M,\emptyset} \sim \Id_M$, and if $d \in \Diff_0(M)$, then $e_d \sim \Id_M$.
\item \label{it:d-F}
Given an orientation preserving diffeomorphism $d \colon M \to M'$ between $n$-manifolds and a framed sphere $\SS \subset M$,
let $\SS' = d \circ \SS$, and let $d^\SS \colon M(\SS) \to M'(\SS')$ be the induced diffeomorphism. Then the following
diagram is commutative:
\[
\xymatrix{
  M \ar[r]^-{e_{M,\SS}} \ar[d]^-{d} & M(\SS) \ar[d]^{d^\SS} \\
  M'  \ar[r]^-{e_{M',\SS'}} & M'(\SS').}
\]
\item \label{it:commut}
If $M$ is an oriented $n$-manifold and $\SS$ and $\SS'$ are \emph{disjoint} framed spheres in~$M$,
then $M(\SS)(\SS') = M(\SS')(\SS)$; we denote this manifold by $M(\SS,\SS')$.
Then the following diagram is commutative:
\[
\xymatrixcolsep{3pc}\xymatrix{
  M \ar[r]^-{e_{M,\SS}} \ar[d]^-{e_{M,\SS'}} & M(\SS) \ar[d]^{e_{M(\SS),\SS'}} \\
  M(\SS')  \ar[r]^-{e_{M(\SS'),\SS}} & M(\SS,\SS').}
\]
\item \label{it:birth}
If $\SS' \subset M(\SS)$ and the attaching sphere $a(\SS')$ intersects
the belt sphere $b(\SS)$ once transversely,
then there is a diffeomorphism $\varphi \colon M \to M(\SS)(\SS')$ (which is defined
in Definition~\ref{def:phi}; it is the identity on $M \cap M(\SS)(\SS')$
and is unique up to isotopy), for which
\[
e_{M(\SS),\SS'} \circ e_{M,\SS} \sim \varphi.
\]
\item \label{it:0-sphere}
Finally, $e_{M,\SS} \sim e_{M,\ol{\SS}}$.
\end{enumerate}
We can define a set of relations $\cR^s$ in $\cF(\G^s)$ analogously.
\end{defn}

Having defined the relations~$\cR$, we can take the quotient category $\cF(\G_n)/\cR$.
This is a symmetric monoidal category when equipped with the disjoint union operation.

\begin{defn}
Let $c \colon \G_n \to \Cob_n$ be the graph morphism that is the identity on the vertices,
assigns the cylindrical cobordism $c_d$ to a diffeomorphism $d$ as in Definition~\ref{def:diffcob},
and assigns the elementary cobordism $W(\SS)$
to the edge $e_{M,\SS}$. This extends to a
symmetric strict monoidal functor
$c \colon \cF(\G_n) \to \Cob_n$.
Similarly, we can define a symmetric monoidal functor $c^s \colon \cF(\G^s) \to \BSut'$.
\end{defn}

\begin{rem}
Note that this is not an embedding as, for example, $c_d = c_{d'}$ if
and only if $d$ and $d'$ are pseudo-isotopic diffeomorphisms; see \cite[Theorem~1.9]{Milnor}.
\end{rem}

In our first main result, we give a presentation of $\Cob_n$, where the generators
are diffeomorphisms and surgery morphisms, and the relations
are given in Definition~\ref{def:relations}.

\begin{thm} \label{thm:presentation}
The functor $c \colon \cF(\G_n) \to \Cob_n$ descends to a functor
\[
\cF(\G_n)/\cR \to \Cob_n
\]
that is an isomorphism of symmetric monoidal categories.

By slight abuse of notation, we will also denote the functor $\cF(\G_n)/\cR \to \Cob_n$ by~$c$.
Then $c$ restricted to $F(\G_n')/\cR$ is an isomorphism onto $\Cob_n'$
and $c$ restricted to $F(\G_n^0)/\cR$ is an isomorphism onto $\Cob_n^0$.
Finally, $c^s \colon \cF(\G^s) \to \BSut'$ descends to a functor $\cF(\G^s)/\cR^s \to \BSut'$
that is an isomorphism of symmetric monoidal categories.
\end{thm}

Gay, Wehrheim, and Woodward~\cite{GWW, WW} introduced the notion of Cerf decomposition to construct
TQFTs by assigning maps to elementary cobordisms. They showed that any two decompositions of a cobordism into
elementary pieces can be related by a short list of moves. An elementary cobordism is one that admits a
Morse function with at most one interior critical point. Every cobordism can be decomposed into
elementary cobordisms, and two decompositions can be related by critical point cancelations or creations,
critical point reversals, and gluing or splitting cylinders.
This relies on the classification of singularities appearing in generic 1-parameter families
of smooth functions based on Thom transversality, and is summarised in the work of Cerf~\cite[pp.~23--24]{Cerf}.

However, Cerf decompositions do not keep track of the framed attaching spheres of the handles in the elementary
cobordisms, which feature in the definition of cobordism maps in Heegaard Floer homology.
Furthermore, the moves are defined on the level of the cobordisms and refer to Morse functions,
unlike our relations in Definition~\ref{def:relations} for surgeries.
Note that the natural definition of Heegaard Floer homology requires taking into account the
embedding of the Heegaard surface into the 3-manifold, hence one has to be particularly careful with various identifications
when defining the cobordism maps; see Section~\ref{sec:HF}.

A \emph{parameterized Cerf decomposition}~$\CC$ of $W$ consists of a decomposition
\[
W = W_0 \cup_{M_1} W_1 \cup_{M_2} \dots \cup_{M_m} W_m
\]
into elementary cobordisms~$W_i$ from $M_i$ to $M_{i+1}$,
together with framed spheres~$\SS_i \subset M_i$ and diffeomorphisms~$d_i \colon M_i(\SS_i) \to M_{i+1}$
that extend to the traces of the surgeries for $i \in \{0,\dots,m\}$;
see Definition~\ref{def:param-Cerf} for more detail.

The surjectivity of~$c$ onto the morphisms of $\Cob_n$ means that every cobordism~$W$ from $M$ to~$M'$ has
a parameterized Cerf decomposition.
Indeed, as we can replace any path of diffeomorphisms with their composition, we can find a path
\[
M = M_0 \xrightarrow{e_{M_0,\SS_0}}  M_0(\SS_0) \xrightarrow{d_0} M_1
\xrightarrow{e_{M_1,\SS_1}} M_1(\SS_1) \xrightarrow{d_1} \dots \xrightarrow{d_m} M_m = M'
\]
in $\G_n$ such that
\[
W = c\left(\prod_{i = 0}^m (d_i \circ e_{M_i,\SS_i})\right).
\]
This is precisely a parameterized Cerf decomposition of~$W$.

A straightforward but very useful consequence of Theorem~\ref{thm:presentation}
is a simple and easily applicable framework in all dimensions
for constructing all functors (e.g., TQFTs) from the oriented cobordism category~$\Cob_n$ to an arbitrary
target category~$C$ via surgery. This framework is
well-suited to the study of Heegaard Floer homology; see Section~\ref{sec:HF} for more detail.
We give a set of necessary and sufficient conditions
for surgery morphisms to give rise to cobordism morphisms independent of the surgery description of the cobordism.
The big advantage of considering surgeries as opposed to handle attachments is that, for an $(n+1)$-dimensional TQFT,
it suffices to work with $n$-manifolds and surgeries on these, without having to consider the
$(n+1)$-dimensional cobordisms themselves. To illustrate the power of this approach, we will classify (2+1)-dimensional TQFTs in terms
of a new algebraic structure called J-algebras. According to Segal~\cite{Segal}, the classification problem
for TQFTs is one that has been around since the inception of the subject, and so has been the aim to construct TQFTs via surgery.

\begin{thm} \label{thm:TQFT}
Let $C$ be a category.
Suppose that we are given a functor
\[
F \colon \Man_n \to C,
\]
and for every oriented $n$-manifold~$M$ and framed sphere~$\SS \subset M$,
a morphism $F_{M,\SS} \colon F(M) \to F(M(\SS))$ that satisfy relations~\eqref{it:isot}--\eqref{it:0-sphere}
(these are spelled out explicitly in Section~\ref{sec:construction}).
For a parameterized Cerf decomposition~$\CC$ of an oriented cobordism~$W$, let
\begin{equation} \label{eqn:comparison}
F(W,\CC) = \prod_{i=0}^{m} \left( F(d_i) \circ F_{M_i,\SS_i} \right) \colon F(M) \to F(M').
\end{equation}
Then $F(W,\CC)$ is independent of the choice of~$\CC$; we denote it by~$F(W)$.
Furthermore, $F \colon \Cob_n \to C$ is a functor that satisfies $F(d) = F(c_d)$
(see Definition~\ref{def:diffcob}) and $F(W(\SS)) = F_{M,\SS}$.

In the opposite direction, every functor $F \colon \Cob_n \to C$ arises in this way.
More precisely, if we let $F_{M,\SS} = F(W(\SS))$ and $F(d) = F(c_d)$,
then these morphisms satisfy relations~\eqref{it:isot}--\eqref{it:0-sphere},
and for any oriented cobordism~$W$, the morphism
$F(W)$ is given by equation~\eqref{eqn:comparison}.

Now suppose that $(C,\otimes,I_C)$ is a symmetric monoidal category.
Then the functor~$F$ is a TQFT if and only if $F \colon \Man_n \to C$ is symmetric and monoidal;
furthermore, given $n$-manifolds~$M$ and~$N$, and a framed sphere~$\SS$ in~$M$, the diagram
\begin{equation} \label{eqn:monoidal}
\xymatrixcolsep{3pc}\xymatrix{
  F(M) \otimes F(N) \ar[r]^-{\Phi_{M,N}} \ar[d]_-{F_{M,\SS} \otimes \text{Id}_{F(N)}} & F(M \sqcup N) \ar[d]^{F_{M \sqcup N, \SS}} \\
  F(M(\SS)) \otimes F(N) \ar[r]^-{\Phi_{M(\SS),N}} & F(M(\SS) \sqcup N).
  }
\end{equation}
is commutative, where $\Phi_{A,B} \colon F(A) \otimes F(B) \to F(A \sqcup B)$ are the comparison morphisms for~$F$.

An analogous result holds for $\Cob_n'$, and we can avoid $\SS = 0$ and framed $n$-spheres.
In the case of $\Cob_n^0$ for $n \ge 2$, we need to avoid $\SS = 0$ and $n$-spheres,
together with separating $(n-1)$-spheres. Finally, for~$\BSut'$, we have a similar result,
and we can avoid $\SS = 0$ and framed $3$-spheres.
\end{thm}

\begin{rem}
To illustrate why working with Cerf decompositions without the parameterization is insufficient
to define the cobordism morphism~$F(W)$, consider the simplest possible case when $W$ itself
is diffeomorphic to~$M \times I$. Then this is a Cerf decomposition with a single component.
Given a diffeomorphism $D \colon M \times I \to W$, let $d_t = D|_{M \times \{t\}}$; then
it is natural to define $F(W)$ as $F(d_1 \circ d_0^{-1})$. However, $D$ is not unique,
and for different choices we only know that the corresponding $d_1 \circ d_0^{-1}$ are
pseudo-isotopic, not necessarily isotopic, and hence a priori might induce different homomorphisms
via~$F$. To avoid this issue, we identify each component~$W_i$ of the Cerf decomposition with a concrete
handle cobordism~$W(\SS_i)$, and once we know this induces a TQFT, we obtain as a corollary
that pseudo-isotopic diffeomorphisms induce the same morphism.

When~$W$ is cylindrical, one might have to pass through a sequence of moves between Cerf decompositions to get from
one parametrization of~$W$ as a product to another. For example, by Kwasik and Schultz~\cite[Corollary]{pseudo},
if $M$ is the connected sum of two metacyclic prism 3-manifolds, then it admits
an automorphism~$d$ that is pseudo-isotopic but not isotopic to the identity.
Hence there is a diffeomorphism $D \colon M \times I \to M \times I$ such that
$D(x,0) = x$ and $D(x,1) = (d(x),1)$ for every $x \in M$. So, if $W$ is the identity
cobordism from $M$ to $M$, then $W = W_0 = M \times I$ with $\SS_0  = \emptyset$
and $d_0 = d$ is a different parameterized Cerf decomposition than for $d_0 = \Id_M$.
The first decomposition arises from the Morse function $f(x,t) = p_I \circ D^{-1}(x,t)$
on $M \times I$, where $p_I \colon M \times I \to I$ is the projection, while the second one
from the Morse function $p_I$. Then $f$ and $p_I$ cannot be connected with a family
of Morse functions with no critical points, as otherwise $d$ and $\Id_M$ would be isotopic.
\end{rem}

It might come as a surprise that handleslide invariance does not feature among the relations in Definition~\ref{def:relations}.
This is because the proof of Theorem~\ref{thm:presentation}
relies on proper and not self-indexing Morse functions, and a handleslide
can be replaced by moving one of the corresponding critical points to a higher level, isotoping
its framed attaching sphere, then moving it back to the same level. So handleslide invariance follows
from relations~\eqref{it:d-F} and~\eqref{it:commut}.

Segal~\cite[p.~34]{Segal} raised a related question on describing TQFTs via surgery
in terms of categories associated to products of spheres (along which the surgered disks are glued),
but this was never completed due to technical difficulties.
For a related result on 2-framed (2+1)-dimensional TQFTs, see the work of Sawin~\cite{Sawin},
where he outlines a Kirby calculus approach. Note that a Kirby calculus approach to constructing numerical invariants
of $3$-manifolds was suggested by Reshetikhin and Turaev in the introduction of~\cite{RT}.

\subsection{Applications to the classification of TQFTs}
Theorem~\ref{thm:TQFT} provides a powerful method for classifying TQFTs.
As our first application, we give a short, five pages long proof of
the classical theorem that the category of (1+1)-dimensional oriented TQFTs is
equivalent to the category of commutative Frobenius algebras. This also serves as a warmup for the
(2+1)-dimensional case: We obtain a complete classification of
(2+1)-dimensional oriented TQFTs with target category $\Vect$.
Specializing to this target allows us to carry out certain computations and
simplifications that are not possible in general symmetric monoidal categories.
As to be expected, the corresponding algebraic structure is more complicated than
in the (1+1)-dimensional case, but surprisingly only moderately, and can probably be simplified further,
which is the subject of future research.
For the definition of \emph{split graded involutive nearly Frobenius algebras} (or split \gnf-algebras in short),
see Definitions~\ref{def:grad-frob} and~\ref{def:splitting},
and for \emph{mapping class group representations} on these, see Definition~\ref{def:repr}.
A \emph{J-algebra} is a split \gnf-algebras endowed with a mapping class group representation.
These form a symmetric monoidal category that we denote by $\JAlg$.
Similar structures, called weight homogeneous tensor representations
were defined by Funar~\cite[p.~411]{Funar},
which correspond to certain lax monoidal (2+1)-dimensional TQFTs.
Our second main result is the following, which answers~\cite[Problem~8.1]{Ohtsuki}.

\begin{thm} \label{thm:2+1}
There is an equivalence between the
symmetric monoidal category of (2+1)-dimensional TQFTs and $\JAlg$.
\end{thm}


Let $\S_g$ denote a closed oriented surface of genus~$g$.
We use Theorem~\ref{thm:2+1} to show that, given a (2+1)-dimensional TQFT~$F$ over~$\C$
such that $\dim F(\S_g) < 2g$ for infinitely many~$g \in \N$,
the action of the mapping class group of~$\S_k$ on~$F(\S_k)$ is trivial for every~$k \in \N$.
This implies the following structure theorem, which we will prove in Proposition~\ref{prop:n}.

\begin{cor} \label{cor:n}
Suppose that $F$ is an oriented (2+1)-dimensional TQFT over~$\C$
such that $\dim F(\S) = n$ for every connected oriented surface~$\S$ for some constant~$n$.
Then $F$ is naturally isomorphic to the TQFT~$(F_1)^{\oplus n}$ given by
$F_1(\S) = \C$ for any surface~$\S$ and $F_1(W) = \Id_\C$ for any cobordism~$W$ (where we identify $\C^{\otimes k}$ with $\C$),
and we take the direct sum of TQFTs as defined by Durhuus and Jonsson~\cite{DJ}.
\end{cor}

\begin{ex}
To illustrate the non-triviality of this seemingly simple statement even for $n=1$,
consider Quinn's TQFT~$Q_\a$ for some $\a \in \R$, restricted to cobordisms of surfaces~\cite{Quinn}.
This is defined as $Q_\a(\S) = \C$ for any surface~$\S$, and a cobordism $W$ from $\S_0$
to~$\S_1$ induces the map $Q_\a(W)(z) = e^{i\a \chi(W,\S_0)} z$ for any $z \in Q_\a(\S_0) = \C$.
According to Corollary~\ref{cor:n}, this is naturally isomorphic to the TQFT $F_1$.
Indeed, for a surface $\S$, consider the transformation $N_\a(\S) \colon F_1(\S) \to Q_\a(\S)$
given by $N_\a(\S)(z) = e^{i\a \chi(\S)/2}z$ for $z \in F_1(\S) = \C$. This is natural
since a cobordism $W$ from $\S_0$ to~$\S_1$ satisfies $\chi(W) = (\chi(\S_0) + \chi(\S_1))/2$,
and hence $\chi(W,\S_0) = (\chi(\S_1) - \chi(\S_0))/2$.
\end{ex}

\begin{ex}
Together with Bartlett, in a forthcoming paper, we will give a non-trivial example of a functor $F \colon \Man_2 \to \Vect_\C$
together with surgery maps, where a simple check of the relations of Theorem~\ref{thm:TQFT}
shows that this data gives rise to a (2+1)-dimensional TQFT. More concretely, let $C$ be a spherical
fusion category. For a surface $\S$, we define $F(\S)$ to be the $\C$-vector space generated by
string-nets over~$C$; these are isotopy classes of embedded $C$-labeled graphs modulo a local equivalence relation.
Given a framed sphere $\SS$ in $\S$, there is a representative of the string-net in its equivalence class disjoint from it,
and performing the surgery on $\S$ along $\SS$ naturally gives rise to a string-net on $\S(\SS)$.
\end{ex}

\subsection{Applications to Heegaard Floer homology} \label{sec:HF}

We use Theorem~\ref{thm:TQFT} to construct functorial cobordism maps induced on sutured Floer homology and link Floer homology,
and a splitting of these along $\spinc$ structures using
a $\spinc$ refinement of Theorem~\ref{thm:TQFT} combined with Kirby calculus~\cite{cobordism}.
Heegaard Floer homology will not feature in the rest of the present paper, but as it was a key motivation
for Theorem~\ref{thm:TQFT}, we discuss the relationship below. For further details, refer to~\cite{cobordism}.

Heegaard Floer homology, defined by Ozsv\'ath and Szab\'o~\cite{OSz, OSz8}, consists of 3-manifold
invariants $\HF^+$, $\HF^-$, $\HF^\infty$, and $\widehat{\HF}$, together with cobordism maps induced on each,
and they admit refinements along $\spinc$ structures.
Every flavor is a type of $(3+1)$-dimensional TQFT, with some caveats such as they are only defined for connected 3-manifolds
and for connected cobordisms between them, there is no unique way of composing $\spinc$ cobordisms, and
to obtain an interesting closed 4-manifold invariant (conjectured to coincide with
the Seiberg-Witten invariant), one has to mix the $+$, $-$, and $\infty$ flavors.
In particular, they are functors from $\Cob_3^0$ to the category of $\Z[U]$-modules.
Mrowka called such a theory a ``secondary TQFT,'' but no precise axioms for these exist to date.
Ozsv\'ath and Szab\'o~\cite{OSz10} constructed the cobordism maps in Heegaard Floer homology
via composing surgery maps, and to check this is independent of the surgery description of the cobordism,
they used Kirby calculus.

The author noticed that there was a gap in the functorial construction
of the Heegaard Floer invariants due to the lack of connection between the 3-manifold and the
Heegaard diagrams used in their definitions. Together with Dylan Thurston~\cite{naturality},
we fixed this by considering Heegaard diagrams embedded in the 3-manifold.
An unexpected consequence of this was that $\widehat{HF}$ depends on the choice of a basepoint;
see the work of Zemke~\cite{Zemke} for a precise formula describing this dependence.

In light of this, I revisited~\cite{cobordism} the construction of the cobordism maps and extended
it to sutured manifold and link cobordisms using Theorem~\ref{thm:TQFT}.
A key point is that one has to
keep track of identifications and what happens to the embedding of the Heegaard diagram
while performing the Kirby moves to make the proof of~\cite[Theorem~3.8]{OSz10}
completely rigorous. For example, see the discussion about diffeomorphisms induced by handleslides
on page~170 of the book of Gompf and Stipsicz~\cite{GS}.

To get the $\spinc$ refinement, Ozsv\'ath and Szab\'o ingeniously attach all 2-handles simultaneously
to circumvent the non-uniqueness of the composition of $\spinc$ cobordisms,
which makes the use of Kirby calculus necessary.
They essentially checked all the necessary invariance properties, modulo the above mentioned
naturality issues due to not keeping track of identifications, and the sufficiency of these properties
is only sketched in the proof of~\cite[Theorem~3.8]{OSz10}. As it turns out~\cite{cobordism, Zemke},
the cobordism maps on $\widehat{\HF}$ also depend on an arc connecting the basepoints,
justifying the extra careful approach of this work.

\subsection*{Organization}

In Section~\ref{sec:TQFTs}, we review cobordism categories and TQFTs.
We define parameterized Cerf decompositions and Morse data in Section~\ref{sec:decomp}.
Lemmas~\ref{lem:lift} and~\ref{lem:unique} imply there is an essentially unique correspondence
between the two. We define a set of moves on Morse data in Section~\ref{sec:moves} that
arise from bifurcations in generic 1-parameter families.
Furthermore, we translate these to moves on Cerf decompositions, and show in Theorem~\ref{thm:Cerf}
that any two Morse data on a cobordism can be connected by a sequence of such moves.
We prove Theorems~\ref{thm:presentation} and~\ref{thm:TQFT} using the machinery of parameterized
Cerf decompositions in Section~\ref{sec:construction}.

In Section~\ref{sec:1+1}, as a warmup,
we reprove the classification of (1+1)-dimensional TQFTs using Theorem~\ref{thm:TQFT}.
We explain how to assign a J-algebra to a (2+1)-dimensional TQFT in Section~\ref{sec:assignment}.
We define split \gnf-algebras in Section~\ref{sec:gnf} and mapping class group representations
on these in Section~\ref{sec:MCG}, and study their algebraic properties.
We prove Theorem~\ref{thm:2+1} in Section~\ref{sec:proof2+1}, and present some examples and applications
in Section~\ref{sec:applications}.

\subsection*{Acknowledgement}

I would like to thank Bruce Bartlett, Andr\'e Henriques, Oscar Randal-Williams,
Graeme Segal, Peter Teichner, and Ulrike Tillmann for helpful discussions,
and the anonymous referees for their constructive suggestions.

The author was supported by a Royal Society Research Fellowship.
This project has received funding from the European Research Council (ERC) under the European
Union's Horizon 2020 research and innovation programme (grant agreement No 674978).
I would also like to thank the Isaac Newton Institute for its hospitality.

\section{Parameterized Cerf decompositions} \label{sec:Cerf}

\subsection{Cobordism categories and TQFTs} \label{sec:TQFTs}

When talking about cobordism categories, it is important to keep the following
definition in mind, see Milnor~\cite[Definition~1.5]{Milnor}.

\begin{defn}
A \emph{cobordism} from~$M_0^n$ to $M_1^n$ is a 5-tuple $(W;V_0,V_1;h_0,h_1)$, where~$W$
is a compact $(n+1)$-manifold such that $\partial W$ is the disjoint union of~$V_0$
and~$V_1$, and $h_i \colon V_i \to M_i$ are diffeomorphisms for $i \in \{0,1\}$.

If~$M_0$ and~$M_1$ are oriented, we require that~$W$ be oriented as well, such that if~$V_0$ and~$V_1$
are given the boundary orientation, then~$h_0$ is orientation reversing, while~$h_1$ is orientation
preserving.
\end{defn}

Given cobordisms from~$M_0$ to~$M_1$ and~$M_1$ to~$M_2$, we can glue them together, but the smooth
structure on the result
is only well-defined up to diffeomorphism fixing the boundaries. Hence, to be able to
define the composition of cobordisms, we consider the following equivalence relation.

\begin{defn}
The cobordisms $(W;V_0,V_1;h_0,h_1)$ and $(W';V_0',V_1';h_0',h_1')$ from $M_0$ to~$M_1$ are
\emph{equivalent} if there is a diffeomorphism $g \colon W \to W'$ such that $g(V_i) = V_i'$
and $h_i' \circ g|_{V_i} = h_i$ for $i \in \{0,1\}$.
\end{defn}

\begin{defn} \label{def:diffcob}
We can assign a cobordism to any diffeomorphism as follows.
Suppose that $h \colon M \to M'$ is a diffeomorphism of $n$-manifolds. Then let~$c_h$
be the equivalence class of the tuple
\[
(M \times I; M \times \{0\}, M \times \{1\}; p_0, h_1),
\]
where~$p_0(x,0) = x$ and~$h_1(x,1) = h(x)$ for every $x \in M$.
\end{defn}

Recall that two diffeomorphisms $h$, $h' \colon M \to M'$ are \emph{pseudo-isotopic} if
there is a diffeomorphism $g \colon M \times I \to M' \times I$ such that $g(x,i) = (h_i(x),i)$
for~$i \in \{0,1\}$ and~$x \in M$. Note that~$g$ does not have to preserve level sets.
Then $c_{h_0} = c_{h_1}$ if and only if~$h_0$ and~$h_1$ are pseudo-isotopic; see \cite[Theorem~1.9]{Milnor}.
Furthermore, $c_{h'} \circ c_h = c_{h' \circ h}$, where we write the composition of
cobordism from right-to-left, as opposed to Milnor~\cite[Theorem~1.6]{Milnor}.
The following is based on~\cite[Definition~1.5]{Milnor}.

\begin{defn}
Let $\Cob_n$ be the category whose objects are closed oriented $n$-manifolds, and whose
morphisms are equivalence classes of oriented cobordisms. For an $n$-manifold~$M$, the identity
morphism~$i_M := c_{\Id_M}$.
\end{defn}

The description of the identity morphism highlights the role of the parameterizations~$h_i$,
as only using triads~$(W;V_0,V_1)$, we would not have any morphisms from~$M$ to itself.

\begin{defn} \label{def:tqft}
Let $\Vect$ be the category of vector spaces and linear maps over some field~$\F$.
An \emph{$(n+1)$-dimensional topological quantum field theory} is a functor
\[
F \colon \Cob_n \to \Vect
\]
such that for any two closed $n$-manifolds~$M$ and~$M'$, there are \emph{natural} isomorphisms
$\Phi_{M,M'} \colon F(M) \otimes F(M') \to F(M \sqcup M')$ and
$\Phi \colon \F \to F(\emptyset)$, which are part of the data,
that make the following diagrams commutative:
\[
\xymatrixcolsep{9pc}\xymatrix{
(F(M) \otimes F(N)) \otimes F(P) \ar[r]^-{\Phi_{M \sqcup N, P} \circ \left(\Phi_{M,N} \otimes \Id_{F(P)}\right)}
\ar[d] & F((M \sqcup N) \sqcup P) \ar[d] \\
F(M) \otimes (F(N) \otimes F(P)) \ar[r]^-{\Phi_{M, N \sqcup P} \circ \left(\Id_{F(M)} \otimes \Phi_{N,P}\right)} &
F(M \sqcup (N \sqcup P)),
}
\]
\[
\xymatrixcolsep{7pc}\xymatrix{
F(M) \otimes \F \ar[r]^-{\Phi_{M \sqcup \emptyset} \circ \left(\Id_{F(M)} \otimes \Phi \right)} \ar[d] &
F(M \sqcup \emptyset) \ar[d] \\
F(M) \ar[r]^-{=} & F(M).}
\]
In other words, $F$ \emph{preserves the monoidal structure} on $\Cob_n$ given by the disjoint union
and on $\Vect$ given by the tensor product.
Furthermore, the functor $F$ is \emph{symmetric} in the sense that
\[
\xymatrixcolsep{3pc}\xymatrix{
  F(M) \otimes F(M') \ar[r]^-{\Phi_{M,M'}} \ar[d]^-{F(c_s)} & F(M \sqcup M') \ar[d]^{r} \\
  F(M') \otimes F(M) \ar[r]^-{\Phi_{M',M}} & F(M' \sqcup M),}
\]
where $s \colon M \sqcup M' \to M' \sqcup M$ is the diffeomorphism swapping the two
factors, and $r(x \otimes y) = y \otimes x$.

More generally, $\Vect$ could be replaced by any symmetric monoidal category.
Similarly, a TQFT on the category of \emph{connected} $n$-manifolds is a functor
\[
F \colon \Cob_n^0 \to \Vect,
\]
but in this case we drop the conditions on disjoint unions.
\end{defn}

Given an orientation preserving diffeomorphism~$h$, we denote the map~$F(c_h)$ by~$h_*$.
We shall see in Lemma~\ref{lem:diffeo-maps} that if~$F$ arises from a functor
$F \colon \Man_n \to \Vect$ and surgery maps~$F_{M,\SS}$ as in Theorem~\ref{thm:TQFT}, then $h_* = F(h)$.
If~$h$ and~$h'$ are pseudo-isotopic, then $c_h = c_{h'}$, hence $h_* = h'_*$.

The cobordism maps in a TQFT~$F$ satisfy the following naturality property.

\begin{lem} \label{lem:funct}
Let $\W = (W;V_0,V_1;h_0,h_1)$ be an oriented cobordism from~$M_0$ to~$M_1$,
and let $\W' = (W';V_0',V_1';h_0',h_1')$ be an oriented cobordisms from~$M_0'$ to $M_1'$.
If $d \colon W \to W'$ is an orientation preserving diffeomorphism such that $d(V_i) = V_i'$ for $i \in \{0,1\}$,
then we write
\[
d|_{M_i} := h_i' \circ d|_{V_i} \circ h_i^{-1} \colon M_i \to M_i'.
\]
If $F$ is a TQFT, then the following diagram is commutative:
\[
\xymatrix{
  F(M_0) \ar[r]^-{F(c)} \ar[d]^-{(d|_{M_0})_*} & F(M_1) \ar[d]^{(d|_{M_1})_*} \\
  F(M_0')  \ar[r]^-{F(c')} & F(M_1'),}
\]
where~$c$ is the equivalence class of~$\W$ and~$c'$ is the equivalence class of~$\W'$.
\end{lem}

\begin{proof}
As $(d|_{M_i})_* = F(c_{d|_{M_i}})$, this follows from the functoriality of~$F$,
once we observe that the cobordisms $c' \circ c_{d|_{M_0}}$ and
$c_{d|_{M_1}} \circ c$ are equivalent via~$d$.
\end{proof}

\subsection{Parameterized Cerf decompositions and Morse data} \label{sec:decomp}

To simplify the notation for cobordisms, from now on, we will suppress the diffeomorphisms~$h_0$ and~$h_1$
and identify~$V_i$ and~$M_i$. So an oriented cobordism from~$M_0$
to~$M_1$ is viewed as a compact $(n+1)$-manifold $W$ with~$\partial W = -M_0 \cup M_1$.
With this convention, two cobordisms~$W$ and~$W'$ from~$M_0$ to~$M_1$ are equivalent if there is a diffeomorphism
$d \colon W \to W'$ that fixes the boundary pointwise.
We say that $f \colon W \to [a,b]$ is a Morse function if $f^{-1}(a) = M_0$, $f^{-1}(b) = M_1$,
and~$f$ has only non-degenerate critical points, all lying in the interior of~$W$.

Recall from Definition~\ref{def:sphere} that,
given an oriented $n$-manifold~$M$, a framed $k$-sphere $\SS \subset M$ is an orientation reversing embedding of~$S^k \times D^{n-k}$ into~$M$.
We write~$W(\SS)$ for the manifold obtained by attaching the handle~$D^{k+1} \times D^{n-k}$ to~$M \times I$ along~$\SS \times \{1\}$;
this is a cobordism from~$M$ to the manifold~$M(\SS)$ obtained by surgery on~$M$ along~$\SS$.
We now recall and extend~\cite[Definition~3.10]{Milnor}.

\begin{defn}
A cobordism~$W$ from~$M_0$ to~$M_1$ is \emph{elementary} if there is a Morse function $f \colon W \to [a,b]$
such that it has at most one critical point. A framed attaching sphere~$\SS$ for~$W$ is~$\emptyset$
if~$f$ has no critical points; otherwise, it is a framed sphere in~$M_0$ such that there is a diffeomorphism
$D \colon W(\SS) \to W$ that is the identity along~$M_0$ (where we identify $M_0$ with $M_0 \times \{0\}$).
\end{defn}

It is a classical result of Morse theory~\cite[Definition~3.9 and Theorem~3.13]{Milnor}
that every elementary cobordism admits a framed attaching sphere in
the above sense.

\begin{defn} \label{def:param-Cerf}
A \emph{parameterized Cerf decomposition} of an oriented cobordism~$W$ from~$M$ to~$M'$ consists of
\begin{itemize}
\item a Cerf decomposition
\[
W = W_0 \cup_{M_1} W_1 \cup_{M_2} \dots \cup_{M_m} W_m
\]
in the sense of Gay, Wehrheim, and Woodward~\cite{GWW}; i.e., each~$W_i$ is an elementary cobordism from~$M_i$ to~$M_{i+1}$,
where $M_0 = M$ and $M_{m+1} = M'$,
\item a framed attaching sphere $\SS_i \subset M_i$ for~$W_i$ of dimension~$k_i$ for $i \in \{0,\dots,m\}$,
\item an orientation preserving diffeomorphism $d_i \colon M_i(\SS_i) \to M_{i+1}$, well-defined up to isotopy, such that there exists a
diffeomorphism $D_i \colon W(\SS_i) \to W_i$ with $D_i|_{M_i \times \{0\}} = p_0$
and~$D_i|_{M_i(\SS_i)} = d_i$, where $p_0(x,0) = x$.
\end{itemize}
\end{defn}

\begin{rem}
The existence of the diffeomorphism~$D_i$ ensures that the cobordism
\[
(W(\SS_i); M_i \times \{0\}, M_i(\SS_i); p_0, d_i)
\]
is equivalent to $(W_i; M_i, M_{i+1}; \Id_{M_i}, \Id_{M_{i+1}})$. So we are replacing
each elementary component in the Cerf decomposition of~$W$ by an equivalent handle cobordism.
In particular, the composition of these handle cobordisms is equivalent to~$(W; M, M'; \Id_M, \Id_{M'})$.
\end{rem}

The following definition is based on~\cite[Definition~3.1]{Milnor}.


\begin{defn} \label{defn:grad-like}
Let $f$ be a Morse function on the oriented cobordism~$W$. We say that the vector field~$v$ on~$W$
is gradient-like for~$f$ if $v_p(f) > 0$ for every $p \in W \setminus \crit(f)$,
and for every point $p \in \crit(f)$, there exists a local positively oriented coordinate system
$(x_1, \dots, x_{n+1})$ centered at~$p$ in which
\begin{equation} \label{eqn:normal}
f = f(p) - x_1^2 - \dots -x_k^2 + x_{k+1}^2 + \dots + x_{n+1}^2,
\end{equation}
and where $v$ is the Euclidean gradient; i.e.,
\begin{equation} \label{eqn:v}
v = 2\left(-x_1 \frac{\partial}{\partial x_1} - \dots - x_k \frac{\partial}{\partial x_k} +
x_{k+1} \frac{\partial}{\partial x_{k+1}} + \dots + x_{n+1} \frac{\partial}{\partial x_{n+1}}\right).
\end{equation}
\end{defn}

The space of positive coordinate systems at a Morse critical point in which~$f$ is of the normal form~\eqref{eqn:normal}
is homotopy equivalent to $SO(k,n+1-k)$, and hence is connected for $k \in \{0, n+1\}$, and has two components
otherwise; see Cerf~\cite[p.168]{Cerf}. However, the space of gradient vector fields~$v$ induced by such coordinate
systems is connected for every~$k$. Indeed, if~$k \not\in \{0,n+1\}$ and $(x_1,\dots,x_{n+1})$ is a positive
coordinate system in which~$f$ is of the form~\eqref{eqn:normal}, then
\[
(-x_1,x_2,\dots,x_n,-x_{n+1})
\]
is also a positive coordinate system as in~\eqref{eqn:normal},
but which lies in the opposite component since it reverses the orientation of both the positive and negative
definite subspaces. In both coordinate systems~$v$ is of the same form.

 \begin{defn}
A \emph{Morse datum} (cf.~\cite[Definition~2.1]{GWW}) for the cobordism~$W$ is a triple~$(f,\ub,v)$,
where
\begin{itemize}
\item $\ub = (b_0, \dots, b_{m+1}) \in \R^{m+2}$ is an ordered tuple; i.e., $b_0 < b_1 < \dots < b_{m+1}$,
\item $f \colon W \to [b_0, b_{m+1}]$ is a Morse function such that
each~$b_i$ is a regular value of~$f$, and~$f$ has at most one critical value in each interval~$(b_{i-1},b_i)$, and
\item $v$ is a gradient-like vector field for~$f$.
\end{itemize}
\end{defn}

We now explain how to construct a parameterized Cerf decomposition from a Morse datum.

\begin{defn} \label{def:CM}
Suppose that~$W$ is an elementary cobordism from~$M$ to~$M'$, together with a Morse function~$f$
and a gradient-like vector field~$v$.
If~$f$ has no critical points, then one obtains a diffeomorphism $d_v \colon M \to M'$ by flowing
along $w = v / v(f)$.
When~$f$ has one critical point~$p$ of index~$k$,
then we obtain a framed sphere $\SS \colon S^{k-1} \times D^{n-k+1} \to M$,
and a diffeomorphism $d_v \colon M(\SS) \to M'$, well-defined up to isotopy, as follows.
(Note that Milnor~\cite[Definition~3.9]{Milnor} calls~$\SS$ the \emph{characteristic embedding}.
We review his construction to be able to define the map~$d_v$.)

Let~$W^s(p)$ be the stable manifold of~$p$.
We define the attaching sphere~$a(\SS)$ to be $W^s(p) \cap M$, with the following framing.
As in Milnor~\cite[p.~16]{Milnor2}, choose a positive coordinate system
\[
(x_1,\dots,x_{n+1}) \colon U \to \R^{n+1}
\]
centered at~$p$ in which~$f$ is of the form~\eqref{eqn:normal}, and let~$\eps$ be so small that the image of
$(x_1,\dots,x_{n+1})$ contains a ball of radius $\sqrt{2\eps}$ centered at the origin.
Let~$c = f(p)$,  $x_- = (x_1,\dots,x_k)$, and $x_+ = (x_{k+1}, \dots, x_{n+1})$.
Define the cell~$e$ to be the subset of~$U$ where $|x_-|^2 \le \eps$ and $x_+ = 0$.
Furthermore, let $E$ be a regular neighborhood of~$e$ of width $\sqrt{\eps/2}$,
extending all the way to~$f^{-1}(c-\eps)$; i.e.,
\[
E = \{\, |x_+|^2 \le \eps/2 \,\} \cap \{\, c - |x_-|^2 + |x_+|^2 \ge c - \eps \,\}.
\]
This is diffeomorphic to the $k$-handle $D^k \times D^{n-k+1}$ via the map
\[
e(x_-,x_+) = \left(\frac{x_-}{\sqrt{|x_+|^2 + \eps}}, \sqrt{\frac{2}{\eps}} x_+ \right),
\]
which identifies $E \cap f^{-1}(c-\eps)$ with $S^{k-1} \times D^{n-k+1}$.
For $s \in S^{k-1} \times D^{n-k+1}$, flow from $e^{-1}(s) \in E \cap f^{-1}(c-\eps)$ along~$-w$ to~$M$
to obtain $\SS(s)$.

It is straightforward to check that~$v$ is transverse to~$\partial E \setminus f^{-1}(c-\eps)$.
The diffeomorphism~$d_v$ is defined by flowing from~$M \setminus \text{Im}(\SS)$ along~$w$
to~$f^{-1}(c-\eps) \setminus E$, and identifying the part $D^k \times S^{n-k}$ of $M(\SS)$
with $E \setminus f^{-1}(c-\eps)$ via~$e^{-1}$,
then flowing again along~$w$ to~$M'$ (as we are not flowing from
a level set, for different points, we need to flow for a different amount of time to reach~$M'$).
Note that $d_v|_{M \setminus \text{Im}(\SS)}$ is simply given by the flow of~$v$.
It is easy to see that $d_v$ extends to a diffeomorphism from $W(\SS)$ to $W$ that is the identity on~$M$;
see \cite[Theorem~3.13]{Milnor}.
\end{defn}

\begin{rem} \label{rem:dependence}
The above construction depends on the choice of~$\eps$ and local coordinate system as follows.
The attaching sphere~$a(\SS)$ is unique, and different choices give isotopic framings.
Furthermore, if $\SS_i$ and $d_{v,i} \colon M(\SS_i) \to M'$ for $i \in \{1,2\}$
arise from different coordinate systems and $\eps_i$,
then there is an isotopically unique diffeomorphism $\psi \colon M(\SS_1) \to M(\SS_2)$
such that $d_{v,1}$ is isotopic to $d_{v,2} \circ \psi$.
This ambiguity will not cause any problems in the proof of Theorem~\ref{thm:presentation}
due to relation~\eqref{it:isot} of Definition~\ref{def:relations}.

The framed sphere $\SS$ and the diffeomorphism~$d_v$ depend on~$v$ only up to isotopy,
since the space of gradient-like vector fields~$v$
compatible with a given Morse function~$f$ is connected. The only caveat is that when~$k \not\in \{0,n+1\}$, the space of
coordinate systems is homotopy equivalent to $SO(k,n+1-k)$, which has two components. The two components correspond
to non-isotopic framed spheres. If~$\SS$ is one, then $\ol{\SS}$ represents the other isotopy class;
cf.~relation~\eqref{it:0-sphere} in Definition~\ref{def:relations}.
\end{rem}

\begin{defn}
Let $W$ be an oriented cobordism from~$M$ to~$M'$.
We say that the Morse datum~$(f,\ub, v)$ \emph{induces}
the parameterized Cerf decomposition~$\CC$ of~$W$ if $M_i = f^{-1}(b_i)$ and~$W_i = f^{-1}([b_i,b_{i+1}])$.
Furthermore, for each elementary cobordism~$W_i$, the framed attaching sphere~$\SS_i$
and the diffeomorphism $d_i \colon M_i(\SS_i) \to M_{i+1}$
are obtained from $f|_{W_i}$ and $v|_{W_i}$ as in Definition~\ref{def:CM} for some choice of compatible
local coordinate systems and radii~$\eps_i$ at the critical points.
\end{defn}


%

Hence, the Morse datum~$(f,\ub,v)$ gives rise to a well-defined parameterized Cerf decomposition
that we denote by~$\CC(f,\ub,v)$, up to possibly replacing a framed sphere~$\SS$ with~$\ol{\SS}$,
and up to the ambiguity explained in Remark~\ref{rem:dependence}.
The following result, which is a slight extension of \cite[Theorem~3.12]{Milnor} to include
the parametrization, states that this assignment is surjective.

\begin{lem} \label{lem:lift}
Let~$\CC$ be a parameterized Cerf decomposition of the oriented cobordism~$W$.
Then there exists a Morse datum~$(f,\ub, v)$ inducing~$\CC$.
\end{lem}

\begin{proof}
By definition, each diffeomorphism $d_i \colon M_i(\SS_i) \to M_{i+1}$ extends to a diffeomorphism
$D_i \colon W(\SS_i) \to W_i$. We claim that there is a Morse function $f_i' \colon W(\SS_i) \to \R$
and a gradient-like vector field~$v_i'$ on $W(\SS_i)$ such that~$f_i'$ has a single critical point in the handle
if~$\SS_i \neq \emptyset$, and the diffeomorphism
$d_{v_i'}$ induced by $f_i'$ and $v_i'$ on $W(\SS_i)$ as in Definition~\ref{def:CM} is~$\Id_{M_i(\SS_i)}$.
If $\SS_i = \emptyset$, then we take~$f_i'$ to be the projection $p_2 \colon M_i \times I \to I$ and $v_i'$ to be~$\partial/\partial t$.

If $\SS_i \neq \emptyset$ is a framed $(k-1)$-sphere, then consider the functions
\[
s(x_1,\dots,x_{n+1}) = 1/2 - x_1^2 - \dots - x_k^2 + x_{k+1}^2 + \dots + x_{n+1}^2 \text{ and }
\]
\[
u(x_1,\dots,x_{n+1}) = \sqrt{(x_1^2 + \dots + x_k^2)(x_{k+1}^2 + \dots + x_{n+1}^2)}
\]
on~$\R^{n+1}$. Let
\[
H = \{\, \ux \in \R^{n+1} \colon 0 \le s(\ux) \le 1 \text{, } u(\ux) \le 1 \,\}.
\]
Consider $\SS_i \colon S^{k-1} \times D^{n-k+1} \to M_i$, then
\[
G = \left( \im(\SS_i) \times I \right) \cup \left( D^k \times D^{n-k+1} \right) \subset W(\SS_i)
\]
is diffeomorphic to~$H$ if we smooth the corners after attaching the handle.
We choose a diffeomorphism $\phi \colon G \to H$ such that it maps $\im(\SS_i) \times \{0\}$
to $H \cap \{s = 0\}$ and $D^k \times S^{n-k}$ to $H \cap \{s = 1\}$.
Furthermore, there is a small $\nu  \in \R_+$ such that for any $\ux \in H$
with $s(\ux) \in (0,1)$ and $u(\ux) \in [1-\nu, 1]$,
we have $\phi^{-1}(\ux) \in M_i \times \{s(\ux)\}$. For $y \in (M_i \times I) \setminus G$, we let $f_i'(y) = p_2(y)$,
where $p_2(x,t) = t$,
while for $y \in G$, let $f_i'(y) = s(\phi(y))$. This is a smooth function by construction.
The gradient-like vector field~$v_i'$ on $W(\SS_i)$ is defined on~$G$ by pulling back the Euclidean gradient of~$s$ on~$H$ via~$\phi$.
We extend this to $(M_i \times I) \setminus G$ via~$\partial/\partial t$.
It is now straightforward to check that the function~$f_i'$ and the gradient-like vector field~$v_i'$ induce the identity diffeomorphism
from~$M_i(\SS_i)$ to itself if we apply the construction in Definition~\ref{def:CM} with
the radius $\eps = 1$.

Let $a_i \colon I \to [b_{i-1},b_i]$ be the affine equivalence $a_i(t) = (1-t)b_{i-1} + t b_i$,
and we set $f_i := a_i \circ f_i' \circ D_i^{-1}$.
By~\cite[Lemma~2.6]{GWW}, we can modify the~$f_i$ by an ambient isotopy on a collar neighborhood of~$M_i$
such that they patch together to a Morse function~$f$.
If~$v_i = D_i^*(v_i')$, possibly modified on a collar of $M_i$ so that for different~$i$ they fit together to
a smooth vector field~$v$, then the induced diffeomorphism from~$M(\SS_i)$ to~$M_{i+1}$ will be isotopic to~$d_i$.
\end{proof}


\begin{lem} \label{lem:unique}
Let $\CC$ be a parameterized Cerf decomposition of the cobordism~$W$. Suppose that the Morse data $(f,\ub,v)$
and $(f',\ub',v')$ both induce~$\CC$, in the sense that, for given local coordinate systems
about the critical points and radii, the framings of the attaching spheres and the
diffeomorphisms~$d_i$ coincide. Then there exist orientation preserving diffeomorphisms $D \colon W \to W$
and $\phi \colon \R \to \R$ such that
\begin{enumerate}
\item $\ub' = \phi(\ub)$,
\item $f' = \phi \circ f \circ D^{-1}$,
\item $\nu \cdot v' = D_*(v)$ for some positive function $\nu \in C^\infty(W,\R_+)$, and
\item $D|_{M_i} = \Id_{M_i}$.
\end{enumerate}
\end{lem}

\begin{proof}
First, suppose that~$W$ is an elementary cobordism, $\ub = \ub'$, and $|\ub| = |\ub'| = 2$.
For an illustration, see Figure~\ref{fig:gradients}.
Let the critical points of~$f$ and~$f'$ be~$p$ and~$p'$ with values~$c$ and~$c'$, respectively.
Choose coordinate charts $\ux \colon U \to \R^{n+1}$ and $\ux' \colon U' \to \R^{n+1}$ about~$p$ and~$p'$,
respectively, such that their images coincide with the disk~$D(\underline{0},\sqrt{2\eps})$,
and in which~$f$ and~$f'$ have the normal form of equation~\eqref{eqn:normal},
while~$v$ and~$v'$ have the normal form~\eqref{eqn:v}. Furthermore, we write $K_p = W^s(p) \cup W^u(p)$
and $K_{p'} = W^s(p') \cup W^u(p')$, where the stable and unstable manifolds for~$p$ are always
with respect to~$v$, while for~$p'$ they are with respect to~$v'$.

Let $\phi_0 \colon [b_0,b_1] \to [b_0,b_1]$ be a diffeomorphism such that $\phi_0(b_i) = b_i$ for $i \in \{0,1\}$,
and such that $\phi_0(t) = c'- c + t$ for $t \in [c - 2\eps, c + 2\eps]$. Then~$v$ is also a gradient-like
vector field for $\phi_0 \circ f$; moreover, $\phi_0 \circ f(p) = f(p')$, and the Morse datum $(\phi_0 \circ f, \ub, v)$
induces the same parameterized Cerf decomposition~$\CC$.
Hence, we can assume that $f(p) = f(p') = c$.


Let $\g \colon Z \to W$ and
$\g' \colon Z' \to W$ for $Z$, $Z' \subset W \times \R$ be the flows of~$v$ and~$v'$, respectively.
For $x \in W$, the set $I_x := (\{x\} \times \R) \cap Z$ is a closed interval  $\{x\} \times [-\a(x),\omega(x)]$
when $x \not \in K_p$, a half-interval $\{x\} \times [-\a(x),\infty)$ when $x \in W^s(p)$,
and a half-interval $\{x\} \times (-\infty,\omega(x)]$ for $x \in W^u(p)$.
Using~$Z'$, we obtain the interval~$I_x'$ and the functions $\alpha'$ and $\omega'$ in an analogous way.

Let $D(p) = p'$. We define the diffeomorphism~$D$ on $W \setminus \{p\}$ as follows.
First, note that $W^s(p) \cap M = W^s(p') \cap M = a(\SS)$ and $W^u(p) \cap M' = W^u(p') \cap M' = d(b(\SS))$,
where $\SS \colon S^{k-1} \times D^{n-k+1} \to M$ is the framed sphere and $d \colon M(\SS) \to M'$
is the diffeomorphism in the Cerf decomposition $\CC$ induced by both $(f,\ub,v)$ and $(f',\ub',c')$.
If $x \in M \cup d(b(\SS))$ and $t \in I_x$, then
there is the unique parameter value $t'(x,t) \in I_x'$ for which
\[
f'(\g'(x,t'(x,t))) = f(\g(x,t)).
\]
Indeed, $f$ is monotonically increasing along the flow-line $\g(x,s)$ for $s \in I_x$,
and $f'$ is monotonically increasing along~$\g'(x,s')$ for $s' \in I_x'$.
Furthermore, $f(\g(x,\inf(I_x))) = f'(\g'(x,\inf(I_x')))$
and $f(\g(x,\sup(I_x))) = f'(\g'(x,\sup(I_x')))$ as
$b_0 = b_0'$, $b_1 = b_1'$, and $c = c'$.
If $x \in M \cup d(b(\SS))$ and $t \in I_x$, then let
\[
D(\g(x,t)) = \g'\left(x, t'(x,t) \right).
\]

It is clear that~$D$ restricts to a diffeomorphism
\[
W \setminus W^u(p) \to W \setminus W^u(p')
\]
that fixes $\partial W \setminus W^u(p) = \partial W \setminus W^u(p')$ pointwise.
Indeed, for $x \in M \setminus a(\SS)$, we have $\g(x,\omega(x)) = \g'(x, \omega'(x))$
since the Morse data $(f,\ub,v)$ and $(f',\ub',v')$ induce the same
diffeomorphism $d \colon M(\SS) \to M'$ in~$\CC$.

\begin{figure}
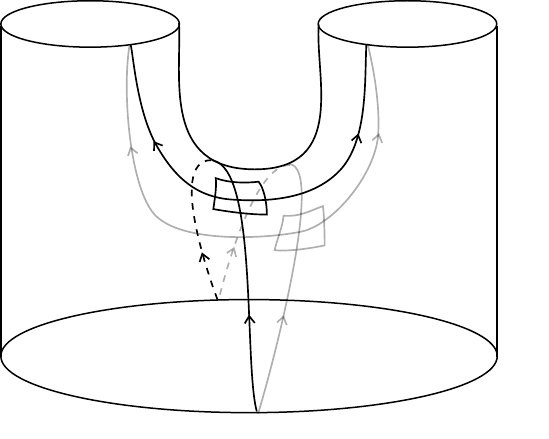
\caption{An elementary cobordism $W$ with two different Morse data that induce the same framed sphere
$\SS$ and diffeomorphism $d \colon M \to M(\SS)$.}
\label{fig:gradients}
\end{figure}

Let~$E$ be the subset of $\R^{n+1}$ constructed in Definition~\ref{def:CM};
it is diffeomorphic to the $k$-handle $D^k \times D^{n-k+1}$.
We denote by~$\partial_- E$ the part of~$\partial E$ corresponding to~$S^{k-1} \times D^{n-k+1}$,
and by $\partial_+ E$ the part corresponding to $D^k \times S^{n-k}$.
Let~$F$ be the smallest subset of~$W$ that contains~$\E = \ux^{-1}(E)$ and is saturated under the flow of~$v$,
and we define~$F'$ containing $\E' = (\ux')^{-1}(E)$ analogously. Note that $F$ is a regular neighborhood
of $K_p$ and $F'$ is a regular neighborhood of $K_{p'}$. Furthermore, let $\partial_\pm \E = \ux^{-1}(\partial_\pm E)$,
and $\partial_\pm \E' = (\ux')^{-1}(\partial_\pm E)$.

Since $(f,\ub,v)$ induces~$\CC$, by definition,
the flow of~$v$ from
\[
\E \cap f^{-1}(c - \eps) = \partial_- \E \approx S^{k-1} \times D^{n-k+1}
\]
gives~$\SS$. Similarly, the flow of~$v'$ from $\E' \cap (f')^{-1}(c - \eps) = \partial_- \E'$ gives~$\SS'$
as $(f',\ub',v')$ also induces~$\CC$. If~$H$ denotes the handle part of $M(\SS)$, which is diffeomorphic to $D^k \times S^{n-k}$,
then $d \colon M(\SS) \to M'$ restricts to a map $d|_H$ that gives a framing of $W^u(p) \cap M' = W^u(p') \cap M'$
that is given by either flowing from $\partial_+ \E$ along~$v$ to~$M'$, or from $\partial_+ \E'$ along~$v'$ to~$M'$.

We claim that
\begin{equation} \label{eqn:D}
D|_{\E} = (\ux')^{-1} \circ \ux \,\colon\, \E \to \E'.
\end{equation}
To see this, it suffices to show that, for any point $e \in \partial \E$,
we have
\begin{equation} \label{eqn:boundaryD}
\ux'(D(e)) = \ux(e) \in \partial E.
\end{equation}
Indeed, if $e \in \E \setminus W^u(p)$, then
there is a unique $t \in \R_{\le 0}$ for which $\g(e,t) \in \partial_- \E$; we write  $e_- = \g(e,t)$.
By definition, $D(e)$ is given by flowing back to~$M$ along~$v$, and then forward along~$v'$ until
the value of~$f'$ agrees with~$f(e)$. We obtain the same point by flowing back along~$v$ to $e_- \in \partial_- \E$,
then forward along~$v'$ from $D(e_-) = (\ux')^{-1} \circ \ux(e_-)$ until~$f'$ becomes~$f(e)$.
Since $(\ux')^{-1} \circ \ux$ takes~$v$ to~$v'$ and~$f$ to~$f'$ as they are in normal form in~$\ux$ and~$\ux'$, respectively,
we see that~$D(e) = (\ux')^{-1} \circ \ux(e)$. If $e \in W^u(p) \setminus \{p\}$, then
there is a unique~$t \in \R_{\ge 0}$ for which $\g(e,t) \in \partial_+ E$; let $e_+ = \g(e,t)$.
In this case, we get~$D(e)$ by flowing forward to~$M'$ along~$v$, then back along~$v'$ until the value of~$f'$
becomes~$f(e)$. We get the same point by flowing back from~$D(e_+) = (\ux')^{-1} \circ \ux(e_+)$.
Just like in the previous case, it follows that $D(e) = (\ux')^{-1} \circ \ux(e)$.

We now prove equation~\eqref{eqn:boundaryD}. Let $r \in \partial_- E$. Since~$v$ and~$v'$ both give the same framed sphere~$\SS$,
we get the same point~$m \in M$ if we flow back along~$v$ from~$\ux^{-1}(r) \in \partial_- \E$
or if we flow back along~$v'$ from $(\ux')^{-1}(r)$.
But $f(\ux^{-1}(r)) = f((\ux')^{-1}(r)) = c -\eps$, hence $D(\ux^{-1}(r)) = (\ux')^{-1}(r)$.
Now let
\[
r \in S^{n-k} := \partial_+ E \cap \{\,x_1 = \dots = x_k = 0\,\}.
\]
Flowing forward along~$v$ from $\ux(S^{n-k})$ to~$M'$, or along~$v'$ from $\ux'(S^{n-k})$ to $M'$ give the same
parametrization of $W^u(p) \cap M' = W^u(p') \cap M'$. Indeed, they induce the same map $M(\SS) \to M'$,
and the handle part of~$M(\SS)$ is identified with~$\partial_+ E$.
So if we flow forward from~$\ux(r)$ to~$M'$ along~$v$ and then back along~$v'$ to $\partial_+ \E'$, then we get~$\ux'(r)$.
However, $f(\ux(r)) = f'(\ux'(r))$, hence $D(\ux(r)) = \ux'(r)$. This concludes the proof of equation~\eqref{eqn:boundaryD},
and by the previous paragraph, the proof of equation~\eqref{eqn:D}.

It follows that~$D$ is smooth in~$\E$. To see that it is smooth along~$W^u(p)$, note that if $x \in W$ and there is
a $t \in \R_{\le 0}$ for which $\g(x,t) \in \partial_+ \E$, then~$D(x)$ can also be obtained by
flowing forward from $D(\g(x,t))$ along~$v'$ until the value of~$f'$ becomes~$f(x)$.
Together with equation~\eqref{eqn:D}, which implies that~$D$ smoothly maps $\partial_+ \E$ to $\partial_+ \E'$,
and the fact that~$D$ maps flow-lines of~$v$ to flow-lines of~$v'$, we obtain that
$D$ is also smooth along~$W^u(p)$.

That $D|_M = \Id_M$ and $D|_{W^u(p) \cap M'} = \Id_{W^u(p) \cap M'}$
follow from the definition of~$D$. To see that $D|_{M' \setminus W^u(p)} = \Id_{M' \setminus W^u(p)}$,
note that~$v$ and~$v'$ induce the same diffeomorphisms $M(\SS) \to M'$. Hence, for every $x \in M \setminus a(\SS)$,
the flow-lines of~$v$ and~$v'$ starting at~$x$ end at the same point of~$M'$.
This concludes the proof when the cobordism is elementary and~$\ub = \ub'$.

We now consider the case of a general Cerf decomposition~$\CC$.
Choose an orientation preserving diffeomorphism $\phi \colon \R \to \R$ such that $\phi(\ub) = \ub'$
and such that~$\phi$ is linear in a neighborhood of each critical value of~$f$
(the latter is to ensure that~$v$ is also gradient-like at the critical points of~$\phi \circ f$).
We can then apply the previous argument to each elementary piece $W_i$
with Morse data~$(\phi \circ f|_{W_i}, (b_{i-1}',b_i'), v|_{W_i})$
and $(f'|_{W_i}, (b_{i-1}', b_i'), v'|_{W_i})$ to obtain diffeomorphisms
$D_i \colon W_i \to W_i$ that piece together to a diffeomorphism $D \colon W \to W$
with the required properties.
\end{proof}

\subsection{Moves on Morse data and parameterized Cerf decompositions} \label{sec:moves}
Next, we define some moves on Morse data.
We show that any two Morse data on the same cobordism can be connected by a sequence of such moves,
and describe what happens to the induced parameterized Cerf decompositions.
In the following, let $\MM = (f,\ub,v)$ and $\MM' = (f',\ub',v')$ be Morse data
on the oriented cobordism~$W$, and let $\CC = \CC(\MM)$ and $\CC' = \CC(\MM')$
be the induced parameterized Cerf decompositions. Furthermore, we denote by~$p_i$ the critical
point of~$f$ in~$W_i$, assuming~$W_i$ is not cylindrical.

We say that~$\MM$ and~$\MM'$ are related by a \emph{critical point cancelation}
(cf.~the analogous move of~\cite[Definition~2.8]{GWW}) if
there exists a one-parameter family
\[
\{\, (f_t,\ub_t,v_t) \,\colon t\, \in [-1,1] \,\}
\]
of triples such that
\begin{itemize}
\item $(f_{-1}, \ub_{-1}, v_{-1}) = \MM$ and $(f_1 ,\ub_1, v_1) = \MM'$,
\item $f_t$ is a family of smooth functions and $v_t$ is a family of smooth vector fields,
\item $(f_t,\ub_t,v_t)$ is a Morse datum for every $t \in [-1,1] \setminus \{0\}$,
\item $\ub_t$ is a constant $\ub = (b_0,\dots,b_{m+1})$ for $t \in [-1,0)$,
and there is a $j \in \{1,\dots,m\}$ such that $\ub_t = \ub \setminus \{b_j\}$ for $t \in (0,1]$,
\item the critical points $p_{j-1}(t) \in f_t^{-1}([b_{j-1},b_j])$ and $p_j(t) \in f_t^{-1}([b_j,b_{j+1}])$
of~$f_t$ for $t < 0$ cancel at $t = 0$, and $f_t$ has no critical values in $[b_{j-1},b_{j+1}]$ for~$t > 0$,
\item $W^u(p_{j-1}(t))$ and $W^s(p_j(t))$ are transverse and intersect in a single flow-line for every $t \in [-1,0)$,
\item $\{f_t \,\colon\, t \in [-1,1] \,\}$ is a ``chemin \'el\'ementaire de mort'' supported in a
small neighborhood~$U$ of
\[
\left(W^u(p_{j-1}(t)) \cup W^s(p_j(t))\right) \cap f^{-1}[b_{j-1}(t),b_{j+1}(t)];
\]
see Cerf~\cite[Section 2.3, p.71]{Cerf}.
Inside~$U$, the path~$f_t$ is of normal form, while outside~$U$, both $f_t$ and $v_t$ are constant.
\end{itemize}
Cerf \cite[Chapter II.2]{Cerf} proved that, given a pair of ascending and descending manifolds
for a pair of consecutive critical points that intersect in a single flow-line, the space of standard
neighborhoods is connected, and hence any two ``chemin \'el\'ementaire de mort'' starting at~$f$ compatible with
this stable and unstable manifold are homotopic through such families.
A \emph{critical point creation} is the reverse of a critical point cancelation.

We now define the diffeomorphism~$\varphi$ appearing in relation~\eqref{it:birth} of Definition~\ref{def:relations}.
Hatcher proved that $\Diff(D^3,\partial D^3)$ is contractible, hence every diffeomorphism of a $3$-manifold
supported in a ball is isotopic to the identity. So, when $n \le 3$, the diffeomorphism~$\varphi$
is uniquely characterized up to isotopy by the property that it fixes $M \cap M(\SS)(\SS')$.
However, in higher dimensions, $\Diff(D^n,\partial D^n)$ might be disconnected.
The reader only interested in the $n \le 3$ case, which covers all the applications
in this paper, can safely skip the following definition.

\begin{defn} \label{def:phi}
Suppose that $\SS' \subset M(\SS)$ is a framed sphere such that $a(\SS')$ intersects $b(\SS)$ transversely in one point.
Let~$W$ be the cobordism obtained by attaching a handle~$h$ to~$M \times I$ along~$\SS \times \{1\}$,
followed by a handle~$h'$ attached along $\SS'$. Consider
\[
B = \im(\SS) \cup (\im(\SS') \cap M)
\]
with its corners smoothed. This is diffeomorphic to a disk since $|a(\SS') \cap b(\SS)| = 1$.
Furthermore, let
\[
H = (B \times I) \cup h \cup h';
\]
this is diffeomorphic to $B \times I$.
Let $F \colon M \times I \to W$ be a diffeomorphism such that $F(x,0) = (x,0)$ for every~$x \in M$
and~$F(x,t) = (x,t)$ for every $x \in M \setminus B$ and $t \in I$. Then let $\varphi = F|_{M \times \{1\}}$.
To define~$F$, one only needs to choose a diffeomorphism from $B \times I$ to $H$ that is
the identity along~$(B \times \{0\}) \cup (\partial B \times I)$.
If~$F'$ is another such map, then the induced~$\varphi'$ differs from~$\varphi$ by a pseudo-isotopy supported in
the disk~$H \cap M(\SS)(\SS')$. By Cerf~\cite{Cerf}, for $n \ge 5$, any diffeomorphism of~$D^n$ that
fixes~$\partial D^n$ and is pseudo-isotopic to the identity is actually isotopic to the identity, as~$D^n$
is simply-connected. The only case when we do not know whether~$\varphi$ is well-defined up to isotopy is when~$n=4$.

The following construction works in all dimensions.
Now let~$W$ be the cobordism obtained by composing $W(\SS)$ and $W(\SS')$.
By Lemma~\ref{lem:lift}, there is a Morse function~$f$ on~$W$ and a gradient-like vector field~$v$
that are compatible with the natural parameterized Cerf decomposition of~$W$
with diffeomorphisms~$\Id_{M(\SS)}$ and~$\Id_{M(\SS)(\SS')}$.
In particular, $f$ has exactly two critical points~$p$ and~$p'$ at the centers of~$h$ and~$h'$, respectively.
Furthermore, the stable manifold $W^s(p)$ is the core of~$h$ union~$\SS \times I$, the unstable manifold
$W^u(p) \cap W(\SS)$ is the co-core of~$h$, and similarly, $W^s(p') \cap W(\SS')$ is the core of~$h'$ union~$\SS' \times I$,
while~$W^u(p')$ is the co-core of~$h'$.
There is a homotopically unique 1-parameter family $\{\,f_t \,\colon\, t \in [-1,1] \,\}$
of smooth functions $(W,\partial W) \to (I,\partial I)$ such that $f_{-1} = f$, it has a single death bifurcation at $t = 0$, and
the stable manifold of the larger critical point and the unstable manifold of the smaller critical point
remain transverse for~$t \in [-1,0)$.
In the terminology of Cerf~\cite[Proposition~2, Chapitre~III]{Cerf}, there is a ``chemin \'el\'ementaire;''
i.e., an elementary path canceling the two critical points that can be described in a local model in a neighborhood~$U$
of $W^u(p) \cup W^s(p')$. Outside~$U$, the family~$f_t$ is constant.
In particular, $f_1$ has no critical points, and
according to Cerf~\cite{Cerf}, the space of such paths is connected.
Hence, if~$f_t$ and~$f_t'$ are two different paths, then~$f_1$ and~$f_1'$ are homotopic through smooth functions
with no critical points. The gradient flows of~$f_1$ and~$f_1'$
give rise to isotopic diffeomorphisms from~$M$ to~$M(\SS)(\SS')$, and changing the metric also preserves
the isotopy class.
\end{defn}

It is important to note that keeping the ascending and descending manifolds of the canceling critical points transverse throughout
(or equivalently, the pair of spheres obtained by intersecting them with~$M(\SS)$) is what ensures the uniqueness.
The space of ascending and descending manifolds intersecting in a single flow-line might have several components,
each of which might result in different cancelations.
Also see the First Cancelation Theorem of Morse in the book of Milnor~\cite[Theorem~5.4]{Milnor}.

\begin{rem}
In relation~\eqref{it:isot} of Definition~\ref{def:relations},
to prove Theorem~\ref{thm:TQFT}, it would suffice to assume that $d \sim \Id_M$ whenever~$d$
is isotopic to the identity \emph{and supported in a ball}. However, according to the classical
result of Palis and Smale~\cite{PS}, such diffeomorphisms generate~$\Diff_0(M)$.
\end{rem}

\begin{lem} \label{lem:cancelation}
Suppose that the Morse data $\MM = (f,\ub,v)$ and $\MM' = (f',\ub',v')$ are related by a
critical point cancelation $\{\, (f_t,\ub_t,v_t) \,\colon t\, \in [-1,1] \,\}$.
Then the corresponding parameterized Cerf decompositions
$\CC = \CC(\MM)$ and $\CC' = \CC(\MM')$ are related as follows.

The attaching sphere $a(\SS_{j+1})$ intersects $d_j(b(\SS_j))$
in a single point, where $b(\SS_j) \subset M_j(\SS_j)$ is the belt sphere of the handle in~$W_j(\SS_j)$.
The cobordism $W_j \cup W_{j+1}$ is cylindrical. We obtain~$\CC'$ from~$\CC$ by removing $M_{j+1}$, more precisely,
\[
    M_i' =
\begin{cases}
    M_i & \text{if } i < j+1, \\
    M_{i+1} &  \text{otherwise.}
\end{cases}
\]
We obtain the framed attaching spheres~$\SS_i'$ and the diffeomorphisms~$d_i'$ for $i \neq j$
analogously. We have $\SS_j' = \emptyset$, and let $S_{j+1} = d_j^{-1} \circ \SS_{j+1} \subset M_j(\SS_j)$. To determine
\[
d_j' \colon M_j'(\SS_j') = M_j \to M_{j+1}' =  M_{j+2},
\]
note that there is a diffeomorphism
\[
\varphi \colon M_j \to M_j(\SS_j)(S_{j+1})
\]
defined as in Definition~\ref{def:phi}.
Furthermore, $d_j$ induces a diffeomorphism
\[
d_j^{S_{j+1}} \colon M_j(\SS_j)(S_{j+1}) \to M_{j+1}(\SS_{j+1}).
\]
Then
\begin{equation} \label{eqn:cancelation}
d_j' \approx d_{j+1} \circ d_j^{S_{j+1}} \circ \varphi,
\end{equation}
where `$\approx$' means `isotopic to.'
\end{lem}

\begin{proof}
We prove equation~\eqref{eqn:cancelation}, the rest of the statement is straightforward.
Let~$\oW$ be the cobordism obtained by gluing $W(\SS_j)$ and~$W(S_{j+1})$ along~$M(\SS_j)$.
This carries a parameterized Cerf decomposition $\ol{\CC}$, with diffeomorphisms $\Id_{M(\SS_j)}$
and $\Id_{M(\SS_j)(S_{j+1})}$. According to Lemma~\ref{lem:lift}, there exists a Morse
datum $(\ol{f},\ol{\ub},\ol{v})$ inducing~$\ol{\CC}$.

Next, we construct a diffeomorphism $G \colon \oW \to W_j \cup W_{j+1}$.
Choose an extension $D_i \colon W_i(\SS_i) \to W_i$ of $d_i$ for $i \in \{j, j+1\}$.
Then~$D_j$ and~$D_{j+1}$ glue together to a diffeomorphism
\[
G_0 \colon W(\SS_j) \cup_{d_j} W(\SS_{j+1}) \to W_j \cup W_{j+1}.
\]
Furthermore, we can glue together~$\Id_{W(\SS_j)}$ and the diffeomorphism
\[
D_j^{S_{j+1}} \colon W(S_{j+1}) \to W(\SS_{j+1})
\]
extending $d_j^{S_{j+1}}$
to a diffeomorphism $G_1 \colon \oW \to W(\SS_j) \cup_{d_j} W(\SS_{j+1})$.
Then we set $G = G_0 \circ G_1$.

The Morse datum $(f \circ\, G, (b_{j-1}, b_j, b_{j+1}), G^*(v))$ on $\oW$ also induces the parameterized Cerf
decomposition~$\ol{\CC}$. Hence, by Lemma~\ref{lem:unique}, there exists a diffeomorphism $D \colon \oW \to \oW$
that fixes $M_j$, $M(\SS_j)$, and $M(\SS_j)(S_{j+1})$ pointwise, and such that $f \circ G \circ D = \ol{f}$
and $(G \circ D)^*(v) = \nu \cdot \ol{v}$ for some $\nu \in C^\infty(\oW,\R_+)$. In particular, $f_t \circ G \circ D$ for $t \in [-1,1]$
is a ``chemin \'el\'ementaire de mort'' starting from $\ol{f}$ and ending at a function $f_1 \circ G \circ D$
with no critical points that induces the diffeomorphism $\varphi \colon M_j \to M_j(\SS_j)(S_{j+1})$, up to isotopy.
Indeed, by Cerf~\cite[Chapter 2.3]{Cerf}, the space of ``chemin \'el\'ementaire'' starting at a given Morse function
that cancel two consecutive critical points with a single flow-line between them, and which is supported in
a neighborhood of their stable and unstable manifolds where it is in normal form is connected,
and so their endpoints can be connected through Morse functions with no critical points. So, for any choice of
gradient-like vector fields, the endpoints induce isotopic diffeomorphisms.
Hence~$f_1$ on $W_j \cup W_{j+1}$ induces a diffeomorphism $d_j' \colon M_j \to M_{j+2}$ that is conjugate
to~$\varphi$ along~$G$. As $G|_{M} = \Id_M$ and $G|_{M(\SS_j)(S_{j+1})} = d_{j+1} \circ d_j^{S_{j+1}}$,
we obtain equation~\eqref{eqn:cancelation}.
\end{proof}


We say that the Morse data~$\MM$ and~$\MM'$ are related by a \emph{critical value crossing} if
there exists a one-parameter family
\[
\{\, (f_t,\ub_t,v_t) \,\colon\, t \in [-1,1] \,\}
\]
of triples such that
\begin{itemize}
\item $(f_{-1}, \ub_{-1}, v_{-1}) = \MM$ and $(f_1 ,\ub_1, v_1) = \MM'$,
\item $f_t$ is a family of Morse functions with critical set $\text{Crit}(f_t) = \{p_0,\dots,p_m\}$
independent of~$t$, and $v_t$ is a family of smooth vector fields,
\item $(f_t,\ub_t,v_t)$ is a Morse datum for every $t \in [-1,1] \setminus \{0\}$,
\item there is a $j \in \{0,\dots,m\}$ such that $b_i(t) = b_i$ is independent of~$t$ for $i \neq j+1$,
where $\ub_t = (b_0(t), \dots, b_{m+1}(t))$,
\item two critical values cross each other; i.e., $f_t(p_j) < f_t(p_{j+1})$ for~$t < 0$ and
$f_t(p_j) > f_t(p_{j+1})$ for $t > 0$, with equality for $t = 0$,
\item $W^u(p_j) \cap W^s(p_{j+1}) = \emptyset$ for every $t \in [-1,1]$,
\item $\{f_t \,\colon\, t \in [-1,1] \,\}$ is a ``chemin \'el\'ementaire
de croisement ascendant or descendant'' with support in a
small neighborhood~$U$ of
\begin{align*}
&W^s(p_j) \cap f^{-1}[b_j,b_{j+2}] \text{ or} \\
&W^s(p_{j+1}) \cap f^{-1}[b_j,b_{j+2}];
\end{align*}
see Cerf~\cite[Chapter~II, p.40]{Cerf}.
Inside~$U$, the path~$f_t$ is of normal form, while outside~$U$ both $f_t$ and $v_t$ are constant.
\end{itemize}

\begin{lem} \label{lem:switch}
Suppose that the Morse data~$\M = (f,\ub,v)$ and $\M' = (f',\ub',v')$ are related by a critical value crossing
$\{\, (f_t,\ub_t,v_t) \,\colon\, t \in [-1,1] \,\}$,
and consider  the induced parameterized Cerf decompositions~$\CC$ and~$\CC'$. Then these
satisfy the following properties:
\begin{enumerate}
\item \label{it:cross} in $\CC$, the part $W_j \cup_{M_{j+1}} W_{j+1}$ is replaced by $W_j' \cup_{M_{j+1}'} W_{j+1}'$,
the rest of the decomposition is unchanged,
\item \label{it:empty} $\text{Im}(\SS_j) \cap \text{Im}(\SS_j') = \emptyset$ and
$\im(\SS_{j+1}) \cap d_j(D^{k_j} \times S^{n-k_j}) = \emptyset$,
\item \label{it:spheres} $d_j' \circ \SS_j = \SS_{j+1}'$ (where we consider $\SS_j$ in $M_j(\SS_j')$)
and $d_j \circ \SS_j' = \SS_{j+1}$ (where we consider $\SS_j'$ in $M_j(\SS_j)$), and


\item \label{it:commute} the following diagram is commutative up to isotopy:
\[
\xymatrix{
  M_j(\SS_j,\SS_j') \ar[r]^-{(d_j)^{\SS_j'}} \ar[d]^-{(d_j')^{\SS_j}} & M_{j+1}(\SS_{j+1}) \ar[d]^{d_{j+1}} \\
  M_{j+1}'(\SS_{j+1}')  \ar[r]^-{d_{j+1}'} & M_{j+2}.}
\]
\end{enumerate}
\end{lem}

\begin{proof}
Without loss of generality, suppose we are dealing with an ascending path; i.e., the critical value
$f_t(p_j)$ increases until it gets above $f_t(p_{j+1}) = f(p_{j+1})$. The deformation of $(f_t,v_t)$ is supported in a
saturated neighborhood~$U$ of~$W^s(p_j) \cap f_t^{-1}([b_j,\infty))$.
To see~\eqref{it:cross}, note that if $i \not \in \{j,j+1\}$, then on $W_i$ the function and the vector field remain
unchanged, and so do the regular values $b_i$ and $b_{i+1}$. The deformation is supported inside $W_j \cup W_{j+1}$,
and $b_{j+1}(t)$ stays between the critical values $f_t(p_j)$ and $f_t(p_{j+1})$ for every $t \in [-1,1]$.
Part~\eqref{it:empty} follows from the facts that $W^s(p_j) \cap W^s(p_{j+1}) = \emptyset$ and
\[
W^u(p_j) \cap W^s(p_{j+1}) \cap M_{j+1} = \emptyset.
\]

To prove~\eqref{it:spheres}, recall from Definition~\ref{def:CM}
that~$\SS_j$ is given by $W^s(p_j) \cap M_j$, with framing coming from
a local normal form of~$f$ about~$p_j$. Along an elementary path, this local form remains the same except for
a constant shift. In particular, $W^s(p_j)$ intersects $M_j$ in $a(\SS_j)$ with the same framing,
and $M_{j+1}'$ in $a(\SS_{j+1}')$. Hence, if we flow from $\im(\SS_j)$ along $v_1$ to $M_{j+1}'$,
we obtain $d_j' \circ \SS_j = \SS_{j+1}'$ as $\text{Im}(\SS_j) \cap \text{Im}(\SS_j') = \emptyset$.
Similarly, $W^s(p_{j+1})$ intersects~$M_j$ in~$a(\SS_j')$ and~$M_{j+1}$ in~$a(\SS_{j+1})$, so flowing along $v = v_{-1}$,
we see that $d_j \circ \SS_j' = \SS_{j+1}$.

Finally, we show part~\eqref{it:commute}; i.e., that
\[
d_{j+1} \circ d_j^{\SS_j'}(x) = d_{j+1}' \circ (d_j')^{\SS_j}(x)
\]
for every $x \in M_j(\SS_j,\SS_j')$. Since the deformation $(f_t,v_t)$ is supported in a
neighborhood of $W^s(p_j)$, for every $x \in M_j \setminus (a(\SS_j) \cup a(\SS_j'))$ this is clear since
both compositions are  induced by flowing along~$v$ from~$M_j$ to~$M_{j+2}$.
When $x$ is in the handle part of $M_j(\SS_j,\SS_j')$ corresponding to~$\SS_j'$, both compositions
are obtained by flowing along~$v$ from the corresponding point of a standard neighborhood of~$p_{j+1}$
to~$M_{j+2}$. When $x$ is in the handle part corresponding to~$\SS_j$, flowing up to~$M_{j+2}$ along~$v$ or~$v'$
give isotopic diffeomorphisms since for an elementary deformation
$f_t - f$ is constant near~$p_j$ and $v_t$ is the Euclidean gradient.
\end{proof}


We say that~$\M$ and~$\M'$ are related by \emph{an isotopy of the gradient} if $f = f'$ and~$\ub = \ub'$.
Given a parameterized Cerf decomposition~$\CC$, an \emph{isotopy of a framed attaching sphere} is a move described as follows.
Let~$\varphi_t \colon M_j \to M_j$ for~$t \in I$ be an isotopy,
and let~$\SS_j' = \varphi_1 \circ \SS_j$. There is an induced map
\[
\varphi_1' = (\varphi_1)^{\SS_j} \colon M_j(\SS_j) \to M_j(\SS_j'),
\]
and we let $d_j' := d_j \circ (\varphi_1')^{-1}$. It is easy to see that~$d_j'$ extends to
a diffeomorphism $D_j' \colon W(\SS_j') \to W_j$ via the formula
\[
D_j'(x,t) = \left(D_j \circ \varphi_t^{-1}(x),t \right)
\]
for $(x,t) \in M_j \times I$, and extending to the handle in the natural way.

\begin{lem} \label{lem:unique-param}
Let~$(f,\ub)$ be a Morse datum for the cobordism~$W$. If~$\CC$ and~$\CC'$ are parameterized Cerf
decompositions induced by the triples~$(f,\ub, v)$ and $(f,\ub,v')$, respectively,
then they are related by isotopies of the framed attaching spheres~$\SS_i$ and of the diffeomorphisms~$d_i$,
and possibly by reversing framed spheres.
\end{lem}

\begin{proof}
This is a direct consequence of Remark~\ref{rem:dependence}.
\end{proof}

The Morse data~$\M$ and~$\M'$ are related by \emph{adding or removing a regular value} if
$|\ub \bigtriangleup \ub'| = 1$, where $\bigtriangleup$ denotes the symmetric difference.
In this case, there is an~$i \in \N$ for which either $[b_i,b_{i+1}]$
contains no critical value of~$f$, or $[b_i',b_{i+1}']$ contains no critical value of~$f'$.
Then the corresponding parameterized Cerf decompositions are related by \emph{merging or splitting a product:}
Suppose that one of~$W_j$ and~$W_{j+1}$ is cylindrical; i.e.,
$\SS_j$ or~$\SS_{j+1}$ is empty. We describe the case when~$\SS_j = \emptyset$, the other case is analogous.
Then we remove~$M_{j+1}$ and merge~$W_j$ and~$W_{j+1}$. We set~$\SS_j' = d_j^{-1} \circ \SS_{j+1}$
and
\[
d_j' = d_{j+1} \circ (d_j)^{\SS_j'} \colon M_j(\SS_j') \to M_{j+2},
\]
where $(d_j)^{\SS_j'} \colon M_j(\SS_j') \to M_{j+1}(\SS_{j+1})$
is the diffeomorphism induced by~$d_j \colon M_j \to M_{j+1}$.
Splitting a product is the reverse of the above move.
In general, we have the following result for changing~$\ub$.

\begin{lem} \label{lem:changing-b}
Suppose that $(f,\ub,v)$ and $(f,\ub',v)$ are Morse data for the cobordism~$W$, and let~$\CC$ and~$\CC'$
be the corresponding parameterized Cerf decompositions. Then $(f,\ub \cup \ub',v)$ is also a Morse datum
for~$W$, and $\CC'' = \CC(f,\ub \cup \ub',v)$ can be obtained
from both~$\CC$ and~$\CC'$ by splitting products. In particular, one can get from~$\CC$ to~$\CC'$ by splitting
then merging products.
\end{lem}

Finally, $\M$ and $\M'$ are related by a \emph{left-right equivalence} if
there are diffeomorphisms $\Phi \colon W \to W$ and $\varphi \colon \R \to \R$
such that $f' = \varphi \circ f \circ \Phi^{-1}$, $\ub' = \varphi(\ub)$, $v' = \Phi_*(v)$,
$\Phi|_M \colon M \to M$ is isotopic to~$\Id_M$,
and $\Phi|_{M'} \colon M' \to M'$ is isotopic to~$\Id_{M'}$.
Then we obtain $\CC(\M')$ from $\CC(\M)$ by a \emph{diffeomorphism equivalence}; i.e.,
setting $W_i' = \Phi(W_i)$, $\SS_i' = \Phi \circ \SS_i$, and
\[
d_i' = \Phi_{i+1} \circ d_i \circ \left(\Phi_i^{\SS_i} \right)^{-1},
\]
where $\Phi_i = \Phi|_{M_i}$.

The content of the following lemma is that an isotopy of one of the~$d_j$
can be written in terms of the above moves on parameterized Cerf decompositions.

\begin{lem} \label{lem:isot-d}
Suppose that the parameterized Cerf decomposition~$\CC'$ is obtained from~$\CC$ by replacing one
of the diffeomorphisms~$d_j$ by a diffeomorphism~$d_j' = \phi \circ d_j$,
where $\phi \colon M_{j+1} \to M_{j+1}$ is isotopic to~$\Id_{M_{j+1}}$.
If we extend~$\phi$ to a diffeomorphism $\Phi \colon W \to W$ isotopic to~$\Id_W$
and supported in a collar neighborhood of~$M_{j+1}$, then $\CC'$ can also be obtained
from~$\CC$ by performing the diffeomorphism equivalence corresponding to~$\Phi$,
and then isotoping $\phi \circ \SS_{j+1}$ back to~$\SS_{j+1}$.
\end{lem}

\begin{proof}
It is clear that $W_i = W_i'$, $M_i = M_i'$, and $\SS_i = \SS_i'$ for any~$i \in \{0,\dots,m+1\}$.
What we do need to check is that $d_j = d_j'$ and $d_{j+1} = d_{j+1}'$. If we use the notation
$\Phi_i = \Phi|_{M_i}$, then $\Phi_i = \Id_{M_i}$ unless~$i = j+1$.
Hence, the diffeomorphism equivalence replaces $d_j$ by $\Phi_{j+1} \circ d_j = \phi \circ d_j$ and
$d_{j+1}$ by $d_{j+1} \circ \left(\Phi_{j+1}^{\SS_{j+1}} \right)^{-1} = d_{j+1} \circ \left(\phi^{\SS_{j+1}} \right)^{-1}$.
Then isotoping~$\phi \circ \SS_{j+1}$ back to $\SS_{j+1}$ replaces $d_{j+1} \circ \left(\phi^{\SS_{j+1}} \right)^{-1}$
by
\[
d_{j+1} \circ \left(\phi^{\SS_{j+1}} \right)^{-1} \circ \phi^{\SS_{j+1}} = d_{j+1}. \qedhere
\]
\end{proof}

\begin{thm} \label{thm:Cerf}
Let $\M = (f,\ub,v)$ and $\M' = (f',\ub',v')$ be Morse data on the oriented cobordism~$W$. Then they can be connected by a
sequence of critical point creations and cancelations, critical value crossings,
isotopies of the gradient, adding or removing regular values, and left-right equivalences.

Furthermore, if the ends of each component of the cobordism~$W$ are non-empty,
then we can avoid index~$0$ and~$n+1$ critical points throughout.
If, in addition, we assume that~$n \ge 2$, and the cobordism~$W$ and each level
set~$f^{-1}(b_i)$ and~$(f')^{-1}(b_j')$ are connected,
then we can choose the above sequence such that in the corresponding parameterized Cerf decompositions all level sets
are connected. In particular, there are no index~$0$ or~$n+1$ critical points throughout, and no index~$n$
critical points with separating attaching spheres.
\end{thm}

\begin{proof}
Connect~$f$ and~$f'$ by a generic one-parameter family~$\{\, f_s \,\colon\, s \in [0,1] \,\}$ of smooth functions.
This family fails to be a Morse function at the parameter values $c_1, \dots, c_l$, where either
we have a birth-death singularity, or two critical points have the same value.
We also choose parameter values $s_0, \dots, s_{2l+1}$ such that
\[
0 = s_0 < s_1 < c_1 < s_2 < s_3 < c_2 <  \dots < s_{2l-2} < s_{2l-1} < c_l < s_{2l} < s_{2l+1} = 1,
\]
and $s_{2i-1}$ and $s_{2i}$ are close to~$c_i$ in a sense to be specified below. For every $i \in \{0, \dots, 2l+1\}$,
let $v_i$ be a gradient-like vector field for $f_i = f_{s_i}$.
Furthermore, for every $i \in \{0, \dots, l\}$, choose the ordered tuples $\ub_{2i}$ and $\ub_{2i+1}$
such that $\M_{2i} = (f_{2i},\ub_{2i},v_{2i})$ and $\M_{2i+1} = (f_{2i+1},\ub_{2i+1},v_{2i+1})$ are Morse data,
and such that they can be connected by a continuous path of tuples $\ub(s)$ consisting of regular values of~$f_s$
for $s \in [s_{2i}, s_{2i+1}]$.
Then, by \cite[Lemma~3.1]{GWW}, the Morse data $\M_{2i}$ and $\M_{2i+1}$ are related by
a left-right equivalence and an isotopy of the gradient.
Furthermore, by Lemma~\ref{lem:changing-b}, different choices of~$\ub$ give decompositions
related by adding and removing regular values.

It remains to prove that~$\M_{2i-1}$ and~$\M_{2i}$ are related by one of the moves listed in the statement
for a fixed $i \in \{0,\dots,l\}$.
To simplify the notation, let $\M_- = \M_{2i-1}$, $\M_+ = \M_{2i}$, $s_- = s_{2i-1}$, $s_+ = s_{2i}$, $f_\pm = f_{s_\pm}$,
$v_\pm = v_{s_\pm}$, and $c = c_i$.
Choose an ordered tuple~$\ub$ such that there is exactly one element of~$\ub$ between any two consecutive critical points of~$f_c$.

First, suppose that the function~$f_c$ has a death singularity at~$p \in W$ with $f_c(p) \in (b_j, b_{j+1})$.
According to Cerf~\cite[p.~71, Proposition~2]{Cerf}, we can modifying the family~$f_s$ such that it
becomes a ``chemin \'el\'ementaire de mort'' for~$s$ in a neighborhood of~$c$.
(Note that this modification does change the function at the endpoint of the family,
but it is left-right equivalent and hence isotopic to the original endpoint via a two-sided isotopy.
Our modified family is the juxtaposition of the chemin \'el\'ementaire de mort and this
two-sided isotopy, and we take $s_-$ and $s_+$ to be the start and endpoint of the chemin \'el\'ementaire de mort.)
In particular, $f_s$ is constant in~$s$ outside a ball~$B \subset f_c^{-1}([b_j,b_{j+1}])$
containing~$p$ for $s \in [s_-,s_+]$, if~$s_\pm$ are very close to~$c$. Furthermore, there is a coordinate system
about~$p$ in which
\[
f_s(\ux) = f_c(p) + x_1^3 + s x_1 - x_2^2 - \dots - x_k^2 + x_{k+1}^2 + \dots + x_{n+1}^2.
\]
We choose~$v_-$ and~$v_+$ to be gradient-like vector fields for~$f_-$ and~$f_+$, respectively, that coincide outside~$B$.
Notice that~$f_c(p)$ lies between the values of the two critical points that cancel for~$s < 0$,
hence $(f_-,\ub_-)$ is a Morse datum for~$\ub_- = \ub \cup \{f_c(p)\}$.
Then $(f_-,\ub_-, v_-)$ and $(f_+,\ub,v_+)$ are Morse data for~$W$.
It follows from the above construction that in~$\M_-$ the attaching sphere and the belt sphere of the canceling
pair of critical points intersect in a single point. So $\M_-$ and $\M_+$ are  related by a critical point cancelation.


Now consider the case when~$f_c$ has two critical points at~$p$ and~$q$ such that
\[
f_c(p) = f_c(q) \in [b_j,b_{j+1}].
\]
Then we can modify the family $f_s$ in the interval $[s_-,s_+]$ such that it becomes a ``chemin \'el\'ementaire de 1-croisement''
in a neighborhood of~$c$; this is possible by Cerf~\cite[p.~49, Proposition~2]{Cerf}.
In particular, $f_s$ is independent of~$s$ outside a neighborhood~$N$ of either~$W^s(p)$ or~$W^s(q)$,
and the points~$p$ and~$q$ remain critical throughout. Furthermore, for $s \in [s_-, c)$, we have $f_s(p) < f_s(q)$,
while for $s \in (c, s_+]$, we have $f_s(q) < f_s(p)$. In fact, we can arrange that a fixed vector field~$v$ on~$W$
remains gradient-like for every~$f_s$. If we set $\ub_\pm = \ub \cup \{(f_{s_\pm}(p) + f_{s_\pm}(q))/2\}$, then
$(f_-,\ub_-, v)$ and $(f_+,\ub_+, v)$ are Morse data. 
Then we can get from~$\M_-$ to~$\M_+$ by a critical value crossing
and isotopies of the gradient. 

If each component of the cobordism~$W$ has non-empty ends, then we can avoid index~$0$ and~$n+1$ critical
points using Cerf theory as in the work of Kirby~\cite{Kirby}.
The statement on connected Cerf decompositions follows from~\cite[Theorem~3.6]{GWW}.
\end{proof}


\section{The presentation of $\Cob_n$ and constructing TQFTs}
\label{sec:construction}

In this section, we describe how Theorem~\ref{thm:Cerf}, together with the lemmas of the
previous section translating moves on Morse data to moves on parameterized Cerf decompositions,
imply Theorem~\ref{thm:presentation}. Then we show how Theorem~\ref{thm:TQFT} follows from
Theorem~\ref{thm:presentation}. We now restate Theorem~\ref{thm:presentation} for the reader's convenience.

\begin{thm*}
The functor $c \colon \cF(\G_n) \to \Cob_n$ descends to a functor
\[
\cF(\G_n)/\cR \to \Cob_n
\]
that is an isomorphism of symmetric monoidal categories.

By slight abuse of notation, we will also denote the functor $\cF(\G_n)/\cR \to \Cob_n$ by~$c$.
Then $c$ restricted to $F(\G_n')/\cR$ is an isomorphism onto $\Cob_n'$
and $c$ restricted to $F(\G_n^0)/\cR$ is an isomorphism onto $\Cob_n^0$.
Finally, $c^s \colon \cF(\G^s) \to \BSut'$ descends to a functor $\cF(\G^s)/\cR^s$
that is an isomorphism of symmetric monoidal categories.
\end{thm*}

\begin{proof}[Proof of Theorem~\ref{thm:presentation}]

To show that $c$ descends to $\cF(\G_n)/\cR$, it suffices to check that if we apply~$c$ to
any relation in $\cR$, then the resulting relation holds in the cobordism category~$\Cob_n$.
Applying~$c$ to relation~\eqref{it:isot} gives $c_{d \circ d'} = c_d \circ c_{d'}$.
Furthermore, if $d \in \Diff_0(M)$, then $c_d = i_M$.
Both of these hold by the discussion following Definition~\ref{def:diffcob}.
If we apply~$c$ to relation~\eqref{it:d-F}, then we obtain
\[
c_{d^\SS} \circ W(\SS) = W(\SS') \circ c_d,
\]
where $d \colon M \to M'$ is a diffeomorphism, $\SS$ is a framed sphere in~$M$, and $\SS' = d \circ \SS$,
which is straightforward.
Relation~\eqref{it:commut} amounts to
\[
W(M(\SS),\SS') \circ W(M,\SS) = W(M(\SS'),\SS) \circ W(M,\SS')
\]
for disjoint framed spheres $\SS$ and $\SS'$ in $M$ (here we also specify the manifold we are performing
surgery on in the notation of the trace to avoid ambiguity). This is standard in handle theory;
see the proof of Lemma~\ref{lem:switch}.
Relation~\eqref{it:birth}, cancelation, amounts to
\[
W(\SS') \circ W(\SS) = c_\varphi,
\]
and follows from the discussion in Definition~\ref{def:phi}.
Finally, relation~\eqref{it:0-sphere}, holds as $W(\SS) = W(\ol{\SS})$.

We now show that $c \colon \cF(\G_n)/\cR \to \Cob_n$ is an isomorphism of symmetric mo\-no\-idal categories.
Suppose that~$W$ is an oriented cobordism from~$M$ to~$M'$.
Choose a Morse datum $(f,\ub,v)$ for $W$.
By Definition~\ref{def:CM}, this induces a parameterized Cerf
decomposition~$\CC$ of~$W$, consisting of a decomposition
\[
W = W_0 \cup_{M_1} W_1 \cup_{M_2} \dots \cup_{M_m} W_m,
\]
together with framed attaching spheres~$\SS_i$ and diffeomorphisms~$d_i \colon M_i(\SS_i) \to M_{i+1}$.
When $n \ge 2$ and~$W$, $M$, and~$M'$ are all connected, we can assume that each~$M_i$ is connected
as well by~\cite[Lemma~2.5]{GWW}.
As explained in the introduction after the statement of Theorem~\ref{thm:presentation},
the decomposition~$\CC$ corresponds to the morphism
\[
f_\CC := \prod_{i=0}^{m} d_i \circ e_{M_i,\SS_i}
\]
in the category $\cF(\G_n)$. Then $W = c(f_\CC)$,
showing that the functor $c \colon \cF(\G_n)/\cR \to \Cob_n$ is surjective onto the morphisms of $\Cob_n$.

Suppose that $W$ and $W'$ are equivalent cobordisms from~$M$ to~$M'$,
with equivalence given by the diffeomorphism~$h \colon W \to W'$ fixing~$M$ and~$M'$ pointwise.
Let~$\CC$ be a parameterized Cerf decomposition of~$W$, as above.
Then~$h$ induces a parameterized Cerf decomposition~$\CC'$ of~$W'$ by setting $W_i' = h(W_i)$, $\SS_i' = d \circ \SS_i$,
and
\[
d_i' = h_{i+1} \circ d_i \circ \left( h_i^{\SS_i} \right)^{-1} \colon M_i'(\SS_i') \to M_{i+1}',
\]
where $h_i = h|_{M_i}$ for $i \in \{0,\dots,m\}$.
We claim that
\[
f_\CC \sim f_{\CC'},
\]
where $f_{\CC'}$ is the morphism in $\cF(\G_n)$ arising from~$\CC'$.
Indeed, consider the diagram
\[
\xymatrixcolsep{3pc}\xymatrix{
  M_i \ar[r]^-{e_{M_i,\SS_i}} \ar[d]^-{h_i} & M_i(\SS_i) \ar[d]^{h_i^{\SS_i}} \ar[r]^{d_i} & M_{i+1} \ar[d]^{h_{i+1}} \\
  M_i'  \ar[r]^-{e_{M_i',\SS_i'}} & M_i'(\SS_i') \ar[r]^{d_i'} & M_{i+1}'.}
\]
The rectangle on the left is commutative because of relation~\eqref{it:d-F} of Definition~\ref{def:relations},
while the rectangle on the right commutes by the above definition of~$d_i'$ and relation~\eqref{it:isot}.
Putting the above rectangles together for $i \in \{0, \dots, m\}$, and using the property
that $h_0 = \Id_M$ and $h_{m+1} = \Id_{M'}$, the claim follows.

As we shall see, the content of Theorem~\ref{thm:Cerf} is that,
for any two parameterized Cerf decompositions~$\CC$ and $\CC'$ of a cobordism~$W$,
we can get from $f_\CC$ to $f_{\CC'}$ via relations in~$\cR$. Together with
the previous paragraph, this implies that~$c$ is injective on~$\cF(\G_n)/\cR$.

By Lemma~\ref{lem:lift}, there exist Morse data $\M = (f,\ub,v)$
and $\M' = (f',\ub',v')$ inducing~$\CC$ and~$\CC'$, respectively.
It suffices to prove that $f_\CC \sim f_{\CC'}$
when~$\M'$ is obtained from~$\M$ by one of the moves listed in Theorem~\ref{thm:Cerf}, since
any two Morse data can be connected by a sequence of such moves.

First, suppose that~$\M'$ is obtained from~$\M$ by a critical point cancelation.
Then what we need to show is that
\begin{equation} \label{eqn:canc}
d_{j+1} \circ e_{M_{j+1},\SS_{j+1}} \circ d_j \circ e_{M_j,\SS_j} \sim d_j'.
\end{equation}
By Lemma~\ref{lem:cancelation}, $d_j' \approx d_{j+1} \circ d_j^{S_{j+1}} \circ \varphi$, where $S_{j+1} = d_j^{-1} \circ \SS_{j+1}$.
Hence, using relation~\eqref{it:isot}, equation~\eqref{eqn:canc} reduces to
\[
e_{M_{j+1},\SS_{j+1}} \circ d_j \circ e_{M_j,\SS_j} \sim d_j^{S_{j+1}} \circ \varphi.
\]
By relation~\eqref{it:birth} of Definition~\ref{def:relations}, we have
\[
\varphi \sim e_{M_j(\SS_j),S_{j+1}} \circ e_{M_j,\SS_j}.
\]
Now, according to relation~\eqref{it:d-F},
\[
d_j^{S_{j+1}} \circ e_{M_j(\SS_j),S_{j+1}} \sim e_{M_{j+1},\SS_{j+1}} \circ d_j,
\]
and the result follows. The case of a critical point creation follows by reversing the roles of~$\M$ and~$\M'$.

Now assume that~$\M$ and~$\M'$ are related by a critical value crossing.
Then we will show that
\begin{equation} \label{eqn:switch}
d_{j+1} \circ e_{M_{j+1},\SS_{j+1}} \circ d_j \circ e_{M_j,\SS_j} \sim
d_{j+1}' \circ e_{M_{j+1}',\SS_{j+1}'} \circ d_j' \circ e_{M_j,\SS_j'}.
\end{equation}
Using relation~\eqref{it:d-F} and part~\eqref{it:spheres} of Lemma~\ref{lem:switch},
\[
e_{M_{j+1},\SS_{j+1}} \circ d_j \sim (d_j)^{\SS_j'} \circ e_{M_j(\SS_j),\SS_j'},
\]
and similarly,
\[
e_{M_{j+1}',\SS_{j+1}'} \circ d_j' \sim (d_j')^{\SS_j} \circ e_{M_j'(\SS_j'),\SS_j}.
\]
Substitute these into equation~\eqref{eqn:switch}, and notice that, by relation~\eqref{it:commut}
and part~\eqref{it:empty} of Lemma~\ref{lem:switch}, we have
\[
e_{M_j(\SS_j),\SS_j'} \circ e_{M_j,\SS_j} \sim e_{M_j'(\SS_j'),\SS_j} \circ  e_{M_j,\SS_j'},
\]
so it suffices to prove that
\[
d_{j+1} \circ (d_j)^{\SS_j'} \sim d_{j+1}' \circ (d_j')^{\SS_j}.
\]
But this follows from part~\eqref{it:commute} of Lemma~\ref{lem:switch} and relation~\eqref{it:isot}.

Assume now that~$\M'$ is obtained from~$\M$ via an isotopy of the gradient~$v$.
By Lemma~\ref{lem:unique-param}, the induced parameterized Cerf decompositions
$\CC$ and $\CC'$ are related by a sequence of isotopies of the framed attaching
spheres~$\SS_i$ and of the diffeomorphisms~$d_i$, and reversing framed $0$-spheres.
First suppose that~$\CC$ and~$\CC'$ are related by an isotopy of~$\SS_j$.
More precisely, let~$\varphi_t$ be an ambient isotopy of the framed attaching sphere~$\SS_j$.
Recall that $d_j' = d_j \circ (\varphi_1')^{-1}$, where $\varphi_1' = (\varphi_1)^{\SS_j}$, everything else remains the same.
By relation~\eqref{it:d-F},
\[
e_{M_j,\SS_j'} \circ \varphi_1 \sim \varphi_1' \circ e_{M_j,\SS_j}.
\]
However, $\varphi_1$ is isotopic to the identity, hence~$\varphi_1 \sim \Id_{F(M_j)}$. Using relation~\eqref{it:isot},
\[
d_j' \circ e_{M_j,\SS_j'} \sim d_j \circ (\varphi_1')^{-1} \circ e_{M_j,\SS_j'} \sim
d_j \circ e_{M_j,\SS_j},
\]
hence $f_\CC \sim f_{\CC'}$.
If~$\CC$ and~$\CC'$ are related by an isotopy of one of the diffeomorphisms~$d_j$,
then invariance follows from relation~\eqref{it:isot} of Definition~\ref{def:relations}.
The map is also unchanged by reversing a framed sphere by relation~\eqref{it:0-sphere}.

Now consider the case when $\M'$ is obtained from $\M$ by adding or removing a regular value.
Then~$\CC'$ is obtained from~$\CC$ by merging or splitting a product.
Without loss of generality, suppose we are merging the cylindrical~$W_j$ to~$W_{j+1}$.
The cases when~$W_{j+1}$ is cylindrical and when we are splitting a product are analogous.
Recall that $\SS_j' = d_j^{-1} \circ \SS_{j+1}$ and $d_j' = d_{j+1} \circ (d_j)^{\SS_j'}$.
Then
\[
d_j' \circ e_{M_j,\SS_j'} \sim d_{j+1} \circ (d_j)^{\SS_j'} \circ e_{M_j,\SS_j'}.
\]
According to relation~\eqref{it:d-F}, applied to $d_j \colon (M_j,\SS_j') \to (M_{j+1}, \SS_{j+1})$, we have
\[
(d_j)^{\SS_j'} \circ e_{M_j,\SS_j'} \sim e_{M_{j+1},\SS_{j+1}} \circ d_j.
\]
Hence, as $e_{M_j,\emptyset} \sim \Id_{M_j}$,
\[
d_j' \circ e_{M_j,\SS_j'} \sim d_{j+1} \circ e_{M_{j+1},\SS_{j+1}} \circ d_j \circ e_{M_j,\emptyset},
\]
and the result follows for merging a product.

Finally, suppose that $\M'$ is obtained from~$\M$ by a left-right equivalence.
In this case, $\CC$ and $\CC'$ are related by a diffeomorphism equivalence $\Phi \colon W \to W$.
Then, by the definition of $d_i'$,
\[
f_{\CC'} = \prod_{i=0}^{m} \left( \Phi_{i+1} \circ d_i \circ \left(\Phi_i^{\SS_i} \right)^{-1} \circ e_{M_i',\SS_i'} \right).
\]
If we apply relation~\eqref{it:d-F} to the diffeomorphism $\Phi_i \colon (M_i,\SS_i) \to (M_i',\SS_i')$,
we obtain that
\[
\left(\Phi_i^{\SS_i} \right)^{-1} \circ e_{M_i',\SS_i'} \sim e_{M_i,\SS_i} \circ (\Phi_i)^{-1}.
\]
Substituting this into the previous formula, and using the fact that $\Phi_0 \sim \Id_{M}$
and $\Phi_m \sim \Id_{M'}$, we obtain that $f_\CC \sim f_{\CC'}$.
This concludes the proof of Theorem~\ref{thm:presentation} in the case of $\Cob_n$.

For $\Cob_n'$ and $\Cob_n^0$, we apply the second paragraph of Theorem~\ref{thm:Cerf}.
In the case of~$\BSut'$, our objects are 3-manifolds with boundary, but the cobordisms
are products along the boundary, hence we only need to consider handles attached along the
interior, and the proof is completely analogous to the case of $\Cob_3'$.
\end{proof}

We now restate Theorem~\ref{thm:TQFT}, spelled out in more detail.

\begin{thm*}
Let $C$ be a category.
Suppose that we are given a functor
\[
F \colon \Man_n \to C,
\]
and for every oriented $n$-manifold~$M$ and framed sphere~$\SS \subset M$,
a morphism $F_{M,\SS} \colon F(M) \to F(M(\SS))$
that satisfy relations~\eqref{it:isot}--\eqref{it:0-sphere}:
\begin{enumerate}
\item \label{it:1}
  We have $F_{M,\emptyset} = \Id_{F(M)}$, and if $d \in \Diff_0(M)$, then $F(d) = \Id_{F(M)}$.
\item \label{it:2} Given an orientation preserving diffeomorphism $d \colon M \to M'$
  between $n$-manifolds and a framed sphere $\SS \subset M$,
  let $\SS' = d \circ \SS$, and let $d^\SS \colon M(\SS) \to M'(\SS')$ be the induced diffeomorphism. Then the following
  diagram is commutative:
  \[
    \xymatrix{
    F(M) \ar[r]^-{F_{M,\SS}} \ar[d]^-{F(d)} & F(M(\SS)) \ar[d]^{F(d^\SS)} \\
    F(M')  \ar[r]^-{F_{M',\SS'}} & F(M'(\SS')).}
  \]
\item \label{it:3} If $M$ is an oriented $n$-manifold and $\SS$ and $\SS'$
  are \emph{disjoint} framed spheres in~$M$, then the following diagram is commutative:
  \[
  \xymatrixcolsep{3pc}\xymatrix{
    F(M) \ar[r]^-{F_{M,\SS}} \ar[d]^-{F_{M,\SS'}} & F(M(\SS)) \ar[d]^{F_{M(\SS),\SS'}} \\
    F(M(\SS'))  \ar[r]^-{F_{M(\SS'),\SS}} & F(M(\SS,\SS')).}
  \]
\item \label{it:4} If $\SS' \subset M(\SS)$ and $a(\SS')$ intersects $b(\SS)$ once transversely,
  then there is a diffeomorphism $\varphi \colon M \to M(\SS)(\SS')$ (see Definition~\ref{def:phi}),
  for which
  \[
    F_{M(\SS),\SS'} \circ F_{M,\SS} = F(\varphi).
  \]
\item \label{it:5} $F_{M,\SS} = F_{M,\ol{\SS}}$.
\end{enumerate}
For a parameterized Cerf decomposition~$\CC$ of an oriented cobordism~$W$, let
\begin{equation} \label{eqn:comparison1}
F(W,\CC) = \prod_{i=0}^{m} \left( F(d_i) \circ F_{M_i,\SS_i} \right) \colon F(M) \to F(M').
\end{equation}
Then $F(W,\CC)$ is independent of the choice of~$\CC$; we denote it by~$F(W)$.
Furthermore, $F \colon \Cob_n \to C$ is a functor that satisfies $F(d) = F(c_d)$
(see Definition~\ref{def:diffcob}) and $F(W(\SS)) = F_{M,\SS}$.

In the opposite direction, every functor $F \colon \Cob_n \to C$ arises in this way.
More precisely, if we let $F_{M,\SS} = F(W(\SS))$ and $F(d) = F(c_d)$,
then these morphisms satisfy relations~\eqref{it:1}--\eqref{it:5},
and for any oriented cobordism~$W$, the morphism
$F(W)$ is given by equation~\eqref{eqn:comparison1}.

Now suppose that $(C,\otimes,I_C)$ is a symmetric monoidal category.
Then the functor~$F$ is a TQFT if and only if $F \colon \Man_n \to C$ is symmetric and monoidal;
furthermore, given $n$-manifolds~$M$ and~$N$, and a framed sphere~$\SS$ in~$M$, the diagram
\begin{equation} \label{eqn:monoidal1}
\xymatrixcolsep{3pc}\xymatrix{
  F(M) \otimes F(N) \ar[r]^-{\Phi_{M,N}} \ar[d]_-{F_{M,\SS} \otimes \text{Id}_{F(N)}} & F(M \sqcup N) \ar[d]^{F_{M \sqcup N, \SS}} \\
  F(M(\SS)) \otimes F(N) \ar[r]^-{\Phi_{M(\SS),N}} & F(M(\SS) \sqcup N).
  }
\end{equation}
is commutative, where $\Phi_{A,B} \colon F(A) \otimes F(B) \to F(A \sqcup B)$ are the comparison morphisms for~$F$.

An analogous result holds for $\Cob_n'$, and we can avoid $\SS = 0$ and framed $n$-spheres.
In the case of $\Cob_n^0$ for $n \ge 2$, we need to avoid $\SS = 0$ and $n$-spheres,
together with separating $(n-1)$-spheres. Finally, for~$\BSut'$, we have a similar result,
and we can avoid $\SS = 0$ and framed $3$-spheres.
\end{thm*}

\begin{proof}[Proof of Theorem~\ref{thm:TQFT}]
Suppose that $C$ is a category, $F \colon \Man_n \to C$ is a functor, and we are given morphisms~$F_{M,\SS}$
that satisfy the relations~\eqref{it:isot}--\eqref{it:0-sphere}
listed in Definition~\ref{def:relations}. Then~$F$ extends to a functor $F \colon \cF(\G_n)/\cR \to C$
such that $F(e_{M,\SS}) = F_{M,\SS}$ and $F(e_d) = F(d)$.
By Theorem~\ref{thm:presentation}, the map $c \colon \cF(\G_n)/\cR \to \Cob_n$
is an isomorphism of categories, and $F \circ c^{-1} \colon \Cob_n \to C$ is the desired functor.

We now show that if~$W$ is an oriented cobordism from~$M$ to~$M'$, then $F \circ c^{-1}([W])$
is given by equation~\eqref{eqn:comparison1}, where $[W]$ is the equivalence class of~$W$.
Choose a parameterized Cerf
decomposition~$\CC$ of~$W$, consisting of a decomposition
\[
W = W_0 \cup_{M_1} W_1 \cup_{M_2} \dots \cup_{M_m} W_m,
\]
together with framed attaching spheres~$\SS_i$ and diffeomorphisms~$d_i \colon M_i(\SS_i) \to M_{i+1}$.
When $n \ge 2$ and~$W$, $M$, and~$M'$ are all connected, we can assume that each~$M_i$ is connected
as well by~\cite[Lemma~2.5]{GWW}.
Then
\[
c^{-1}([W]) = f_\CC = \prod_{i=0}^{m} \left( d_i \circ e_{M_i,\SS_i} \right) \colon M \to M',
\]
and so $F \circ c^{-1}([W]) = F(W,\CC)$, as required.


In the opposite direction, if we are given a functor $F \colon \Cob_n \to C$, then
\[
F \circ c \colon \cF(\G_n)/\cR \to C
\]
is also a functor.
Hence, if we let $F(h) = F(c_h)$ for a diffeomorphism $h \colon M \to M'$,
and, given a framed sphere~$\SS$ in~$M$, we define
$F_{M,\SS} \colon F(M) \to F(M(\SS))$ to be $F(W(\SS))$, then these maps satisfy
relations~\eqref{it:isot}--\eqref{it:0-sphere}. The correspondence is one-to-one by
the following.

\begin{lem} \label{lem:diffeo-maps}
Suppose that $F$ arises from a functor $F \colon \Man_n \to C$ and surgery morphisms~$F_{M,\SS}$ as in Theorem~\ref{thm:TQFT}.
Then, for any diffeomorphism $h \colon M \to M'$, we have
\[
F(h) = h_*.
\]
\end{lem}

\begin{proof}
Recall that $h_*$ is defined as $F(c_h)$, where~$c_h$ is the cylindrical cobordism
\[
(M \times I; M \times \{0\}, M \times \{1\}; p_0, h_1).
\]
Then this is in itself a
parameterized Cerf decomposition $\CC$ of a single level,
and so $F(c_h, \CC) = F(h) \circ F_{M,\emptyset} = F(h)$.
\end{proof}

If $(C,\otimes,I_C)$ is a symmetric monoidal category and $F \colon \Man_n \to C$ is a symmetric monoidal functor,
then the extension $F \colon \Cob_n \to C$ automatically satisfies all properties of a TQFT
listed in Definition~\ref{def:tqft} (i.e., it is symmetric and monoidal) as these properties do not involve cobordisms,
except we need to check that the comparison morphisms
$\Phi_{M,N} \colon F(M) \otimes F(N) \to F(M \sqcup N)$ are natural. This follows from
the commutativity of diagram~\eqref{eqn:monoidal1}.
Indeed, naturality of the comparison morphisms amounts to the commutativity of the diagram
\begin{equation} \label{eqn:coh}
\xymatrix{
  F(M) \otimes F(N) \ar[r]^-{\Phi_{M,N}} \ar[d]_-{F(V) \otimes F(W)} & F(M \sqcup N) \ar[d]^{F(V \sqcup W)} \\
  F(M') \otimes F(N') \ar[r]^-{\Phi_{M',N'}} & F(M' \sqcup N'),
  }
\end{equation}
where $V$ is an oriented cobordism from~$M$ to~$M'$ and $W$ is an oriented cobordism from~$N$ to~$N'$.
It suffices to show this when either $V$ or $W$ is a trivial cobordism, according to the following diagram:
\[\xymatrix{
  F(M) \otimes F(N) \ar[r]^-{\Phi_{M,N}} \ar[d]_-{F(V) \otimes F(i_N)} & F(M \sqcup N) \ar[d]^{F(V \sqcup i_N)} \\
  F(M') \otimes F(N) \ar[r]^-{\Phi_{M',N}} \ar[d]_-{F(i_{M'}) \otimes F(W)} & F(M' \sqcup N) \ar[d]^-{F(i_{M'} \sqcup W)} \\
  F(M') \otimes F(N') \ar[r]^-{\Phi_{M',N'}} & F(M' \sqcup N').
  }
\]
By the symmetry of~$F$, if diagram~\eqref{eqn:coh} commutes when $W$ is trivial, then it also commutes whenever $V$ is trivial.
So, it suffices to show that the diagram
\begin{equation} \label{eqn:onesided}
\xymatrix{
  F(M) \otimes F(N) \ar[r]^-{\Phi_{M,N}} \ar[d]_-{F(V) \otimes \Id_{F(N)}} & F(M \sqcup N) \ar[d]^{F(V \sqcup i_N)} \\
  F(M') \otimes F(N) \ar[r]^-{\Phi_{M',N}} & F(M' \sqcup N)
  }
\end{equation}
is commutative for any oriented cobordism~$V$ from~$M$ to~$M'$. Let $\CC$ be a parameterized Cerf decomposition of~$V$. Then
\[
F(V) = F(V,\CC) = \prod_{i=0}^{m} \left( F(d_i) \circ F_{M_i,\SS_i} \right) \colon F(M) \to F(M').
\]
Since $F \colon \Man_n \to C$ is monoidal,
the comparison morphisms are natural with respect to diffeomorphisms, hence the diagram
\[
\xymatrixcolsep{4pc}\xymatrix{
  F(M_i(\SS_i)) \otimes F(N) \ar[r]^-{\Phi_{M_i(\SS_i),N}} \ar[d]_-{F(d_i) \otimes \Id_{F(N)}} &
  F(M_i(\SS_i) \sqcup N) \ar[d]^{F(d_i \sqcup \Id_N)} \\
  F(M_{i+1}) \otimes F(N) \ar[r]^-{\Phi_{M_{i+1},N}} & F(M_{i+1} \sqcup N)
  }
\]
is commutative. Together with the commutativity of diagram~\eqref{eqn:monoidal1}
for the framed spheres~$\SS_i$ in~$M_i$, we obtain that
diagram~\eqref{eqn:onesided} is also commutative.
Hence, we see that if $F \colon \Man_n \to C$ is symmetric and monoidal and
diagram~\eqref{eqn:monoidal1} commutes, then the extension $F \colon \Cob_n \to C$ is
also symmetric and monoidal; i.e., a TQFT.
This concludes the proof of Theorem~\ref{thm:TQFT} in case of the category~$\Cob_n$.

The results for $\Cob_n'$, $\Cob_n^0$, and $\BSut'$ follow from the respective parts
of Theorem~\ref{thm:presentation} analogously.
\end{proof}

\section{Classifying (1+1)-dimensional TQFTs} \label{sec:1+1}

Recall that a \emph{Frobenius algebra} is a finite-dimensional unital associative $\F$-algebra $A$ with
multiplication~$\mu \colon A \otimes A \to A$ and a trace functional
$\theta \colon A \to \F$ such that $\ker(\theta)$ contains no non-zero left ideal of~$A$.
Then $\sigma(a,b) = \theta(ab)$ is a non-degenerate bilinear form.
In particular, $\sigma$ sets up an isomorphism between~$A$ and~$A^*$. Dualizing the algebra structure,
we also get a coalgebra structure on~$A$ with counit; we denote the coproduct by
$\delta \colon A \to A \otimes A$. Note that~$\delta$ is obtained by
dualizing the product $A \otimes A \to A$, and using the fact that $(A \otimes A)^* \approx A^* \otimes A^*$
since~$A$ is finite-dimensional. The Frobenius algebra~$A$ is called \emph{commutative}
if the product~$\mu$ is commutative and the coproduct~$\delta$ is cocommutative.

In this section, we give a short proof of the following classical result on the classification of (1+1)-dimensional TQFTs
using Theorem~\ref{thm:TQFT}; see~\cite{Kock}.
This can be viewed as a warm-up for the following section, where we classify (2+1)-dimensional TQFTs.
Here all $1$-manifolds and cobordisms are assumed to be oriented.

\begin{thm} \label{thm:1+1}
There is an equivalence between the category of (1+1)-dimensional TQFTs and
the category of finite-dimensional commutative Frobenius algebras.
\end{thm}

\begin{proof}
It is straightforward to see that a (1+1)-dimensional TQFT
\[
F \colon \Cob_2 \to \Vect_\F
\]
gives rise to a Frobenius algebra. Indeed, let $A := F(S^1)$. If $S$ is a pair-of-pants
cobordism from $S^1 \sqcup S^1$ to $S^1$, then the multiplication is given by
\[
F(S) \colon F(S^1 \sqcup S^1) \cong F(S^1) \otimes F(S^1) = A \otimes A \to F(S^1) = A,
\]
where the first map is the natural isomorphism coming from the monoidal structure of~$F$.
If~$\D$ denotes the cobordism from~$S^1$ to~$\emptyset$ given by a disk, then $\theta := F(\D)$.
We can also view the disk as a cobordism from~$\emptyset$ to~$S^1$ which we denote by~$\ol{\D}$ .
Then $F\left(\ol{\D}\right)(1) \in A$ is the unit.
It is now straightforward to check that these form a Frobenius algebra.
Commutativity follows from the symmetry of~$F$.

The non-trivial direction is associating a TQFT to a Frobenius algebra.
Given a Frobenius algebra~$A$, we describe the ingredients of
Theorem~\ref{thm:TQFT} needed to define a TQFT, namely, a symmetric monoidal functor $F \colon \Man_1 \to \Vect_\F$
and maps induced by framed spheres that satisfy the required relations.

Throughout this paper, for oriented manifolds~$X$, $Y$, we denote by~$\Diff(X,Y)$ the set of \emph{orientation
preserving} diffeomorphisms from~$X$ to~$Y$, and we write $\Diff(X) := \Diff(X,X)$. Furthermore,
\[
\MCG(X) = \Diff(X)/\Diff_0(X)
\]
is the \emph{oriented} mapping class group of~$X$. The group $\Diff(Y)$ acts on $\Diff(X,Y)$ by composition.
By slight abuse of notation, we write
\[
\MCG(X,Y) := \Diff(X,Y)/\Diff_0(Y),
\]
even though this is not actually a group, only an affine copy of~$\MCG(X)$ if~$X$ and~$Y$ are diffeomorphic,
and the empty set otherwise.

Let $C_k = S^1 \times \{1, \dots, k\}$; i.e., the disjoint union of~$k$ copies of~$S^1$.
Given a closed $1$-manifold~$M$ of~$k$ components, note that $MCG(C_k, M)$
is an affine copy of the symmetric group~$S_k$. An element of $\MCG(C_k, M)$ can be thought of as a labeling of
the components of~$M$ by the integers~$1, \dots, k$. Given mapping classes $\phi$, $\phi' \in \MCG(C_k, M)$,
their difference $(\phi')^{-1} \circ \phi$ is an element $\sigma(\phi,\phi')$ of $\MCG(C_k, C_k)$,
which is canonically isomorphic to~$S_k$.

For a closed $1$-manifold~$M$, let~$F(M)$ be the set of those elements~$a$ of
\[
\prod_{\phi \in \MCG(C_k, M)} A^{\otimes k}
\]
such that for any $\phi$, $\phi' \in \MCG(C_k, M)$ the coordinates $a(\phi)$ and $a(\phi')$ in $A^{\otimes k}$
differ by the permutation of factors given by $\sigma(\phi,\phi') \in S_k$.
Notice that the function~$a$ is uniquely determined by its value~$a(\phi)$ for any $\phi \in \MCG(C_k, M)$;
i.e., for any labeling of the components of~$M$ by the numbers~$1, \dots, k$.
Note that this construction is an instance of a Kan extension.

Suppose that $M$ and~$M'$ are diffeomorphic $1$-manifolds; i.e., they have the same number
of components~$k$, and let $d \in \Diff(M,M')$. Given an element $a \in F(M)$
and $\phi \in \MCG(C_k, M)$, we define
\[
(F(d)(a))([d] \circ \phi) = a(\phi),
\]
where $[d] \in \MCG(M,M')$ is the isotopy class of~$d$.

If $M$ and $N$ are 1-manifolds of $k$ and $l$ components, respectively, then
we define the natural isomorphism $\Phi_{M,N} \colon F(M) \otimes F(N) \to F(M \sqcup N)$ as follows.
Let $\phi \in \MCG(C_k,M)$ and $\psi \in \MCG(C_l,N)$. The mapping class
\[
\phi \sqcup \psi \in \MCG(C_{k+l},M \sqcup N)
\]
is defined to be $\phi$ on $S^1 \times \{1,\dots,k\}$, and on $S^1 \times \{k+1,\dots,k+l\}$ it
maps $(x,k+i)$ to $\psi(x,i)$.
If $a \in F(M)$ and $b \in F(N)$, then we let $\Phi_{M,N}(a \otimes b) = a \sqcup b \in F(M \sqcup N)$,
where $(a \sqcup b)(\phi \sqcup \psi) = a(\phi) \otimes b(\psi) \in A^{\otimes (k+l)}$.
We leave it to the reader to check that the $F \colon \Man_1 \to \Vect_\F$
defined above is a symmetric monoidal functor.

We now define the surgery maps.
A framed $0$-sphere in a closed $1$-manifold~$M$ of~$k$ components is given by an embedding
\[
\SS \colon S^0 \times D^1 = \{-1,1\} \times [-1,1] \hookrightarrow M.
\]
Since we only consider oriented cobordisms, the framing should be orientation reversing, and is hence unique
up to isotopy. So $\SS$ is completely determined by a pair of points $\SS = \{s_-,s_+\}$.

If~$s_-$ and $s_+$ lie in different components~$M_-$ and~$M_+$ of~$M$, respectively,
then we define the map
\[
F_{M,\SS} \colon F(M) \to F(M(\SS))
\]
as follows.
Let $a \in F(M)$, and let $\phi \in \MCG(C_k, M)$ correspond to a labeling of the components of~$M$
such that~$M_-$ is labeled~$k-1$ and~$M_+$ is labeled~$k$. This gives rise to a labeling~$\phi_\SS$
of the components of~$M(\SS)$, where the component arising from surgery on~$M_-$ and~$M_+$ is labeled~$k-1$,
while every other component is unchanged and retains its label. Then~$F_{M,\SS}(a)$ is the element of~$F(M(\SS))$
for which $F_{M,\SS}(a)(\phi_\SS)$ is the image of $a(\phi)$ under the map
\[
A^{\otimes (k-2)} \otimes  A \otimes A \to A^{\otimes (k-2)} \otimes A
\]
that multiplies the last two factors using the algebra product of~$A$; i.e.,
takes $a_1 \otimes \dots \otimes a_{k-2} \otimes a_{k-1} \otimes a_k$ to $a_1 \otimes \dots \otimes a_{k-2} \otimes (a_{k-1} a_k)$.
It is straightforward to see that the above definition of~$F_{M,\SS}(a)$ is independent of the choice of~$\phi$.
Indeed, if~$\phi'$ is another labeling such that~$M_-$ is labeled~$k-1$ and~$M_+$ is labeled~$k$,
then~$F_{M,\SS}(a)(\phi_\SS)$ and~$F_{M,\SS}(a)(\phi'_\SS)$ differ by the action of the permutation~$\sigma(\phi_\SS,\phi'_\SS)$
that fixes~$k-1$, and maps to $\sigma(\phi,\phi')$ under the embedding $S_{k-1} \to S_k$.
So, by definition, these two elements of $A^{\otimes (k-1)}$ define the same element~$F_{M,\SS}(a)$ of~$F(M_\SS)$.

Now suppose that~$s_-$ and~$s_+$ lie in the same component~$M_s$ of~$M$. Then~$M_\SS$ has~$k+1$ components.
The component~$M_s$ splits into a component~$M_-$ corresponding to the arc of~$M_s \setminus \SS$ going from~$s_-$ to~$s_+$,
and a component~$M_+$ corresponding to the arc of~$M_s \setminus \SS$ going from~$s_+$ to~$s_-$.
Let~$\phi$ be a labeling of the components of~$M$ such that~$M_s$ is labeled~$k$. Then we denote by~$\phi_\SS$
the labeling of the components of~$M_\SS$ where each component of~$M \setminus M_s$ retains its label,
$M_-$ is labeled~$k$, and~$M_+$ is labeled~$k+1$. Given $a \in F(M)$, we define~$F_{M,\SS}(a)(\phi_\SS) \in A^{\otimes (k+1)}$
by applying to~$a(\phi) \in A^{\otimes k}$ the map $A^{\otimes k} \to A^{\otimes (k+1)}$
that sends $a_1 \otimes \dots \otimes a_{k-1} \otimes a_k$ to $a_1 \otimes \dots \otimes a_{k-1} \otimes \delta(a_k)$,
where~$\delta$ is the coproduct of the Frobenius algebra~$A$. As in the previous case, $F_{M,\SS}(a)$ is independent of the
choice of~$\phi$.

Surgery along the framed attaching sphere of a $0$-handle results in the manifold~$M(0) = M \sqcup S^1$.
Chose an arbitrary labeling~$\phi$ of the components of~$M$ with the numbers $1, \dots, k$. We obtain the labeling $\phi_0$
of the components of~$M(0)$ by labeling the new~$S^1$ component~$k+1$.
Let $\iota_k \colon A^{\otimes k} \to A^{\otimes (k+1)}$ be the map $\iota_k(x) = x \otimes 1$,
where~$1$ is the unit of~$A$.
For~$a \in F(M)$, we define $F_{M,0}(a)(\phi_0) = \iota_k(a(\phi))$;
the map~$F_{M,0}$ is independent of the choice of~$\phi$.

Finally, a framed $1$-sphere in a $1$-manifold~$M$ of~$k$ components is simply an embedding $\SS \colon S^1 \hookrightarrow M$.
Let~$S$ be the image of~$\SS$, then $M(\SS) = M \setminus S$. Let~$\phi$ be a labeling of the
components of~$M$ such that~$S$ is given the label~$k$, and let $\phi_\SS$ be the corresponding labeling of~$M(\SS)$.
Let $t_k \colon A^{\otimes k} \to A^{\otimes (k-1)}$ be the map given by extending linearly
\[
t_k(a_1 \otimes \dots \otimes  a_{k-1} \otimes a_k) = \theta(a_k) \cdot a_1 \otimes \dots \otimes a_{k-1}.
\]
For $a \in F(M)$, let $F_{M,\SS}(a)(\phi_\SS) = t_k(a(\phi))$. Again, this gives a well-defined map~$F_{M,\SS}$ independent of the choice
of labeling~$\phi$.

Now all we need to check is that relations~\eqref{it:isot}--\eqref{it:0-sphere} of Definition~\ref{def:relations} hold
and diagram~\eqref{eqn:monoidal} commutes for the data defined above.
We only give an outline here and leave the details to the reader.
Axiom~$\eqref{it:isot}$ is straightforward, as if $d \in \Diff_0(M)$, then $[d] \circ \phi = \phi \in \MCG(M)$,
and $(F(d)(a))(\phi) = (F(d)(a))([d] \circ \phi) = a(\phi)$; i.e., $F(d) = \Id_{F(M)}$.

Now consider relation~\eqref{it:d-F}, naturality. We check this in the case where~$\SS = \{s_-,s_+\}$ is a framed $0$-sphere
with~$s_-$ and~$s_+$ lying in different components~$M_-$ and~$M_+$ of~$M$, respectively; the other cases are similar.
Choose a labeling~$\phi$ of the components of~$M$ such that~$M_-$ is labeled~$k-1$ and~$M_+$ is labeled~$k$.
For $a_1, \dots, a_k \in A$, let $a$ be the element of~$F(M)$ for which $a(\phi) = a_1 \otimes \dots \otimes a_k$.
Then, by definition,
\[
F_{M,\SS}(a)(\phi_\SS) = a_1 \otimes \dots \otimes a_{k-2} \otimes (a_{k-1}a_k).
\]
Given a diffeomorphism $d \colon M \to M'$, this induces a labeling~$[d] \circ \phi$ of~$M'$.
Then $(F(d)(a))([d] \circ \phi) = a(\phi) = a_1 \otimes \dots \otimes a_k$. Consider~$\SS' = \{d(s_-), d(s_+)\}$.
Under~$[d] \circ \phi$, the component~$M_-'$ of~$M'$ containing~$d(s_-)$ is labeled~$k-1$ and the component $M_+'$ containing~$d(s_+)$
is labeled~$k$. Hence, we can use the labeling~$[d] \circ \phi$ of~$M'$ to compute the map~$F_{M',\SS'}$. This induces the
labeling $([d] \circ \phi)_{\SS'}$, where the component obtained by taking
the connected sum of~$M_-'$ and~$M_+'$ is labeled~$k-1$ and every
other component retains its label. With this notation in place,
\[
\left[F_{M',\SS'} \circ F(d)(a) \right] \left(([d] \circ \phi)_{\SS'}\right) = a_1 \otimes \dots \otimes a_{k-2} \otimes (a_{k-1}a_k).
\]
The diffeomorphism~$d^\SS$ maps $M_- \# M_+$ to $M_-' \# M_+'$, and on the other components it acts just like~$d$.
It follows that $[d^{\SS}] \circ \phi_\SS = ([d] \circ \phi)_{\SS'}$. Furthermore,
\[
\left[F(d^\SS) \circ F_{M,\SS}(a) \right]\left([d^\SS] \circ \phi_\SS \right) =
F_{M,\SS}(a)(\phi_\SS) = a_1 \otimes \dots \otimes a_{k-2} \otimes (a_{k-1}a_k).
\]
This establishes the commutativity of the diagram in relation~\eqref{it:d-F}.

Now consider relation~\eqref{it:commut}, commutativity; i.e., that
\begin{equation} \label{eqn:SS}
F_{M(\SS),\SS'} \circ F_{M,\SS} = F_{M(\SS'),\SS} \circ F_{M,\SS'}.
\end{equation}
Here we have several cases depending on the dimensions of
the attaching spheres. This is obviously true when $\SS = \SS' = 0$.
When~$\SS$ and~$\SS'$ are framed $1$-spheres glued along distinct
components~$S$ and~$S'$ of~$M$, then let~$\phi$ be a labeling of~$M$
such that $S$ is labeled~$k$ and $S'$ is labeled~$k-1$.
As above, let~$a \in F(M)$ be such that $a(\phi) = a_1 \otimes \dots \otimes a_k$.
Then
\[
\left[F_{M(\SS),\SS'} \circ F_{M,\SS}(a)\right](\phi_{\SS,\SS'}) = \theta(a_{k-1})\theta(a_k) \cdot a_1 \otimes \dots \otimes a_{k-2}.
\]
On the other hand, let $\phi'$ be the labeling of the components of~$M$ where~$S$ is labeled~$k-1$ and~$S'$ is labeled~$k$,
otherwise it agrees with~$\phi$. The permutation $\sigma(\phi,\phi') \in S_k$ is the transposition of~$k-1$ and~$k$, and so
\[
a(\phi') = a_1 \otimes \dots \otimes a_{k-2} \otimes a_k \otimes a_{k-1}.
\]
It follows that
\[
\left[F_{M(\SS'),\SS} \circ F_{M,\SS'}(a)\right](\phi'_{\SS',\SS}) = \theta(a_k)\theta(a_{k-1}) \cdot a_1 \otimes \dots \otimes a_{k-2}.
\]
Since $\phi_{\SS,\SS'} = \phi'_{\SS',\SS}$, the result follows from the commutativity of~$\F$ in this case.

When $\SS' = 0$ and $\SS$ is a $1$-sphere in a component~$S$ of~$M$,
then choose a labeling~$\phi$ such that~$S$ is labeled~$k$.
Then
\[
\left[F_{M(\SS),0} \circ F_{M,\SS}(a)\right](\phi_{\SS,0}) = \theta(a_k) \cdot a_1 \otimes \dots \otimes a_{k-1} \otimes 1,
\]
where $\phi_{\SS,0}$ labels the components of $M \setminus S$ just like~$\phi$, and the new $S^1$-component is labeled~$k$.
To compute $F_{M(0),\SS} \circ F_{M,0}(a)$, first note that
\[
F_{M,0}(a)(\phi_0) = a_1 \otimes \dots \otimes a_k \otimes 1.
\]
If~$\tau$ is the transposition of~$k$ and~$k+1$, then
\[
F_{M,0}(a)(\tau \circ \phi_0) = a_1 \otimes \dots \otimes a_{k-1} \otimes 1 \otimes a_k.
\]
As $\tau \circ \phi_0$ labels~$S$ with~$k+1$,
\[
\left[ F_{M(0),\SS} \circ F_{M,0}(a) \right] \left( (\tau \circ \phi_0)_\SS \right) = \theta(a_k) \cdot a_1 \otimes \dots \otimes a_{k-1} \otimes 1,
\]
and $(\tau \circ \phi_0)_\SS = \phi_{\SS,0}$, which proves equation~\eqref{eqn:SS} in this case.

Now suppose that~$\SS = \{s_-,s_+\}$ is a framed $0$-sphere in~$M$. The cases when $\SS' = 0$ or when~$\SS'$ is a $1$-sphere
disjoint from~$\SS$ are similar to the previous one. When $\SS' = \{s_-',s_+'\}$ is also a $0$-sphere, we have four
cases depending on whether~$\SS \cup \SS'$ intersects~$M$ in $c = 1,2,3$, or~$4$ components.
The case~$c = 1$ splits into two subcases depending on whether~$\SS$ and~$\SS'$ are linked. When they
are linked, both sides of equation~\eqref{eqn:SS}
will be of the form $a_1 \otimes \dots \otimes a_{k-1} \otimes (\mu \circ \delta(a_k))$,
where~$\mu$ is the product and~$\delta$ is the coproduct of~$A$. When~$\SS$ and~$\SS'$ are unlinked,
then one side becomes
\[
a_1 \otimes \dots \otimes a_{k-1} \otimes (\delta \otimes \Id_A)(\delta(a_k)),
\]
while the other side is
\[
a_1 \otimes \dots \otimes a_{k-1} \otimes (\Id_A \otimes \delta)(\delta(a_k)).
\]
The two coincide by the coassociativity of the coalgebra $(A,\delta)$.
When $c = 2$ and one of $\SS$ and $\SS'$ lies in a single component~$M_s$ of~$M$, while the other one intersects~$M_s$ in one point,
then the equality boils down to the fact that~$\delta$ is a left and right $A$-module homomorphism; i.e.,
\[
(\mu \otimes \Id_A)(a_{k-1} \otimes \delta(a_k)) = (\delta \circ \mu)(a_{k-1} \otimes a_k) = (\Id_A \otimes \mu)(\delta(a_{k-1}) \otimes a_k).
\]
If $c = 2$ and $\SS$, $\SS'$ both intersect the same two components of~$M$, then both sides of equation~\eqref{eqn:SS} become
$a_1 \otimes \dots \otimes a_{k-2} \otimes (\d \circ \mu(a_{k-1},a_k))$.
When $c = 2$ and $\SS$ and $\SS'$ lie in two distinct components of~$M$, then the result is clear as we have two coproduct maps
acting on distinct components of~$M$.
When $c = 3$ and~$\SS$ and~$\SS'$ share a component, then the result follows from the associativity of the algebra~$(A,\mu)$.
When $c = 3$ and~$\SS$ occupies two components and $\SS'$ a third, then we have a non-interacting product and coproduct.
The case~$c = 4$ is also straightforward as we are dealing with two non-interacting product maps.

We now check relation~\eqref{it:birth}, cancelation.
When $\SS = 0$ and $\SS' \subset M(0)$ is a $0$-sphere that intersects the new
$S^1$ component in one point, then the result follows from the fact that~$1$ is a left and right unit of~$A$.
Now suppose that $\SS$ is a $0$-sphere and $\SS' \subset M(\SS)$ is a $1$-sphere that intersects $b(\SS)$ in one point.
Then~$\SS$ has to occupy a single component of~$M$ that splits into the components~$M_-$ and~$M_+$
when we perform surgery along~$\SS$, and $\SS'$ maps to either~$M_-$ or~$M_+$. The result follows from
the fact that~$\theta$ is a left and right counit of the the coalgebra $(A,\delta)$; i.e., that
\[
(\theta \otimes \Id_A) \circ \delta = \Id_A = (\Id_A \otimes \theta) \circ \delta.
\]

Consider relation~\eqref{it:0-sphere}. If $\SS = \{s_-,s_+\}$ and~$s_-$ and~$s_+$ lie in different components of~$M$,
then $F_{M,\SS}(a)(\phi) = a_1 \otimes \dots \otimes a_{k-2} \otimes a_{k-1}a_k$. In $\ol{\SS}$ we
reverse~$s_-$ and~$s_+$, and so
$F_{M,\ol{\SS}}(a)(\phi) = a_1 \otimes \dots \otimes a_{k-2} \otimes a_k a_{k-1}$.
These coincide as the Frobenius algebra is commutative. When~$s_-$ and~$s_+$ occupy the same component of~$M$,
then $F_{M,\SS} = F_{M,\ol{\SS}}$  follows from cocommutativity.

Finally, the commutativity of diagram~\eqref{eqn:monoidal} follows automatically from the construction
of~$F$ and the surgery maps and does not impose any additional restrictions.

As explained by Kock~\cite[p.~173]{Kock}, given a morphism from the TQFT~$F$ to the TQFT~$G$; i.e.,
a natural transformation $\eta \colon F \Rightarrow G$, the map $\eta_{S^1} \colon F(S^1) \to G(S^1)$
is a homomorphism of Frobenius algebras. Conversely, given commutative Frobenius algebras $A$ and $B$
and a homomorphism $h \colon A \to B$, we can extend this to a natural transformation~$\eta$ between the
corresponding TQFTs~$F$ and~$G$. Indeed, given a 1-manifold~$M$ of $k$ components and $a \in F(M)$,
choose a mapping class $\phi \in \MCG(C_k, M)$. Then we let $\eta_M(a)(\phi) = h^{\otimes k}(a(\phi)) \in B^{\otimes k}$,
where $h^{\otimes k} \colon A^{\otimes k} \to B^{\otimes k}$. The naturality of $\eta$ for diffeomorphisms
and surgery maps follows from the fact that $h$ is a homomorphism of Frobenius algebras, and
naturality for arbitrary cobordisms then follows via equation~\eqref{eqn:comparison} that defines
the cobordism maps.

The two functors we defined between the category of (1+1)-dimensional TQFTs and the category of commutative
Frobenius algebras are inverses of each other up to natural isomorphism, hence they are equivalences
between the two categories.
This concludes the proof of Theorem~\ref{thm:1+1}.
\end{proof}

\section{The algebra of (2+1)-dimensional TQFTs} \label{sec:algebra}

In this section, we apply Theorem~\ref{thm:TQFT} to the study of (2+1)-dimensional TQFTs.
Note that Kontsevich~\cite{Kontsevich} outlined a correspondence between (1+1+1)-dimensional TQFTs and modular functors.
As to be expected, the full (2+1)-dimensional classification leads to an algebraic structure more complicated
than in the (1+1)-dimensional and (1+1+1)-dimensional cases; cf.~Proposition~\ref{prop:nonextending}.
The additional difficulty comes
from the fact that the mapping class groups of connected 2-mani\-folds are non-trivial, unlike for connected 1-mani\-folds.
However, we can make considerable simplifications, leading to a structure just barely more involved than commutative Frobenius
algebras. We expect that the algebra presented below can be further simplified; this is the aim of future research.

\subsection{Canonical surfaces and framed spheres} \label{sec:surfaces}
For every $g \ge 0$, let~$\S_g$ be a fixed oriented surface of genus~$g$ obtained as the connected sum $\#^g (S^1 \times S^1)$,
where $S^1 = \{\, z \in \C \,\colon\, |z| = 1 \,\}$,
and let $\M_g = \MCG(\S_g)$. The connected sums are taken at the point~$(1,1)$ of component~$i$ and
the point $(-1,1)$ of component~$(i+1)$.

Let $l_i = (S^1 \times \{-1\})_i$ be a longitude of summand~$i$, while $m_0 = (\{-1\} \times S^1)_1$ is a meridian
of the first summand, and $m_g = (\{1\} \times S^1)_g$ is a meridian of the last summand. Furthermore, for $i \in \{1,\dots, g-1\}$,
consider the curves
\[
m_i = (\{1\} \times S^1)_i \# (\{-1\} \times S^1)_{i+1}.
\]
If $j \in \{1, \dots, g\}$, we write
\[
s_j = \{\, (\exp(\varepsilon \cos \theta \sqrt{-1}),\exp(\varepsilon \sin \theta \sqrt{-1})) \,\colon\, \theta \in S^1 \,\}
\subset (S^1 \times S^1)_j;
\]
this is the connected sum curve between the $j$-th and $(j+1)$-st $S^1 \times S^1$ summands for $j < g$.
Furthermore, $s_g$ is an inessential curve in the last summand $(S^1 \times S^1)_g$.
Finally, let~$s_0$ be an inessential curve in the first summand oriented from the left.
Each $s_j$ is oriented as the boundary of the $j$-th $S^1 \times S^1$ summand; i.e., as the boundary
of the component of $\S_g \setminus s_j$ with smaller $x$-coordinates.
All the above curves are naturally parameterized by~$S^1$, and if we fix a thin regular neighborhood
of each, we can and will view them as framed spheres $S^1 \times D^1 \hookrightarrow \S_g$.
For an illustration when $g = 4$, see Figure~\ref{fig:surface}.
\begin{figure}
\includegraphics{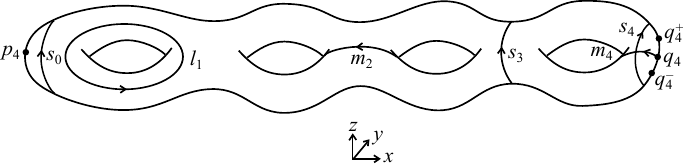}
\caption{The curves $m_i$, $l_i$, and~$s_i$, and the points
$p_4$, $q_4$, and~$p_4^\pm$ on the standard surface~$\S_4$ of genus four.}
\label{fig:surface}
\end{figure}

Let $p_g = (-1,1)_1$ and $q_g = (1,1)_g$ be points
on the first and last $S^1 \times S^1$ summands of~$\S_g$, respectively.
These have neighborhoods parameterized by $D^2$, such that restricting these
to $S^1$, we obtain the fixed parameterizations of $-s_0$ and~$-s_g$.
For $i$, $j \in \Z_{\ge 0}$, let
\[
\PP_{i,j} = \{q_i, p_j\} \in \S_i \sqcup \S_j.
\]
This is a framed $0$-sphere with the framing
$S^0 \times D^2 \hookrightarrow \S_i \sqcup \S_j$ given by the parameterizations of the disks containing $q_i$ and $p_j$.
Furthermore, for every $g \in \Z_{\ge 0}$, let $\PP_g = \{q_g^-,q_g^+\}$ be the framed sphere
given by two points very close to~$q_g$, both lying on~$(S^1 \times \{1\})_g$,
with framing obtained by translating and shrinking the normal framing of~$q_g$,
and also reflecting it in case of $q_g^+$.

From now on, we will use the following isotopically unique identifications:
The diffeomorphism $\S_g(l_g) \approx \S_{g-1}$ for $g > 0$
maps the disk obtained by performing surgery along~$l_g$ on the component of $\S_g \setminus s_{g-1}$
containing $l_g$ to the disk bounded by~$s_{g-1}$ in~$\S_{g-1}$, fixing~$s_{g-1}$ pointwise.
Furthermore, it maps the other component of $\S_g \setminus s_{g-1}$ to $\S_{g-1} \setminus s_{g-1}$ isometrically.
The diffeomorphism $\S_g(\bP_g) \approx \S_{g+1}$ maps the $D^1 \times S^1$ glued during the surgery along~$\bP_g$
to the neighborhood of $m_{g+1}$ in $\S_{g+1}$ given by its framing. Furthermore, it is the identity on
the component of $\S_g \setminus s_g$ disjoint from~$\bP_g$, and maps the result of surgery along $\bP_g$ on the disk
component of $\S_g \setminus s_g$ to the component of $\S_{g+1} \setminus s_g$ containing $m_{g+1}$ isometrically.
The identification $(\S_i \sqcup \S_j)(\bP_{i,j}) \approx \S_{i+j}$ maps the $D^1 \times S^1$ glued
during the surgery along~$\bP_{i,j}$ to the neighborhood of the circle~$s_i$ in~$\S_{i+j}$ given by its framing.
Furthermore, it maps $\S_i \setminus N(q_i)$ and $\S_j \setminus N(p_j)$ to the respective
components of $\S_{i+j} \setminus N(s_i)$ isometrically.
Finally, the diffeomorphism $\S_{i+j}(s_i) \approx \S_i \sqcup \S_j$ maps the components
of $\S_{i+j} \setminus N(s_i)$ to $(\S_i \setminus N(q_i)) \sqcup (\S_j \setminus N(p_j))$ isometrically,
and the $D^2 \times S^0$ introduced during the surgery along~$s_i$ to
$N(q_i) \subset \S_i$ and $N(p_j) \subset \S_j$.

\subsection{Assigning a J-algebra to a TQFT} \label{sec:assignment}
Suppose that the functor
\[
F \colon \Cob_2 \to \Vect_\F
\]
is a TQFT.
We associate to it a tuple
\[
J(F) = (\A,\a,\omega,\{\rho_i \,\colon\, i \in \N\})
\]
as follow.
We write
\[
A_g = F(\S_g).
\]
This vector space comes equipped with a representation
\[
\rho_g \colon \M_g \to \text{Aut}(A_g).
\]
Indeed, given $d \in \Diff(\S_g)$, let $\rho_g(d) = F(c_d) \colon F(\S_g) \to F(\S_g)$,
where $c_d$ is the cobordism associated to~$d$ as in Definition~\ref{def:diffcob}.
We define the involution
\[
*_g \colon A_g \to A_g
\]
as $*_g = \rho_g(\iota_g)$, where
$\iota_g$ is $\pi$-rotation of~$\S_g$ in~$\R^3$ with center at~$\underline{0}$ about the $z$-axis

As defined in Section~\ref{sec:surfaces},
there is a natural identification between $\S_g(l_g)$ and $\S_{g-1}$, and so we can view~$W(l_g)$, the trace
of the surgery along~$l_g$, as a cobordism from~$\S_g$ to~$\S_{g-1}$. We write
\[
\a_g := F(W(l_g)) \colon A_g \to A_{g-1}.
\]
Similarly, we can identify $\S_g(s_j)$ with $\S_j \sqcup \S_{g-j}$, and hence we obtain a map
\[
\d_{j, g-j} := F(W(s_j)) \colon A_g \to A_j \otimes A_{g-j}
\]
for every $j \in \{0, \dots, g\}$, where we map $F(\S_j \sqcup \S_{g-j})$
to $F(S_j) \otimes F(S_{g-j}) = A_j \otimes A_{g-j}$ via the monoidal structure of~$F$.
We can canonically identify $(\S_i \sqcup \S_j)(\PP_{i,j})$ with $\S_{i+j}$, hence we obtain a map
\[
\mu_{i,j} := F(W(\bP_{i,j})) \colon A_i \otimes A_j \to A_{i+j}.
\]
Again, we used the monoidal structure of~$F$
Furthermore, $\S_g(\PP_g)$ is canonically diffeomorphic to~$\S_{g+1}$, hence we obtain a map
\[
\omega_g := F(W(\PP_g)) \colon A_g \to A_{g+1}.
\]

The ball $D^3$, viewed as a cobordism from $\S_0 = S^2$ to $\emptyset$, gives rise to a map
\[
\tau \colon A_0 \to \F,
\]
while viewing $D^3$ as a cobordism from $\emptyset$ to $\S_0$ gives a map
\[
\eps \colon \F \to A_0.
\]
Finally, we set $A = \bigoplus_{i \in \N} A_i$, $\mu = \bigoplus_{i, j \in \N} \mu_{i,j}$,
$\d = \bigoplus_{i,j \in \N} \d_{i,j}$, $\a = \bigoplus_{i \in \N} \a_i$,
$\omega = \bigoplus_{i \in \N} \omega_i$,
and $\A := (A,\mu,\delta,\eps,\tau,*)$. For an illustration of the above operations,
see Figure~\ref{fig:operations}.
\begin{figure}
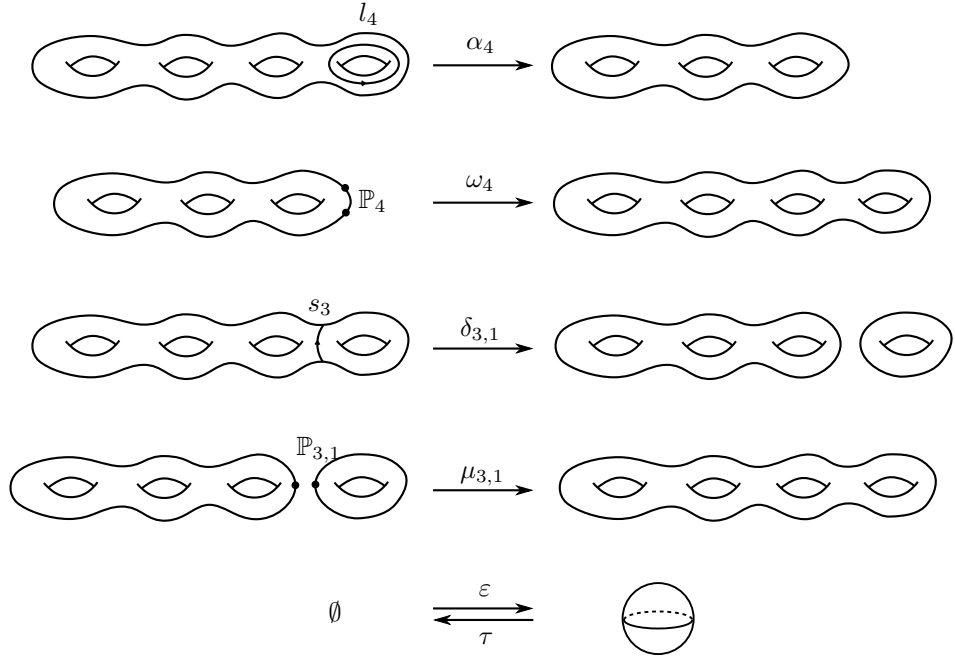
\caption{Given a (2+1)-dimensional TQFT $F$, we obtain the algebraic operations
$\alpha$, $\omega$, $\delta$, $\mu$, $\varepsilon$, and $\tau$ by applying~$F$ to the traces
of the surgeries in the figure.}
\label{fig:operations}
\end{figure}
In what follows, we synthesize the above data into a new algebraic structure called a \emph{J-algebra}.
This consists of the \emph{split \gnf-algebra} $(\A,\a,\omega)$,
together with the \emph{mapping class group representation} $\{\rho_i \colon i \in \N\}$.
We proceed to give the relevant algebraic definitions.

\subsection{Split \gnf-algebras} \label{sec:gnf}
\emph{Nearly Frobenius algebras} were introduced by Cohen and Godin~\cite{nearlyfrob}.
They are like Frobenius algebras, but
without the trace functional, and hence lack the non-degenerate bilinear pairing
that identifies the algebra with its dual. Note that a non-degenerate pairing forces
every Frobenius algebra to be finite dimensional, whereas this is not the case for
nearly Frobenius algebras. We now introduce a graded involutive version of this notion.

\begin{defn} \label{def:grad-frob}
A \emph{graded involutive nearly Frobenius algebra} (or \gnf-algebra for short) is a tuple $\A = (A,\mu,\delta,\eps,\tau,*)$, where
\[
A = \bigoplus_{i = 0}^ \infty A_i
\]
is an $\N$-graded $\F$-vector space such that each~$A_i$ is finite dimensional.
Furthermore,
\begin{enumerate}
\item $\mu \colon A \otimes A \to A$ is a graded linear map, where $A \otimes A$ is the graded
tensor product; i.e.,
\[
(A \otimes A)_n = \bigoplus_{i = 0}^n A_i \otimes A_{n-i} \le A \otimes_\F A,
\]
\item $\mu$ is associative and $\eps \colon \F \to A_0$ is a left unit for~$\mu$,
\item $\delta \colon A \to A \otimes A$ is a graded linear map that is coassociative and~$\tau \colon A_0 \to \F$
    is a partial left counit for~$\delta$ in the sense that $(\tau \otimes \Id_{A_j}) \circ \delta_{0,j} = \Id_{A_j}$,
    where $\d_{i,j} = \pi_{i,j} \circ \d$ and $\pi_{i,j} \colon A \otimes A \to A_i \otimes A_j$ is the projection,
\item \label{it:left} the following diagram is commutative (Frobenius condition):
\[
\xymatrixcolsep{5pc}
\xymatrix{
  A_i \otimes A_{j+k} \ar[r]^-{\Id_{A_i} \otimes \delta_{j,k}} \ar[d]^-{\mu_{i,j+k}} &
  A_i \otimes A_j \otimes A_k \ar[d]^{\mu_{i,j} \otimes \Id_{A_k}} \\
  A_{i+j+k}  \ar[r]^-{\delta_{i+j,k}} & A_{i+j} \otimes A_k,}
  \]
\item $* \colon A \to A$ is a grading-preserving involution that is an antiautomorphism of $(A,\mu,\d)$,
    and such that it is the identity on~$A_0$ and $A_1$. More concretely,
    \[
    \begin{split}
    * \circ \mu &= \mu \circ T,\\
    \d \circ * &= T \circ \d,
    \end{split}
    \]
    where $T = \bigoplus_{i,j = 0}^\infty T_{i,j}$, and $T_{i,j}(x \otimes y) = y^* \otimes x^*$ for $x \in A_i$ and $y \in A_j$.
\end{enumerate}
We shall write $\mu_{i,j} = \mu|_{A_i \otimes A_j}$.
\end{defn}

\begin{defn} \label{def:splitting}
A \emph{modular splitting} of the \gnf-algebra~$\A$ consists of a degree~1 endomorphism $\omega \colon A \to A$
and a degree $-1$ endomorphism $\a \colon A \to A$ such that they are both left $(A,\mu)$-module
homomorphisms, and such that
\[
\begin{split}
\d_{i,j-1} \circ \a_{i+j} &= (\Id_{A_i} \otimes \a_j) \circ \d_{i,j}, \\
\d_{i,j+1} \circ \omega_{i+j} &= (\Id_{A_i} \otimes \omega_j) \circ \d_{i,j}, \text{ and}\\
\a \circ \omega &= \Id_A,
\end{split}
\]
where $\a_i = \a|_{A_i}$ and $\omega_i = \omega|_{A_i}$. We call the triple $(\A,\a,\omega)$ a \emph{split} \gnf-algebra.
\end{defn}

\begin{defn}
Let
\[
(\A,\a,\omega) = (A,\mu,\delta,\eps,\tau,*,\a,\omega) \text{ and }
(\A',\a',\omega') = (A',\mu',\delta',\eps',\tau',*',\a',\omega')
\]
be split \gnf-algebras.
A \emph{homomorphism} from $(\A,\a,\omega)$ to $(\A',\a',\omega')$
is a graded linear map $h \colon A \to A'$ that intertwines the operations
$\mu$, $\delta$, $\eps$, $\tau$, $*$, $\a$, $\omega$ with
$\mu'$, $\delta'$, $\eps'$, $\tau'$, $*'$, $\a'$, $\omega'$, respectively.
\end{defn}


The following straightforward lemma
restates the definition of a split \gnf-algebra
in terms of the operations $\mu_{i,j}$, $\delta_{i,j}$, $\a_i$, and $\omega_i$.

\begin{lem} \label{lem:corresp}
Let $(\A,\a,\omega)$ be a split \gnf-algebra.
Then the product~$\mu$ is associative with left unit~$\eps$:
\begin{equation} \label{eqn:1}
\begin{split}
\mu_{i + j,k} \circ (\mu_{i,j} \otimes \Id_{A_k}) &= \mu_{i, j+k} \circ (\Id_{A_i} \otimes \mu_{j, k}), \\
\mu_{0,j} \circ (\eps \otimes \Id_{A_j}) &= \Id_{A_j}.
\end{split}
\end{equation}
The coproduct~$\d$ is coassociative with left counit~$\tau$:
\begin{equation} \label{eqn:2}
\begin{split}
(\Id_{A_i} \otimes \delta_{j,k}) \circ \delta_{i,j+k} &= (\delta_{i,j} \otimes \Id_{A_k}) \circ \delta_{i+j,k}, \\
(\tau \otimes \Id_{A_j}) \circ \delta_{0,j} &= \Id_{A_j}.
\end{split}
\end{equation}
The operations $\mu$ and~$\d$ satisfy the Frobenius condition
\begin{equation} \label{eqn:3}
\delta_{i+j,k} \circ \mu_{i,j+k} =  (\mu_{i,j} \otimes \Id_{A_k} ) \circ (\Id_{A_i} \otimes \delta_{j,k}).
\end{equation}
The operation~$*$ is an anti-automorphism:
\begin{equation} \label{eqn:4}
\begin{split}
\mu_{i,j}(x^* \otimes y^*) &= \mu_{j,i}(y \otimes x)^*, \\
T_{i,j} \circ \delta_{i,j}(x) &= \delta_{j,i}(x^*),
\end{split}
\end{equation}
where $T_{i,j} \colon A_i \otimes A_j \to A_j \otimes A_i$ is given by $T_{i,j}(x \otimes y) = y^* \otimes x^*$.
Furthermore, $*$ is involutive, and is the identity on~$A_0$ and~$A_1$.

We have
\begin{equation} \label{eqn:5}
\a_{i+1}\circ \omega_i = \Id_{A_i},
\end{equation}
and the maps~$\a_i$ and~$\omega_i$ are compatible with the
product and coproduct in the following sense:
\begin{equation} \label{eqn:6}
\begin{split}
\omega_{i+j} \circ \mu_{i,j} &= \mu_{i,j+1} \circ (\Id_{A_i} \otimes \omega_j), \\
\a_{i+j} \circ \mu_{i,j} &= \mu_{i,j-1} \circ (\Id_{A_i} \otimes \a_j), \\
\delta_{i,j+1} \circ \omega_{i+j} &= (\Id_{A_i} \otimes \omega_j) \circ \delta_{i,j}, \\
\delta_{i,j-1} \circ \a_{i+j} &= (\Id_{A_i} \otimes \a_j) \circ \delta_{i,j}.
\end{split}
\end{equation}

In the opposite direction, suppose that we are given a sequence of
finite-di\-men\-sional $\F$-vector spaces $A_i$ for $i \in \N$,
together with products $\mu_{i,j} \colon A_i \otimes A_j \to A_{i+j}$, coproducts
$\delta_{i,j} \colon A_{i+j} \to A_i \otimes A_j$,
a left unit $\eps \colon \F \to A_0$, a left counit $\tau \colon A_0 \to \F$,
embeddings~$\omega_i \colon A_i \to A_{i+1}$, projections $\a_i \colon A_i \to A_{i-1}$, and
involutions $* \colon A_i \to A_i$ that satisfy equations~\eqref{eqn:1}--\eqref{eqn:6}.
If we set $A = \bigoplus_{i \in \N} A_i$, $\mu = \bigoplus_{i, j \in \N} \mu_{i,j}$,
$\d = \bigoplus_{i,j \in \N} \d_{i,j}$, $\a = \bigoplus_{i \in \N} \a_i$, and $\omega = \bigoplus_{i \in \N} \omega_i$,
then $(A,\mu,\delta,\eps,\tau,*,\a,\omega)$ is a split \gnf-algebra.
\end{lem}

\begin{proof}
It is clear that if $(\A,\a,\omega)$ is a split \gnf-algebra, then the operations
$\mu_{i,j}$, $\d_{i,j}$, $\eps$, $\tau$, $\a_i$, and $\omega_i$ satisfy equations~\eqref{eqn:1}--\eqref{eqn:6}.

Now consider the opposite direction.
It is also straightforward to check that $(A,\mu,\delta,\eps,*)$
satisfy the properties listed in Definitions~\ref{def:grad-frob} and~\ref{def:splitting}
required of a split \gnf-algebra.
The only non-trivial part is showing that~$\d$ is coassociative;
i.e., that $(\d \otimes \Id_A) \circ \d = (\Id_A \otimes \d) \circ \d$.
Restricted to~$A_n$, the left-hand side becomes
\[
\begin{split}
\sum_{i=0}^n \sum_{j=0}^i (\d_{j,i-j} \otimes \Id_{A_{n-i}}) \circ \d_{i,n-i} =
\sum_{i=0}^n \sum_{j=0}^i (\Id_{A_j} \otimes \d_{i-j,n-i}) \circ \d_{j,n-j} = \\
\sum_{j=0}^n \sum_{i=j}^n (\Id_{A_j} \otimes \d_{i-j,n-i}) \circ \d_{j,n-j} =
\sum_{j=0}^n \sum_{k=0}^{n-j} (\Id_{A_j} \otimes \d_{n-k-j,k}) \circ \d_{j,n-j} = \\
\sum_{l=0}^n \sum_{k=0}^l (\Id_{A_{n-l}} \otimes \d_{l-k,k}) \circ \d_{n-l,l},
\end{split}
\]
which is exactly the right-hand side restricted to~$A_n$. Here, the first equality
follows from the coassociativity~\eqref{eqn:2} of the (2+1)-algebra operations~$\d_{i,j}$, followed
by changing the order of summation, and finally setting $k = n-i$ and $l = n-j$.
\end{proof}

\begin{prop} \label{prop:A}
Let $F \colon \Cob_2 \to \Vect_\F$ be a TQFT, and let
\[
J(F) = (\A,\a,\omega,\{\rho_i \,\colon\, i \in \N\})
\]
be the associated tuple as in Section~\ref{sec:assignment}.
Then $(\A,\a,\omega)$ is a split \gnf-algebra.
\end{prop}

\begin{proof}
By the second part of Lemma~\ref{lem:corresp}, it suffices to check that
the data $A_i$, $\mu_{i,j}$, $\d_{i,j}$, $\eps$, $\tau$, $\a_i$, $\omega_i$
assigned to $F$ as in Section~\ref{sec:assignment} satisfy equations~\eqref{eqn:1}--\eqref{eqn:6}.
According to Theorem~\ref{thm:TQFT}, the TQFT $F$ satisfies relations~\eqref{it:isot}--\eqref{it:0-sphere} of Definition~\ref{def:relations}.
Together with the monoidality of~$F$, these imply equations~\eqref{eqn:1}--\eqref{eqn:6} of Lemma~\ref{lem:corresp} as follows.

First, consider equations~\eqref{eqn:1}. The equation
\[
\mu_{i + j,k} \circ (\mu_{i,j} \otimes \Id_{A_k}) = \mu_{i, j+k} \circ (\Id_{A_i} \otimes \mu_{j, k})
\]
follows by applying relation~\eqref{it:commut} to $\S_i \sqcup \S_j \sqcup \S_k$
with $\SS = \PP_{i,j} \subset \S_i \sqcup \S_j$ and $\SS' = \PP_{j,k} \subset \S_j \sqcup \S_k$,
together with the pentagon lemma of monoidality.
To show that
\[
\mu_{0,j} \circ (\eps \otimes \Id_{A_j}) = \Id_{A_j},
\]
we apply relation~\eqref{it:birth} to $\S_j$ with $\SS = 0$ and $\SS' = \PP_{0,j} \subset \S_0 \sqcup \S_j$.

Now consider equations~\eqref{eqn:2}. To show that
\[
(\Id_{A_i} \otimes \delta_{j,k}) \circ \delta_{i,j+k} = (\delta_{i,j} \otimes \Id_{A_k}) \circ \delta_{i+j,k},
\]
apply relation~\eqref{it:commut} to $\S_{i+j+k}$ with $\SS = s_i$ and $\SS' = s_{i+j}$.
For
\[
(\tau \otimes \Id_{A_j}) \circ \delta_{0,j} = \Id_{A_j},
\]
apply relation~\eqref{it:birth} to $\S_j$ with $\SS = s_0$ and $\SS'$ being the 2-sphere
split off by~$s_0$.

Equation~\eqref{eqn:3}, the Frobenius condition
\[
\delta_{i+j,k} \circ \mu_{i,j+k} =  (\mu_{i,j} \otimes \Id_{A_k} ) \circ (\Id_{A_i} \otimes \delta_{j,k}),
\]
follows from applying relation~\eqref{it:commut} to $\S_i \sqcup \S_{j+k}$ with $\SS = \PP_{i,j+k}$
and $\SS' = s_j \subset \S_{j+k}$.

For equations~\eqref{eqn:4},
\[
\mu_{i,j}(x^* \otimes y^*) = \mu_{j,i}(y \otimes x)^*,
\]
follows from relation~\eqref{it:d-F} by applying it to $\S_i \sqcup \S_j$ with $\SS = \PP_{i,j}$
and the diffeomorphism $d \colon \S_i \sqcup \S_j \to \S_j \sqcup \S_i$ being $\iota_i \sqcup \iota_j$,
followed by swapping the two components. Then note that $d^\SS = \iota_{i+j}$, and the result follows.
Now consider
\[
T_{i,j} \circ \delta_{i,j}(x) = \delta_{j,i}(x^*),
\]
where $T_{i,j} \colon A_i \otimes A_j \to A_j \otimes A_i$ is given by $T_{i,j}(x \otimes y) = y^* \otimes x^*$.
This also follows from relation~\eqref{it:d-F} applied to $\S_{i+j}$ with $\SS = s_i$ and $d = \iota_{i+j}$.
Furthermore, $*_i$ is involutive since $\iota_i$ is, and $*_0 = \Id_{A_0}$ and $*_1 = \Id_{A_1}$
as $\iota_0$ and $\iota_1$ are isotopic to the identity, together with relation~\eqref{it:isot}.

To prove equation~\eqref{eqn:5},
\[
\a_{i+1} \circ \omega_i = \Id_{A_i},
\]
we apply relation~\eqref{it:birth} to $\S_i$ with $\SS = \PP_i$ and $\SS' = l_{j+1}$ that form
a canceling pair.

The last set of equations is~\eqref{eqn:6}.
The equation
\[
\omega_{i+j} \circ \mu_{i,j} = \mu_{i,j+1} \circ (\Id_{A_i} \otimes \omega_j),
\]
follows from applying relation~\eqref{it:commut} to $\S_i \sqcup \S_j$ with $\SS = \PP_{i,j}$
and $\SS' = \PP_j \subset \S_j$.
Similarly,
\[
\a_{i+j} \circ \mu_{i,j} = \mu_{i,j-1} \circ (\Id_{A_i} \otimes \a_j)
\]
follows from relation~\eqref{it:commut} applied to $\S_i \sqcup \S_j$ with $\SS = \PP_{i,j}$
and $\SS' = l_j \subset \S_j$.
To obtain
\[
\delta_{i,j+1} \circ \omega_{i+j} = (\Id_{A_i} \otimes \omega_j) \circ \delta_{i,j},
\]
apply relation~\eqref{it:commut} to $\S_{i+j}$ with $\SS = \PP_{i+j}$ and $\SS' = s_i$.
Finally,
\[
\delta_{i,j-1} \circ \a_{i+j} = (\Id_{A_i} \otimes \a_j) \circ \delta_{i,j}
\]
follows by applying relation~\eqref{it:commut} to $\S_{i+j}$ along $\SS = l_{i+j}$
and $\SS' = s_i$.
\end{proof}


\begin{lem} \label{lem:right}
If~$\A$ is a \gnf-algebra, then $\eps$ is also a right unit, $\tau$ is a partial right counit, and
\begin{equation} \label{eqn:right}
\d_{k,i+j} \circ \mu_{j+k,i} = (\Id_{A_k} \otimes \mu_{j,i}) \circ (\d_{k,j} \otimes \Id_{A_i}).
\end{equation}
If $(\a,\omega)$ is a modular splitting of~$\A$, then $A = \ker(\a) \oplus \im(\omega)$, both summands are left
$(A,\mu)$-submodules, and $\omega \circ \a$ is projection onto $\im(\omega)$ along $\ker(\a)$.
\end{lem}

\begin{proof}
By applying $*$ to the equation $\mu(\eps(t) \otimes a) = a$ for $t \in \F$ and $a \in A$,
we obtain that $\mu(a^* \otimes \eps(t)) = a^*$, as $\eps(t) \in A_0$ on which $*$ acts as the identity,
and hence $\mu(a \otimes \eps(t)) = a$ for every $a \in A$.

Similarly, since $\d_{0,j} \circ * = T_{j,0} \circ \d_{j,0}$, we have
\[
* = (\tau \otimes \Id_{A_j}) \circ \delta_{0,j} \circ * = (\tau \otimes \Id_{A_j}) \circ T_{j,0} \circ \d_{j,0} =
(* \otimes \tau) \circ \d_{j,0}
\]
as $\tau \circ * = \tau$ since $*$ acts as the identity on~$A_0$. Applying~$*$ to both sides,
\[
(\Id_{A_j} \otimes \tau) \circ \d_{j,0} = \Id_{A_j}.
\]

To prove equation~\eqref{eqn:right}, we use the sumless Sweedler notation
\[
\d_{m,n}(x) = x_{(1)}^m \otimes x_{(2)}^n,
\]
where~$x \in A_{m+n}$.
Then condition~\eqref{it:left} of Definition~\ref{def:grad-frob} can be written as
\[
\mu_{i,j} \left(a \otimes b_{(1)}^j \right) \otimes b_{(2)}^k =
\mu_{i,j+k}(a \otimes b)_{(1)}^{i+j} \otimes \mu_{i,j+k}(a \otimes b)_{(2)}^k
\]
for every $a \in A_i$ and $b \in A_{j+k}$.
Applying~$T$ to both sides,
\[
\left(b_{(2)}^k\right)^* \otimes \mu_{i,j} \left(a \otimes b_{(1)}^j \right)^*  =  \left(\mu_{i,j+k}(a \otimes b)_{(2)}^k \right)^* \otimes \left(\mu_{i,j+k}(a \otimes b)_{(1)}^{i+j}\right)^*.
\]
Since $*$ is an $(A,\d)$-antihomomorphism, $(x^*)_{(1)}^m \otimes (x^*)_{(2)}^n =
\left(x_{(2)}^m \right)^* \otimes \left(x_{(1)}^n \right)^*$ for every~$x \in A_{m+n}$, hence
\[
\begin{split}
(b^*)_{(1)}^k \otimes \mu_{j,i} \left((b^*)_{(2)}^j \otimes a^* \right) &= \left(\mu_{i,j+k}(a \otimes b)^*\right)_{(1)}^k \otimes
\left(\mu_{i,j+k}(a \otimes b)^* \right)_{(2)}^{i+j} \\
&= \mu_{j+k,i}(b^* \otimes a^*)_{(1)}^k \otimes \mu_{j+k,i}(b^* \otimes a^*)_{(2)}^{i+j}.
\end{split}
\]
As this holds for every $b^* \in A_{j+k}$ and $a^* \in A_i$, we obtain equation~\eqref{eqn:right}.

For the last part, $\ker(\a)$ and $\im(\omega)$ are left $(A,\mu)$-submodules
since~$\a$ and~$\omega$ are left $(A,\mu)$-module homomorphisms.
Since $\a \circ \omega = \Id_A$, we see that~$\a$ is surjective and~$\omega$ is injective.
Furthermore, the endomorphism $\omega \circ \a$ is a projection since
$(\omega \circ \a) \circ (\omega \circ \a) = \omega \circ \a$. As~$\a$ is onto, $\im(\omega \circ \a) = \im(\omega)$,
and since~$\omega$ is injective, $\ker(\omega \circ \a) = \ker(\a)$. It follows that $A = \ker(\a) \oplus \im(\omega)$,
and that $\omega \circ \a$ is projection onto~$\im(\omega)$ along~$\ker(\a)$.
\end{proof}

\begin{rem} \label{rem:duality}
Since $\omega$ is not necessarily $*$-invariant, the splitting $A = \ker(\a) \oplus \im(\omega)$
is not $*$-invariant in general. If we introduce the notation $\ol{\omega}(a) = \omega(a^*)^*$,
then
\[
\mu(\ol{\omega}(a) \otimes b) = \mu(b^* \otimes \omega(a^*))^*
= (\omega \circ \mu(b^* \otimes a^*))^* = \ol{\omega} \circ \mu(a \otimes b).
\]
So, instead of~$\omega$, it is $\ol{\omega}$ that is a right $(A,\mu)$-module homomorphism,
and similarly for~$(A,\d)$.
\end{rem}

\begin{rem}
Given a split \gnf-algebra $(\A,\a,\omega)$, consider the direct system of vector spaces
\[
\omega_{i,j} := \omega_{j-1} \circ \dots \circ \omega_i \colon A_i \to A_j
\]
for $i \le j$, and let
\[
M = \underrightarrow{\lim} A_i = \coprod_{i = 0}^\infty A_i \Big/\sim,
\]
where $x_i \sim x_j$ for $x_i \in A_i$ and $x_j \in A_j$ if and only if there is some~$k \ge i$, $j$
for which $\omega_{i,k}(x_i) = \omega_{j,k}(x_j)$.
Since each~$\omega_i$ is injective, we can choose $k = \max\{i,j\}$.
Furthermore, we can canonically identify~$A_i$ with a subspace~$M_i$ of~$M$, under which~$\omega_i$
becomes the embedding $M_i \hookrightarrow M_{i+1}$. For simplicity, we also use the notation~$\omega_i$
for this embedding. Using the same identification, $\a_i$ descends to a map $\a_i \colon M_i \to M_{i-1}$,
which we also denote by~$\a_i$. Since $\a_i \circ \omega_{i-1} = \Id_{M_{i-1}}$, we have $\a_i(x) = x$ for
every $x \in M_{i-1}$; i.e., $\omega_{i-1} \circ \a_i \colon M_i \to M_i$ is a projection onto~$M_{i-1}$.

Next, we show that the~$\mu_{i,j}$ descend to a well-defined product~$\mu_i \colon A_i \otimes M \to M$.
Given $m \in M$, we define $\mu(a \otimes m)$ for $a \in A_i$ by taking an arbitrary representative $x \in A_j$
of~$m$, and we let $\mu(a \otimes m) = \mu_{i,j}(a \otimes x)$. The equivalence class of this product is
independent of the representative~$x$. Indeed, given two representative $x \sim x'$
such that $x \in A_j$, $x' \in A_k$, and $\omega_{j,k}(x) = x'$, we have
\[
\mu_{i,k}(a \otimes \omega_{j,k}(x)) = \omega_{i+j,i+k} \circ \mu_{i,j}(a \otimes x) \sim \mu_{i,j}(a \otimes x)
\]
as~$\omega$ is a left $(A,\mu)$-module homomorphism.

Similarly, the maps~$\d_{i,j}$ descend to a map $\d_i \colon M \to A_i \otimes M$
as~$\omega$ is a left $(A,\d)$-comodule homomorphism. In particular, for $m \in M$, we define~$\d_i(m)$
to be $\d_{i,n-i}(x)$ for some representative $x \in A_n$  of~$m$.
We now show this is independent of the choice of~$x$. Indeed,
\[
\d_{i,n-i}(x) \sim (\Id_{A_i} \otimes \omega_{n-i}) \circ \d_{i,n-i}(x) = \d_{i,n-i+1} \circ \omega_n(x).
\]
It follows that $M$ is a left $\A$-module.

By taking the direct limit of $A_i$ along the maps $\ol{\omega}_i$,
we get a right $\A$-module~$\ol{M}$. It follows from Remark~\ref{rem:duality} that $*$ provides an anti-isomorphism
between~$M$ and~$\ol{M}$; in particular, $\ol{M} \cong M^{\text{op}}$.
\end{rem}

Next, we present an alternate, simpler definition of a modular splitting. Let
\[
1 := \eps(1_\F) \in A_0 \setminus \{0\}
\]
be the unit of the \gnf-algebra~$\A$.

\begin{lem} \label{lem:splitting}
There is a bijection between modular splittings~$(\a,\omega)$ of the \gnf-algebra~$\A$, and pairs of
elements $(w,\l) \in A_1 \times A_1^*$ for which
\[
(\Id_{A_0} \otimes \l) \circ \d_{0,1}(w) = 1.
\]
Given $(w,\l)$, we get $(\a,\omega)$ by the formulae
\[
\begin{split}
\omega_i(x) &= \mu_{i,1}(x \otimes w), \text{ and} \\
\a_i(x) &= (\Id_{A_{i-1}} \otimes \l) \circ \d_{i-1,1}(x).
\end{split}
\]
In the opposite direction, given $(\a,\omega)$, we let $w = \omega_0(1)$ and $\l = \tau \circ \a_1$.
\end{lem}

\begin{proof}
Suppose we are given a modular splitting $(\a,\omega)$ of $\A$,
and let $w := \omega_0(1) \in A_1$. Then
\[
\mu_{i,1}(x \otimes w) = \mu_{i,1}(x \otimes \omega_0(1)) = \omega_i \circ \mu_{i,0}(x \otimes 1) = \omega_i(x)
\]
for every $i \in \N$ and $x \in A_i$ since $\omega$ is a left $(A,\mu)$-module homomorphism and~$1$ is a unit.
Hence, the element $w \in A_1$ completely determines~$\omega_i$ for every $i \in \N$.
Indeed, if we define~$\omega_i$ by the formula
\[
\omega_i(x) := \mu_{i,1}(x \otimes w),
\]
then it is a left $(A,\mu)$-module homomorphism by the associativity of~$\mu_{i,j}$:
\[
\omega_{i+j} \circ \mu_{i,j}(x,y) = \mu_{i+j,1}(\mu_{i,j}(x,y),w) =
\mu_{i,j+1}(x \otimes \mu_{j,1}(y,w)) = \mu_{i,j+1}(x \otimes \omega_j(y)).
\]
Furthermore, $\omega$ is a left $(A,\d)$-comodule homomorphism as $\d$ is a right $(A,\mu)$-module
homomorphism according to Lemma~\ref{lem:right}:
\[
\begin{split}
\d_{i,j+1} \circ \omega_{i+j}(x) &= \d_{i,j+1} \circ \mu_{i+j,1}(x,w) = \\
(\Id_{A_i} \otimes \mu_{j,1}) \circ (\d_{i,j} \otimes \Id_{A_1})(x \otimes w) &=
(\Id_{A_i} \otimes \omega_j) \circ \d_{i,j}(x).
\end{split}
\]
Similarly, if we are given the splitting~$(\a,\omega)$ and let $\l = \tau \circ \a_1$, then
\[
\begin{split}
(\Id_{A_{i-1}} \otimes \l) \circ \d_{i-1,1} &=
(\Id_{A_{i-1}} \otimes \tau) \circ (\Id_{A_{i-1}} \otimes \a_1) \circ \d_{i-1,1} = \\
(\Id_{A_{i-1}} \otimes \tau) \circ \d_{i-1,0} \circ \a_i &= \a_i
\end{split}
\]
as $\a$ is a left $(A,\d)$-comodule homomorphism and $\tau$ is a counit.
So~$\lambda \in A_1^*$ completely determines~$\a_i$ for every~$i \in \N$ via the formula
\[
\a_i(x) := (\Id_{A_{i-1}} \otimes \l) \circ \d_{i-1,1}.
\]
The~$\a$ defined this way is a left $(A,\mu)$-module homomorphism by the Frobenius condition~\eqref{it:left}:
\[
\begin{split}
\a_{i+j} \circ \mu_{i,j} = (\Id_{A_{i+j-1}} \otimes \l) \circ \d_{i+j-1,1} \circ \mu_{i,j} &= \\
(\Id_{A_{i+j-1}} \otimes \l) \circ (\mu_{i,j-1} \otimes \Id_{A_1}) \circ (\Id_{A_i} \otimes \d_{j-1,1})&=
\mu_{i,j-1} \circ (\Id_{A_i} \otimes \a_j).
\end{split}
\]
Similarly, $\a$ is a left $(A,\d)$-comodule homomorphism by the coassociativity of~$\d$:
\[
\begin{split}
\delta_{i,j-1} \circ \a_{i+j} = \delta_{i,j-1} \circ (\Id_{A_{i+j-1}} \otimes \l) \circ \d_{i+j-1,1} &= \\
(\Id_{A_i} \otimes \Id_{A_{j-1}}  \otimes \l) \circ (\delta_{i,j-1} \otimes \Id_{A_1}) \circ \d_{i+j-1,1} &= \\
(\Id_{A_i} \otimes \Id_{A_{j-1}}  \otimes \l) \circ (\Id_{A_i} \otimes \d_{j-1,1}) \circ \d_{i,j} =
(\Id_{A_i} \otimes \a_j) \circ \delta_{i,j}.
\end{split}
\]
Finally, consider the condition~$\a_{i+1} \circ \omega_i = \Id_{A_i}$. Since
\[
\a_{i+1} \circ \omega_i(x) = \a_{i+1} \circ \mu_{i,1}(x \otimes w) = \mu_{i,0}(x \otimes \a_1(w)),
\]
this is equivalent to having $\mu_{i,0}(x \otimes \a_1(w)) = x$ for every $i \in \N$ and $x \in A_i$.
In particular, if we set $i = 0$ and $x = 1$, we must have $\a_1(w) = 1$, and clearly this is also sufficient.
But $\a_1(w) = (\Id_{A_0} \otimes \l) \circ \d_{0,1}(w)$, so the condition $\a_{i+1} \circ \omega_i = \Id_{A_i}$ is
equivalent to
\[
(\Id_{A_0} \otimes \l) \circ \d_{0,1}(w) = 1.
\]
This concludes the proof of the lemma.
\end{proof}

\begin{rem}
From now on, we use the notation $(\a,\omega)$ and $(w,\l)$ interchangeably for a modular splitting.
Notice that the polynomial algebra $\F[w]$ is a subalgebra of~$(A,\mu)$, and $\F[\l]$ is a subalgebra
of $(A^*,\d^*)$.

We saw in Section~\ref{sec:assignment} that if the split \gnf-algebra~$(\A,\a,\omega)$
arises from a (2+1)-di\-men\-sional TQFT~$F$,
then the map $\omega_i$ geometrically corresponds to performing a surgery on a genus~$i$ surface~$\S_i$
along a framed pair of points~$\bP_i$,
while the operation $\mu_{i,1}$ amounts to connected summing $\S_i$ with~$T^2$, and $w \in F(T^2)$.
\end{rem}

\subsection{Mapping class group representations on split \gnf-algebras} \label{sec:MCG}

Let $F \colon \Cob_2 \to \Vect_\F$ be a TQFT, and let
\[
J(F) = (\A,\a,\omega,\{\rho_i \,\colon\, i \in \N\})
\]
be the associated tuple as in Section~\ref{sec:assignment}.
Then the mapping class group actions~$\rho_i$ on~$A_i$
are compatible with the \gnf-algebra structure of $(\A,\a,\omega)$
in a sense that we now formalize.

\begin{defn}
Let $\SS \colon S^k \times D^{n-k} \hookrightarrow M$ be a framed sphere in the $n$-manifold~$M$.
Then let
\[
\Diff(M,\SS) = \{\, d \in \Diff(M) \,\colon\, d \circ \SS = \SS \,\},
\]
and we set $\MCG(M,\SS) = \Diff(M,\SS)/\Diff_0(M,\SS)$.
\end{defn}

Note that there is a natural forgetful map
\[
f_\SS \colon \MCG(M,\SS) \to \MCG(M).
\]
For $d \in \Diff(M,\SS)$, the induced map $d^\SS \in \Diff(M(\SS))$
fixes the framed belt sphere
\[
\SS^* \colon D^{k+1} \times S^{n-k+1} \hookrightarrow M(\SS)
\]
of the handle attached along~$\SS$, hence
\[
d^\SS \in \Diff(M(\SS),\SS^*).
\]
As $d^{\SS,\SS^*} = d$, this correspondence gives an isomorphism
\begin{equation} \label{eqn:duality}
\MCG(M,\SS) \cong \MCG(M(\SS),\SS^*).
\end{equation}
We denote the image of $\phi \in \MCG(M,\SS)$ under this isomorphism by $\phi^\SS$.

\begin{defn} \label{def:equivar}
Let $\SS$ be a framed sphere in the manifold $M$.
Suppose we are given representations $\rho \colon \MCG(M) \to \text{Aut}(V)$
and $\rho' \colon \MCG(M(\SS)) \to \text{Aut}(V')$.
Then we say that a linear map $h \colon V \to V'$ is $\MCG(M,\SS)$-\emph{equivariant} if
\[
h \circ \rho(f_\SS(\phi)) = \rho'\left(f_{\SS^*}\left(\phi^\SS\right)\right) \circ h
\]
for every $\phi \in \MCG(M,\SS)$.
\end{defn}

\begin{defn} \label{def:repr}
Let $(\A,\a,\omega)$ be a split \gnf-algebra. Then a sequence of homomorphisms
\[
\{\,\rho_i \colon \M_i \to \text{Aut}(A_i) \,|\, i \in \N \,\}
\]
is called a \emph{mapping class group representation} on $(\A,\a,\omega)$ if it satisfies the following properties:

The map $\mu_{i,j}$ is $\MCG(\S_i \sqcup \S_j, \bP_{i,j})$-equivariant
and $\delta_{i,j}$ is $\MCG(\S_{i+j}, s_i)$-equi\-va\-ri\-ant.
Furthermore, $*|_{A_i} = \rho(\iota_i)$, and the representations~$\rho_i$ satisfy the following conditions:
\begin{enumerate}
\item \label{it:condw} $\rho_1(t_1)(w) = w$ and $\rho_1(\tau_m)(w) = w$,
\item \label{it:condl} $\l \circ \rho_1(t_1) = \l$ and $\l \circ \rho_1(\tau_l) = \l$,
\item \label{it:L} $\a_{i+1} \circ \rho_{i+1}(L_{i+1}) \circ \omega_i = \omega_{i-1} \circ \a_i$ for $i > 1$,
\item \label{it:frobtwist} $\a_{n+1} \circ \rho_{n+1}(\sigma_{n+1,i}) \circ \omega_n = \mu_{i,n-i} \circ \delta_{i,n-i}$
    for $n \in \N$ and $0 \le i \le n$,
\end{enumerate}
where $w = \a_1(1)$ and $\l = \tau \circ \a_1$ are as in Lemma~\ref{lem:splitting}, and
\begin{itemize}
\item $\iota_i$ is $\pi$-rotation of the standard~$\S_i$ in~$\R^3$ with center at~$\underline{0}$ about the $z$-axis,
\item $t_1$ is $\pi$-rotation of the standard torus in~$\R^3$ about the $x$-axis,
\item $\tau_m$, $\tau_l \in \Diff(T^2)$ are right-handed Dehn twists about the meridian and longitude, respectively,
\item $L_i \in \Diff(\S_i)$ swaps~$l_i$ and~$l_{i-1}$ counterclockwise, and~$L_i^{l_i, l_{i-1}} \in \Diff_0(\S_{i-2})$,
\item $\sigma_{n+1,i} = a_{n+1} \circ h_{n+1,i} \in \Diff(\S_{n+1})$,
    where $a_{n+1}$ is the identity on the component of $\S_{n+1} \setminus s_n$
    containing~$p_{n+1}$, and it satisfies $a_{n+1}(m_{n+1}) = l_{n+1}$ and $a_{n+1}(l_{n+1}) = -m_{n+1}$,
    while $h_{n+1,i} \in \Diff(\S_{n+1})$ swaps $m_{n+1}$ and
    $s_i \# m_{n+1}$ for some connected sum arc, and
    $(h_{n+1,i})^{m_{n+1}, s_i \# m_{n+1}}$ is isotopic to the identity.
    For more detail, see the proof of Proposition~\ref{prop:AF} and Figure~\ref{fig:relations}.
\end{itemize}
\end{defn}

\begin{defn} \label{def:J-alg}
A \emph{J-algebra} is a four-tuple $(\A,\a,\omega,\{\rho_i \,\colon\, i \in \N\})$, where
$(\A,\a,\omega)$ is a split \gnf-algebra and $\{\rho_i \,\colon\, i \in \N\}$ is a mapping class group
representation on it.
\end{defn}

\begin{defn}
A \emph{homomorphism} between two J-algebras
is a homomorphism of the underlying split \gnf-algebras that intertwines the mapping class group representations.
The \emph{direct sum} of two J-algebras is the direct sum of the underlying split involutive
\gnf-algebras, together with the direct sum of the mapping class group representations.
Then J-algebras together with such homomorphisms form a
symmetric monoidal category that we denote $\JAlg$.
\end{defn}

\begin{prop} \label{prop:AF}
Let $F \colon \Cob_2 \to \Vect_\F$ be a TQFT. Then the tuple
\[
J(F) = (\A,\a,\omega,\{\rho_i \,\colon\, i \in \N\})
\]
defined in Section~\ref{sec:assignment} is a J-algebra.
\end{prop}

\begin{proof}
By Proposition~\ref{prop:A}, the tuple $(\A,\a,\omega)$ is a split \gnf-algebra.
We now show that $\{\rho_i \,\colon\, i \in \N\}$
is a mapping class group representation on it.
It follows from Lemma~\ref{lem:funct} that
the map~$\a_g$ is $\MCG(\S_g,l_g)$-equivariant,
$\d_{i, j}$ is $\MCG(\S_{i+j},s_i)$-equivariant,
$\mu_{i,j}$ is $\MCG(\S_i \sqcup \S_j, \PP_{i,j})$-equivariant,
and finally, $\omega_g$ is $\MCG(\S_g, \PP_g)$-equivariant (see Definition~\ref{def:equivar}).

The equation $\rho_1(t_1)(w) = w$ is equivalent to
$\rho_1(t_1) \circ \omega_0 = \omega_0$.
Let $t_0$ be $\pi$-rotation of $\S_0$ about the $x$-axis.
This satisfies $t_1 = (t_0)^{\bP_0}$, $t_0 \circ \bP_0 = \ol{\bP}_0$, and is isotopic to~$\Id_{\S_0}$.
Hence, if we apply relation~\eqref{it:d-F} of Definition~\ref{def:relations},
naturality of the surgery maps, to $t_0$ and $\bP_0$, then we obtain that
\[
\rho_1(t_1) \circ \omega_0 = F_{\S_0,\ol{\bP}_0} \circ \rho_0(t_0) = \omega_0,
\]
as $F_{\S_0,\ol{\bP}_0} = F_{\S_0,\bP_0} = \omega_0$ by relation~\eqref{it:0-sphere} of Definition~\ref{def:relations}.

Similarly, $\rho_1(\tau_m)(w) = w$ is equivalent to $\rho_1(\tau_m) \circ \omega_0 = \omega_0$.
This holds since we can apply the $\MCG(\S_0, \PP_0)$-equivariance of~$\omega_0$
to the diffeomorphism~$d$ that is a Dehn twist on $\S_0$ about a circle that separates
the two points of $\PP_0$, which is isotopic to the identity in $\MCG(\S_0)$
(but not in $\MCG(\S_0,\PP_0)$), and because $d^{\PP_0}$ is isotopic to $\tau_m$.

The equation $\l \circ \rho_1(t_1) = \l$ is equivalent to $\a_1 \circ \rho_1(t_1) = \a_1$.
This, in turn, follows from relations~\eqref{it:d-F} and~\eqref{it:0-sphere} of Definition~\ref{def:relations}
applied to $t_1$ and $l$,
as $(t_1)^l = t_0$ is isotopic to~$\Id_{\S_0}$.
Similarly, $\l \circ \rho_1(\tau_l) = \l$ is equivalent to $\a_1 \circ \rho_1(\tau_l) = \a_1$,
which follows from the $\MCG(\S_1,l_1)$-equivariance of $\a_1$ applied to $\tau_l$.

Now consider the condition
$\a_{i+1} \circ \rho_{i+1}(L_{i+1}) \circ \omega_i = \omega_{i-1} \circ \a_i$.
Relation~\eqref{it:commut} of Definition~\ref{def:relations} for $\S_i$ with $\SS = l_i$ and $\SS' = \PP_i$ yields
\[
F_{\S_i(\SS),\SS'} \circ F_{\S_i,\SS} = F_{\S_{i+1},\SS} \circ F_{\S_i,\SS'}.
\]
Under the identification $\S_i(l_i) \approx \S_{i-1}$, the framed sphere $\PP_i$
becomes $\PP_{i-1}$, giving  $F_{\S_i(\SS),\SS'} \circ F_{\S_i,\SS} = \omega_{i-1} \circ \alpha_i$.
To compute $F_{\S_{i+1},\SS}$, we apply the naturality relation~\eqref{it:d-F} to
the diffeomorphism $L_{i+1} \colon \S_{i+1} \to \S_{i+1}$ and the framed sphere $l_i$.
As $L_{i+1}(l_i) = l_{i+1}$ and $(L_{i+1})^{l_i}$ is isotopic to $\Id_{\S_i}$
after the natural identifications $\S_{i+1}(l_i) \approx \S_i$ and $\S_{i+1}(l_{i+1}) \approx \S_i$,
we obtain that $F_{\S_{i+1},l_i} = \a_{i+1} \circ \rho_{i+1}(L_{i+1})$.
So
\[
F_{\S_{i+1},\SS} \circ F_{\S_i,\SS'} = \a_{i+1} \circ \rho_{i+1}(L_{i+1}) \circ \omega_i.
\]

\begin{figure}
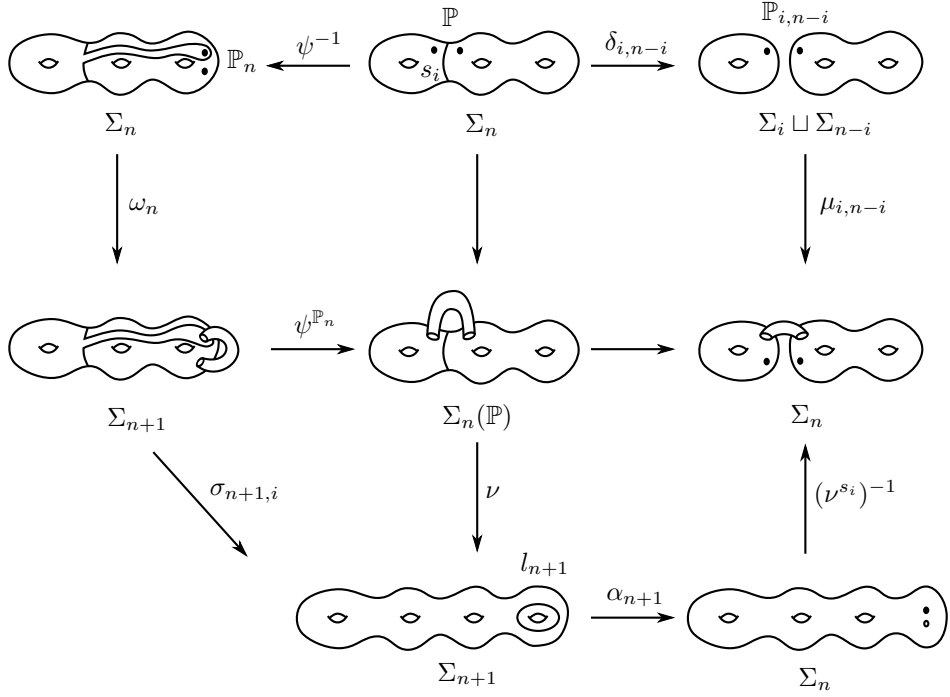
\caption{Given a TQFT, this figure illustrated why the condition
$\a_{n+1} \circ \rho_{n+1}(\sigma_{n+1,i}) \circ \omega_n = \mu_{i,n-i} \circ \delta_{i,n-i}$ holds.
Note that both $\psi$ and $\nu^{s_i}$ are isotopic to~$\Id_{\S_n}$.}
\label{fig:relations}
\end{figure}

Finally, we prove that
$\a_{n+1} \circ \rho_{n+1}(\sigma_{n+1,i}) \circ \omega_n = \mu_{i,n-i} \circ \delta_{i,n-i}$
holds for every $n \in \N$ and $0 \le i \le n$. This also follows from relation~\eqref{it:commut}, applied
to $\S_n$ with $\SS = s_i$, and~$\SS' = \bP$ being a framed pair of points in $\S_n$ such that there
is exactly one point of~$\bP_n$ on each side of~$s_i$ very close to it. Then
\[
F_{M(\SS),\SS'} \circ F_{M,\SS} = F_{\S_n(s_i),\bP} \circ F_{\S_n,s_i} = \mu_{i,n-i} \circ \delta_{i,n-i},
\]
where we identify $\S_n(s_i,\bP)$ with $\S_n$ as in Figure~\ref{fig:relations}; i.e.,
the connected sum tube along $s_i$ between the $i$-th and $(i+1)$-st $S^1 \times S^1$ summand
of $\S_n$ is mapped to the $D^1 \times S^1$ glued during the surgery along~$\bP$.

Consider now
\[
F_{M(\SS'),\SS} \circ F_{M,\SS'} = F_{\S_n(\bP),s_i} \circ F_{\S_n,\bP}.
\]
To compute $F_{\S_n,\bP}$, we need to identify the pair $(\S_n,\bP)$ with
the model framed pair of points $(\S_n,\bP_n)$.
Let~$\psi \in \Diff(\S_n)$ be such that $\psi^{-1}$ maps $\bP$ to~$\bP_n$,
and acts via a finger move on $s_i$. In particular, $\psi^{-1}$ is isotopic to $\Id_{\S_n}$.
The identification between~$\S_n(\bP_n)$ and~$\S_{n+1}$ defined in Section~\ref{sec:surfaces}
maps~$\psi^{-1}(s_i)$ to the connected sum~$s_i \# m_{n+1}$, where the connected sum depends on the choice
of finger move. If we apply relation~\eqref{it:d-F} of Definition~\ref{def:relations} to $\S_n$, $\bP$, and $\psi^{-1}$,
then we obtain that
\[
F_{\S_n,\bP} = \psi^{\bP_n}_* \circ \omega_n \circ \psi^{-1}_* = \psi^{\bP_n}_* \circ \omega_n
\]
as $\psi^{-1}_* = \Id_{F(\S_n)}$ since $\psi^{-1} \in \Diff_0(\S_n)$; see
the left-hand square in Figure~\ref{fig:relations}.

To compute $F_{\S_n(\bP),s_i}$, we identify the pair $(\S_n(\bP),s_i)$ with the model
non-separating framed circle $(\S_{n+1}, l_{n+1})$.
Consider the framed pairs of points $b(l_{n+1})$ in $\S_{n+1}(l_{n+1}) \approx \S_n$
and $b(s_i)$ in $\S_n(\bP,s_i) \approx \S_n$. By the homogeneity of $\S_n$,
there is a diffeomorphism $\nu_0 \in \Diff_0(\S_n)$ such that $\nu_0 \circ b(s_i) = b(l_{n+1})$,
and let $\nu := \nu_0^{b(s_i)}$. Then $\nu(s_i) = l_{n+1}$ and $\nu^{s_i} = \nu_0$ is isotopic to~$\Id_{\S_n}$.
Hence, by relation~\eqref{it:d-F},
\[
F_{\S_n(\bP), s_i} = (\nu^{s_i})^{-1}_* \circ \a_{n+1} \circ \nu_* = \a_{n+1} \circ \nu_*;
\]
see the bottom square in Figure~\ref{fig:relations}.
Consequently,
\[
F_{\S_n(\bP),s_i} \circ F_{\S_n,\bP} = \a_{n+1} \circ \nu_* \circ \psi^{\bP_n}_* \circ \omega_n.
\]

What remains is to show that $\sigma_{n+1,i} = \nu \circ \psi^{\bP_n}$. Note that
\[
\nu = a_{n+1} \circ \left(\psi^{\bP_n}\right)^{-1} \circ h_{b(\bP), s_i},
\]
where $h_{b(\bP), s_i}$ swaps the curves $b(\bP)$ and $s_i$, and $(h_{b(\bP), s_i})^{b(\bP), s_i}$
is isotopic to the identity of $\S_i \sqcup \S_{n-i}$. Hence
\[
\nu \circ \psi^{\bP_n} = a_{n+1} \circ \left(\psi^{\bP_n}\right)^{-1} \circ h_{b(\bP), s_i} \circ \psi^{\bP_n}.
\]
The conjugate $d := (\psi^{\bP_n})^{-1} \circ h_{b(\bP), s_i} \circ \psi^{\bP_n}$ swaps the curves
$(\psi^{\bP_n})^{-1}(b(\bP)) = m_{n+1}$ and $(\psi^{\bP_n})^{-1}(s_i) = s_i \# m_{n+1}$,
and $d^{m_{n+1}, s_i \# m_{n+1}}$ is isotopic to the identity of $\S_{n+1}(m_{n+1}, s_i \# m_{n+1})$,
hence $d = h_{n+1,i}$. It follows that
\[
\nu \circ \psi^{\bP_n} = a_{n+1} \circ h_{n+1,i} = \sigma_{n+1,i}.
\]
This concludes the proof of Proposition~\ref{prop:AF}.
\end{proof}

A mapping class group representation on a split \gnf-algebra automatically satisfies some additional relations
that we will need in the classification of (2+1)-dimensional TQFTs:

\begin{lem} \label{lem:simplify}
Let $\rho_i \colon \M_i \to \text{Aut}(A_i)$ be a mapping class group representation on
the split \gnf-algebra $(\A,\alpha,\omega)$.
Then the map~$\a_i$ is $\MCG(\S_i,l_i)$-equivariant, $\omega_i$ is $\MCG(\S_i,\bP_i)$-equivariant, and
\begin{enumerate}
\item \label{it:t-omega} $\rho_{i+1}(t_{i+1}) \circ \omega_i = \omega_i$,
\item \label{it:alpha-rho} $\alpha_i \circ \rho(r_i)  = \alpha_i$,
\item \label{it:omega-omega} $\rho_{i+2}(S_{i+2}) \circ \omega_{i+1} \circ \omega_i = \omega_{i+1} \circ \omega_i$ for $i \in \N$,
\item \label{it:alpha-alpha} $\a_{i-1} \circ \a_i \circ \rho_i(L_i) = \a_{i-1} \circ \a_i$ for $i > 0$,
\item \label{it:omega-mu} $\rho_{n+1}(h_{n+1,i}) \circ \omega_n \circ \mu_{i,j} = \omega_n \circ \mu_{i,j}$ for $0 \le i \le n$,
\item \label{it:delta-alpha} $\d_{i,j} \circ \a_n \circ \rho_n(u_{n,i}) = \d_{i,j} \circ \a_n$,
\item \label{it:alpha-omega} $\a_{2} \circ \rho_{2}(L_{2}) \circ \omega_1 = \omega_{0} \circ \a_1$,
\end{enumerate}
where $n= i+j$, and
\begin{itemize}
\item $t_i(m_i) = -m_i$, and $t_i^{m_i} \in \Diff_0(\S_{i-1})$,
\item $r_i(l_i) = -l_i$, and $(r_i)^{l_i} \in \Diff_0(\S_{i-1})$.
\item $S_{i+2} \in \Diff(\S_{i+2})$ swaps~$m_{i+1}$ and~$m_{i+2}$, and $S_{i+2}^{m_{i+1},m_{i+2}} \in \Diff_0(\S_i)$,
\item $u_{n,i} \in \Diff(\S_n)$ swaps $s_i \# l_n$ and~$l_n$,
and $(u_{n,i})^{s_i \# l_n, l_n}$ is isotopic to the identity.
\end{itemize}
\end{lem}

%

\begin{proof}
The $\MCG(\S_{i+1}, m_{i+1})$-equivariance of~$\omega_i$
is equivalent to
\[
\rho_1(\tau_m)(w) = w
\]
from condition~\eqref{it:condw} of Definition~\ref{def:repr},
where $\tau_m \in \Diff(\S_1)$
is a right-handed Dehn twist about the meridian~$m_1$. This follows from
the $\MCG(\S_i \sqcup \S_1, \bP_{i,1})$-equi\-va\-ri\-ance of~$\mu_{i,1}$ and the fact that
$\omega_i(x) = \mu_{i,1}(x \otimes w)$.
Indeed, for an arbitrary diffeomorphism $d \in \Diff(\S_{i+1},m_{i+1})$,
we have $d^{m_{i+1}} \in \Diff(\S_i,\bP_i)$.
The $\MCG(\S_{i+1}, m_{i+1})$-equivariance of~$\omega_i$ translates to
\begin{equation} \label{eqn:omegsimpl1}
\rho_{i+1}(d) \circ \mu_{i,1}(x \otimes w) = \mu_{i,1}(\rho_i(d^{m_{i+1}})(x) \otimes w).
\end{equation}
But $d$ fixes the isotopy class of~$s_i$, and we can isotope it in $\Diff(\S_{i+1},m_{i+1})$
such that $d(s_i) = s_i$ as framed spheres. Then~$d^{s_i}$ is isotopic to~$\tau_m^k$
in the torus component of $\S_{i+1}(s_i)$ containing~$m_{i+1}$ for some $k \in \Z$,
and is isotopic to $d^{m_{i+1}}$ in the other component.
By the $\MCG(\S_i \sqcup \S_1, \bP_{i,1})$-equivariance of~$\mu_{i,1}$, we get that
\begin{equation} \label{eqn:omegsimpl2}
\rho_{i+1}(d) \circ \mu_{i,1}(x \otimes w) = \mu_{i,1}\left(\rho_i(d^{m_{i+1}})(x) \otimes \rho_1\left(\tau_m^k\right)(w)\right).
\end{equation}
Comparing equations~\eqref{eqn:omegsimpl1} and~\eqref{eqn:omegsimpl2}, we obtain that
\[
\mu_{i,1}\left(\rho_i(d^{m_{i+1}})(x) \otimes \rho_1\left(\tau_m^k\right)(w)\right) =
\mu_{i,1}(\rho_i(d^{m_{i+1}})(x) \otimes w).
\]
If we consider this for $i = 0$, $x = 1 \in A_0$, and $d = \tau_m \in \Diff(\S_1)$,
then $k = 1$, and we obtain that $\rho_1(\tau_m)(w) = w$.
In the opposite direction, $\rho_1(\tau_m)(w) = w$ and equation~\eqref{eqn:omegsimpl2} together
imply equation~\eqref{eqn:omegsimpl1}.

Similarly, the $\MCG(\S_i,l_i)$-equivariance of~$\a_i$
is equivalent to
\[
\lambda \circ \rho_1(\tau_l) = \lambda
\]
from condition~\eqref{it:condl} of Definition~\ref{def:repr},
where $\tau_l \in \Diff(\S_1)$ is a right-handed Dehn twist about the longitude~$l_1$.
This follows from the $\MCG(\S_{i+j}, s_i)$-equi\-vari\-ance of $\delta_{i,j}$, together with
the fact that
\[
\a_i = (\Id_{A_{i-1}} \otimes \l) \circ \d_{i-1,1}.
\]

Now consider property~\eqref{it:t-omega}; i.e.,
\[
\rho_{i+1}(t_{i+1}) \circ \omega_i = \omega_i,
\]
where $t_{i+1}(m_{i+1}) = -m_{i+1}$ and $t_{i+1}^{m_{i+1}} \in \Diff_0(\S_i)$.
Since $t_{i+1}$ fixes $s_i$ and $t_{i+1}^{s_i}$ is isotopic to $\Id_{\S_i} \sqcup t_1$,
we can apply the $\MCG(\S_i \sqcup \S_1, \bP_{i,1})$-equivariance of $\mu_{i,1}$ to obtain
\[
\rho_{i+1}(t_{i+1}) \circ \omega_i(x) = \rho_{i+1}(t_{i+1}) \circ \mu_{i,1}(x \otimes w) =
\mu_{i,1}(x \otimes \rho_1(t_1)(w)).
\]
In particular, property~\eqref{it:t-omega} is equivalent to
\[
\mu_{i,1}(x \otimes \rho_1(t_1)(w)) = \mu_{i,1}(x \otimes w)
\]
for every $i \in \N$ and $x \in A_i$. In particular, if we take~$i = 0$ and~$x = 1$,
we obtain
\begin{equation} \label{eqn:t}
\rho_1(t_1)(w) = w
\end{equation}
from condition~\eqref{it:condw} of Definition~\ref{def:repr},
and this is clearly also sufficient, hence equivalent to property~\eqref{it:t-omega}.

Now look at property~\eqref{it:alpha-rho}; i.e,
\[
\alpha_i \circ \rho(r_i)  = \alpha_i,
\]
where $r_i(l_i) = -l_i$ and $r_i^{l_i} \in \Diff_0(\S_{i-1})$.
By Lemma~\ref{lem:splitting}, this is equivalent to
\[
(\Id_{A_{i-1}} \otimes \l) \circ \d_{i-1,1} \circ \rho_i(r_i) = (\Id_{A_{i-1}} \otimes \l) \circ \d_{i-1,1}.
\]
Using the $\MCG(\S_i,s_{i-1})$-equivariance of $\d_{i-1,1}$ and that $r_i^{s_{i-1}} \approx \Id_{\S_{i-1}} \sqcup r_1$,
this is further equivalent to
\[
(\Id_{A_{i-1}} \otimes (\l \circ \rho_1(r_1))) \circ \d_{i-1,1} = (\Id_{A_{i-1}} \otimes \l) \circ \d_{i-1,1}.
\]
Notice that $r_1 = t_1$. If we set  $i=1$ and apply $\tau \otimes \Id_{A_1}$ to both sides,
we obtain the necessary and sufficient condition
\begin{equation} \label{eqn:r}
\l \circ \rho_1(t_1) = \l
\end{equation}
from condition~\eqref{it:condl} of Definition~\ref{def:repr}.

Next, consider property~\eqref{it:omega-omega}; i.e.,
\[
\rho_{i+2}(S_{i+2}) \circ \omega_{i+1} \circ \omega_i = \omega_{i+1} \circ \omega_i,
\]
where $S_{i+2}$ swaps $m_{i+1}$ and $m_{i+2}$, and $S_{i+2}^{m_{i+1},m_{i+2}} \in \Diff_0(\S_i)$.
By Lemma~\ref{lem:splitting} and the associativity of~$\mu$, this is equivalent to
\[
\begin{split}
\rho_{i+2}(S_{i+2}) \circ \mu_{i+1,1}(\mu_{i,1}(x \otimes w) \otimes w) &=
\rho_{i+2}(S_{i+2}) \circ \mu_{i,2}(x \otimes \mu_{1,1}(w \otimes w)) \\
&= \mu_{i,2}(x \otimes \mu_{1,1}(w \otimes w))
\end{split}
\]
for every $x \in A_i$. Since $\mu_{i,2}$ is $\MCG(\S_i \sqcup \S_2, \bP_{i,2})$-equivariant and $S_{i+2}$
fixes~$s_i$ as a framed sphere, in fact, $S_{i+2}^{s_i} = \Id_{\S_i} \sqcup S_2$, this condition can be expressed as
\[
\mu_{i,2}(x \otimes \rho_2(S_2) \circ \mu_{1,1}(w \otimes w)) = \mu_{i,2}(x \otimes \mu_{1,1}(w \otimes w))
\]
for every $i \in \N$ and $x \in A_i$. In particular, if we set $i = 0$ and $x = 1$,
we obtain the necessary and sufficient condition
\[
\rho_2(S_2) \circ \mu_{1,1}(w \otimes w) = \mu_{1,1}(w \otimes w).
\]
Now consider the diffeomorphism $d := \iota_2 \circ S_2 \circ \iota_2$.
This swaps the meridians~$m_0$ and~$m_1$ of $\S_2$, but fixes~$m_2$, hence lies in $\Diff(\S_2,m_2)$.
Furthermore, $d^{m_2}$ is isotopic to the automorphism~$t_1$ of the torus, and we have $\rho_1(t_1)(w) = w$
by condition~\eqref{it:condw} of Definition~\ref{def:repr}. Hence,
\[
\rho_2(d) \circ \omega_1(w) = \omega_1(\rho_1(t_1)(w)) = \omega_1(w) = \mu_{1,1}(w \otimes w).
\]
On the other hand, $\rho_2(d) = *_2 \circ \rho_2(S_2) \circ *_2$, hence the left-hand side
of the above equation is
$
*_2 \circ \rho_2(S_2) \circ *_2 \circ \omega_1(w).
$
But $*_1 = \Id_{A_1}$ since $\iota_1$ is isotopic to~$\Id_{T^2}$, hence
\[
*_2 \circ \omega_1(w) = *_2 \circ \mu_{1,1}(w \otimes w) = \mu_{1,1}(w^* \otimes w^*) = \mu_{1,1}(w \otimes w).
\]
It follows that
\[
\rho_2(S_2) \circ \mu_{1,1}(w \otimes w) = *_2 \circ \mu_{1,1}(w \otimes w) = \mu_{1,1}(w \otimes w),
\]
establishing property~\eqref{it:omega-omega}.

Similarly, we can prove property~\eqref{it:alpha-alpha}; i.e.,
\[
\a_{i-1} \circ \a_i \circ \rho_i(L_i) = \a_{i-1} \circ \a_i,
\]
where $L_i$ swaps $l_{i-1}$ and $l_i$, and $L_i^{l_{i-1},l_i} \in \Diff_0(\S_{i-2})$.
By the coassociativity of~$\d$, and since $\d_{i-2,2}$ is $\MCG(\S_i,s_{i-2})$-equivariant and $L_i^{s_{i-2}} = \Id_{\S_{i-2}} \sqcup L_2$,
the left-hand side is
\[
\begin{split}
(\Id_{A_{i-2}} \otimes \l) \circ \d_{i-2,1} \circ (\Id_{A_{i-1}} \otimes \l) \circ \d_{i-1,1} \circ \rho_i(L_i) &= \\
(\Id_{A_{i-2}} \otimes \l) \circ (\Id_{A_{i-2}} \otimes \Id_{A_1} \otimes \l) \circ (\d_{i-2,1} \otimes \Id_{A_1}) \circ \d_{i-1,1} \circ \rho_i(L_i) &= \\
(\Id_{A_{i-2}} \otimes \l \otimes \l) \circ (\Id_{A_{i-2}} \otimes \d_{1,1}) \circ \d_{i-2,2} \circ \rho_i(L_i) &= \\
(\Id_{A_{i-2}} \otimes \l \otimes \l) \circ (\Id_{A_{i-2}} \otimes \d_{1,1}) \circ (\Id_{A_{i-2}} \otimes \rho_2(L_2)) \circ \d_{i-2,2}.
\end{split}
\]
Since $L_2$ is isotopic to $\iota_2$, we have $\rho_2(L_2) = *_2$, and property~\eqref{it:alpha-alpha} is equivalent to
\[
\begin{split}
(\Id_{A_{i-2}} \otimes \l \otimes \l) \circ (\Id_{A_{i-2}} \otimes \d_{1,1}) \circ (\Id_{A_{i-2}} \otimes *_2) \circ \d_{i-2,2} &= \\
(\Id_{A_{i-2}} \otimes \l \otimes \l) \circ (\Id_{A_{i-2}} \otimes \d_{1,1}) \circ \d_{i-2,2}.
\end{split}
\]
In particular, if we set $i=2$ and apply $\tau \otimes \Id_{A_2}$ to both sides,
we get the necessary and sufficient condition
\[
(\l \otimes \l) \circ \d_{1,1} \circ *_2 = (\l \otimes \l) \circ \d_{1,1}.
\]
However, since $\d_{1,1} \circ *_2 = T \circ \d_{1,1}$, and because $*_1 = \Id_{A_1}$ as $\iota_1$ is isotopic to
the identity, the above equation automatically follows from the \gnf-algebra axioms, and from the
$\MCG(\S_i,s_{i-2})$-equivariance of~$\d_{i-2,2}$.

We now prove property~\eqref{it:omega-mu}; i.e,
\[
\rho_{n+1}(h_{n+1,i}) \circ \omega_n \circ \mu_{i,j}(x \otimes y) = \omega_n \circ \mu_{i,j}(x \otimes y),
\]
where $h_{n+1,i}$ swaps $s_i \# m_{n+1}$ with $m_{n+1}$, and
$h^{s_i \# m_{n+1}, m_{n+1}}$ is isotopic to the identity.
Using our formula for $\omega_n$, the above equation becomes equivalent to
\[
\rho_{n+1}(h_{n+1,i}) \circ \mu_{n,1}(\mu_{i,j}(x \otimes y) \otimes w) = \mu_{n,1}(\mu_{i,j}(x \otimes y) \otimes w).
\]
As $h_{n+1,i}$ fixes~$s_i$ and $h_{n+1,i}^{s_i} = \Id_{\S_i} \sqcup h_{j+1,0}$,
using the associativity of~$\mu$ and the $\MCG(\S_{n+1},s_i)$-equivariance
of~$\mu_{i,j+1}$, this is further equivalent to
\[
\mu_{i,j+1}(x \otimes \rho_{j+1}(h_{j+1,0}) \circ \mu_{j,1}(y \otimes w)) = \mu_{i,j+1}(x \otimes \mu_{j,1}(y \otimes w)).
\]
As $s_0 \# m_{j+1}$ is isotopic to $m_{j+1}$, the diffeomorphism $h_{j+1,0}$ is isotopic to the
identity, so property~\eqref{it:omega-mu} follows.

Property~\eqref{it:delta-alpha} is dual to property~\eqref{it:omega-mu}. It states that
\[
\d_{i,j} \circ \a_n \circ \rho_n(u_{n,i}) = \d_{i,j} \circ \a_n,
\]
where $u_{n,i} \in \Diff(\S_n)$ swaps $s_i \# l_n$ and~$l_n$,
and $u^{s_i \# l_n, l_n}$ is isotopic to the identity.
Using the coassociativity of~$\d$, the left-hand side becomes
\[
\begin{split}
\d_{i,j} \circ (\Id_{A_{i+j}} \otimes \l) \circ \d_{i+j,1} \circ \rho_n(u_{n,i}) &= \\
(\Id_{A_i} \otimes \Id_{A_j} \otimes \l) \circ (\d_{i,j} \otimes \Id_{A_1}) \circ \d_{i+j,1} \circ \rho_n(u_{n,i}) &= \\
(\Id_{A_i} \otimes \Id_{A_j} \otimes \l) \circ (\Id_{A_i} \otimes \d_{j,1}) \circ \d_{i,j+1} \circ \rho_n(u_{n,i}).
\end{split}
\]
Note that $u_{n, i}$ fixes~$s_i$ and $u_{n,i}^{s_i} = \Id_{\S_i} \sqcup u_{j+1,0}$. Hence,
by the $\MCG(\S_{n+1},s_i)$-equivariance of $\d_{i,j+1}$, the left-hand side of equation~\eqref{it:delta-alpha} further equals
\[
\begin{split}
(\Id_{A_i} \otimes \Id_{A_j} \otimes \l) \circ (\Id_{A_i} \otimes \d_{j,1} \circ \rho_{j+1}(u_{j+1,0})) \circ \d_{i,j+1} &= \\
(\Id_{A_i} \otimes [\a_{j+1} \circ \rho_{j+1}(u_{j+1,0}))]) \circ \d_{i,j+1}.&
\end{split}
\]
But $s_0 \# l_n$ is isotopic to~$l_n$, hence $u_{j+1,0}$ is isotopic to the identity.
Analogously, the right-hand side of equation~\eqref{it:delta-alpha} is
\[
(\Id_{A_i} \otimes \a_{j+1}) \circ \d_{i,j+1},
\]
and so equation~\eqref{it:delta-alpha} follows.

Finally, consider property~\eqref{it:alpha-omega}. More generally, consider
\[
\a_{i+1} \circ \rho_{i+1}(L_{i+1}) \circ \omega_i = \omega_{i-1} \circ \a_i.
\]
We first remark that if we apply $\a_i$ to both sides, the resulting equation follows from
property~\eqref{it:alpha-alpha} and the split \gnf-algebra axioms.
Secondly, we prove that this automatically holds on~$\im(\mu_{i-1,1})$,
and hence for~$i = 1$. Indeed, suppose that $x = \mu_{i-1,1}(a \otimes b)$. Then
\[
\begin{split}
\a_{i+1} \circ \rho_{i+1}(L_{i+1}) \circ \omega_i(x) = \a_{i+1} \circ \mu_{i-1,2}(a \otimes  \rho_2(L_2) \circ \mu_{1,1}(b \otimes w)) &= \\
\a_{i+1} \circ \mu_{i-1,2}(a \otimes \mu_{1,1}(w \otimes  b)) &= \\
(\Id_{A_i} \otimes \l) \circ \d_{i,1} \circ \mu_{i-1,2}(a \otimes \mu_{1,1}(w \otimes b)) &= \\
(\Id_{A_i} \otimes \l) \circ (\mu_{i-1,1} \otimes \Id_{A_1}) \circ (\Id_{A_{i-1}} \otimes \d_{1,1})(a \otimes \mu_{1,1}(w \otimes b)) &= \\
(\mu_{i-1,1} \otimes \l) \circ (a \otimes [\d_{1,1} \circ \mu_{1,1}(w \otimes b)]) &= \\
(\mu_{i-1,1} \otimes \l) \circ (a \otimes [(\mu_{1,0} \otimes \Id_{A_1}) \circ (\Id_{A_1} \otimes \d_{0,1})(w \otimes b)]) &= \\
(\mu_{i-1,1} \otimes \l)(a \otimes \mu_{1,0}(w \otimes b_{(1)}) \otimes b_{(2)}) &= \\
\l(b_{(2)})(a \cdot w \cdot b_{(1)}).
\end{split}
\]
Here we used that~$\mu_{i-1,2}$ is $\MCG(\S_{i-1} \sqcup \S_2, \bP_{i-1,2})$-equivariant,
$L_{i+1}^{s_{i-1}} = \Id_{\S_{i-1}} \sqcup L_2$,
that~$\rho_2(L_2) = *_2$, and the Frobenius condition twice.
Furthermore, $\d_{0,1}(b) = b_{(1)} \otimes b_{(2)}$ in sumless Sweedler notation,
and $\cdot$ stands for the algebra multiplication~$\mu$.
On the other hand, the right-hand side of equation~\eqref{it:alpha-omega} becomes
\[
\begin{split}
\omega_{i-1} \circ \a_i \circ \mu_{i-1,1}(a \otimes b) &= \\
\omega_{i-1} \circ (\Id_{A_{i-1}} \otimes \l) \circ \d_{i-1,1} \circ \mu_{i-1,1}(a \otimes b) &= \\
\omega_{i-1} \circ (\Id_{A_{i-1}} \otimes \l) \circ (\mu_{i-1,0} \otimes \Id_{A_1}) \circ (\Id_{A_{i-1}} \otimes \d_{0,1})(a \otimes b) &= \\
\omega_{i-1} \circ (\mu_{i-1,0} \otimes \l)(a \otimes \d_{0,1}(b)) &= \\
[(\mu_{i-1,0} \otimes \l)(a \otimes b_{(1)} \otimes b_{(2)})] \cdot w &= \\
\l(b_{(2)})(a \cdot b_{(1)} \cdot w).
\end{split}
\]
The claim follows once we observe that $b_{(1)} \cdot w \in A_1$, hence
$ b_{(1)} \cdot w =(b_{(1)} \cdot w)^* = w^* \cdot b_{(1)}^* = w \cdot b_{(1)}$
since $*_0 = \Id_{A_0}$ and $*_1 = \Id_{A_1}$.
\end{proof}

\section{The classification of (2+1)-dimensional TQFTs} \label{sec:proof2+1}

Both (2+1)-dimensional TQFTs and J-algebras form symmetric monoidal
categories, with morphisms the monoidal natural transformations.
Durhuus and Jonsson~\cite{DJ} defined the notion of direct sum of TQFTs. Given $(n+1)$-dimensional
TQFTs~$F_1$ and~$F_2$, they let
\[
(F_1 \oplus F_2)(M) = F_1(M) \oplus F_2(M)
\]
for every \emph{connected} $n$-manifold~$M$,
while in general $(F_1 \oplus F_2)(M)$ is the tensor product of the vector spaces assigned to the components of~$M$.
To a connected cobordism~$W$, they assign the direct sum $F_1(W) \oplus F_2(W)$, and to a disconnected cobordism
the tensor product of the values of the components.

In this section, we shall prove the following
classification of (2+1)-dimensional TQFTs,
which is Theorem~\ref{thm:2+1} from the introduction.

\begin{thm*}
There is an equivalence between the
symmetric monoidal category of (2+1)-dimensional TQFTs and $\JAlg$.
\end{thm*}

In Section~\ref{sec:assignment}, we saw how to assign a J-algebra $J(F)$ to a TQFT~$F$.
Now suppose that we are given a J-algebra
$\AA = (\A,\a,\omega,\{\rho_i \,\colon\, i \in \N\})$. Then we associate to it a TQFT $F := T(\AA)$ as follows.
By Theorem~\ref{thm:TQFT}, it suffices to construct a symmetric monoidal functor $F \colon \Man_2 \to \Vect$
and maps~$F_{M,\SS}$ for any framed sphere~$\SS$ in a surface~$M$. The following constructions are
all determined by the naturality of the TQFT under diffeomorphisms.
After constructing the groups~$F(M)$ and the surgery maps~$F_{M,\SS}$, we check what algebraic properties
relations~\eqref{it:isot}--\eqref{it:0-sphere} of Definition~\ref{def:relations} translate to.

First, we construct $F(M)$ for a surface~$M$ with $k$ components of genera $g_1 > \dots > g_r$
with multiplicities $n_1, \dots, n_r$, respectively. In particular, $n_1 + \dots + n_r = k$, and
we denote the vector
\[
(\underbrace{g_1, \dots, g_1}_{n_1}, \dots, \underbrace{g_r, \dots, g_r}_{n_r})
\]
of genera by~$\ug$.
Let
\[
\S_{\ug} = \coprod_{i = 1}^r \coprod_{j = 1}^{n_i} \S_{g_i}.
\]
We follow the same scheme of Kan extension as one dimension lower in Section~\ref{sec:1+1}. In particular, let
\[
A_{\ug} = A_{g_1}^{\otimes n_1} \otimes \dots \otimes A_{g_r}^{\otimes n_r},
\]
and $F(M)$ is defined to be the set of those elements~$v$ of
\[
\prod_{\phi \in \Diff(\S_{\ug},M)} A_{\ug}
\]
for which
\[
v(\phi') = ((\phi')^{-1} \circ \phi)\cdot v(\phi)
\]
for every $\phi$, $\phi' \in \Diff(\S, M)$.
Note that here $(\phi')^{-1} \circ \phi \in \Diff(\S_{\ug})$, which acts on~$A_{\ug}$
via the representations~$\rho_i$ and permuting the factors with the same genus. More precisely,
the action of $\Diff(\S_{\ug})$ on~$A_{\ug}$ factors through the action of
\[
\MCG(\S_{\ug}) \cong \prod_{i = 1}^r \M_{g_i} \text{ Wr}_{\{1,\dots,n_i\}} \text{ } S_{n_i},
\]
where $\text{Wr}$ denotes the unrestricted wreath product,
the group~$\M_{g_i}$ acts on~$A_{g_i}$ via~$\rho_{g_i}$,
while $S_{n_i}$ permutes the factors of~$A_{g_i}^{\otimes n_i}$.

Suppose that $M$ and~$M'$ are diffeomorphic surfaces; i.e., they have the same number
of components~$k$ with genera $g_i = g_i'$ and multiplicities $n_i = n_i'$ for every $i \in \{1,\dots,r\}$,
and let $d \in \Diff(M,M')$.
Given an element $v \in F(M)$ and $\phi \in \Diff(\S_{\ug}, M)$, we let
\begin{equation} \label{eqn:diffaction}
[F(d)(v)](d \circ \phi) = v(\phi).
\end{equation}

If $M$ and $N$ are surfaces of diffeomorphism types $\S_{\ug}$ and $\S_{\uh}$, respectively, then
we define the natural isomorphism
\[
\Phi_{M,N} \colon F(M) \otimes F(N) \to F(M \sqcup N)
\]
as follows.
Let $\phi \in \MCG(\S_{\ug}, M)$ and $\psi \in \MCG(\S_{\uh},N)$.
We let $\ug \sqcup \uh$ be the vector obtained by putting the coordinates of
$\ug$ and $\uh$ in nonincreasing order. Then $\S_{\ug} \sqcup \S_{\uh}$ is of diffeomorphism type $\S_{\ug \sqcup \uh}$.
The diffeomorphism $\phi \Box \psi \in \MCG(\S_{\ug \sqcup \uh}, M \sqcup N)$
is defined as follows. If $g$ is a coordinate of $\ug$ of multiplicity~$m$ and of $\uh$ of multiplicity~$n$,
then for $(x,i) \in \S_g \times \{1,\dots,m\}$ we let $(\phi \Box \psi)(x,i) = \phi(x,i)$,
and for $(x,j) \in \S_g \times \{m+1,\dots,m+l\}$ we have $(\phi \Box \psi)(x,j) = \psi(x,j-m)$.
There is an analogous isomorphism $i_{\ug,\uh} \colon A_{\ug} \otimes A_{\uh} \to A_{\ug \sqcup \uh}$.
If $a \in F(M)$ and $b \in F(N)$, then we let $\Phi_{M,N}(a \otimes b) = a \sqcup b \in F(M \sqcup N)$
where
\[
(a \sqcup b)(\phi \Box \psi) = i_{\ug,\uh}(a(\phi) \otimes b(\psi)) \in A_{\ug \sqcup \uh}.
\]
We leave it to the reader to check that the assignment $F \colon \Man_2 \to \Vect_\F$
defined above is a symmetric monoidal functor.

We now define the surgery maps~$F_{M,\SS}$ for a surface~$M$ of diffeomorphism type~$\S_{\ug}$,
equipped with a framed sphere~$\SS \subset M$.

First, suppose that~$\SS = 0$; then $M(\SS) = M \sqcup S^2$.
If $M$ is of diffeomorphism type~$\S_{\ug}$, then $M(\SS)$ is of type $\S_{(\ug,0)}$,
where $(\ug, 0)$ is $\ug$ with an extra~$0$ at the end.
Let $i_{\ug, 0} \colon \S_{(\ug, 0)} \to \S_{\ug} \sqcup S^2$ be the
natural identification that maps the last component of $\S_{(\ug, 0)}$ to $S^2$.
Given $\phi \in \Diff(\S_{\ug}, M)$, let
\[
\phi_0 =  (\phi \sqcup \Id_{S^2}) \circ i_{\ug,0} \in \Diff\left(\S_{(\ug, 0)}, M(\SS)\right).
\]
For $v \in F(M)$, we let
\[
F_{M,0}(v)(\phi_0) = v(\phi) \otimes 1 \in A_{\ug} \otimes A_0,
\]
where $1 \in A_0$ is the image of~$1 \in \F$ under the map~$\eps$. The element $F_{M,0}(v)$ is independent of
the choice of~$\phi$.

Now suppose that $\SS \colon S^2 \hookrightarrow M$ is a framed $2$-sphere with image~$S \subset M$.
Then $M(\SS) = M \setminus S$. Choose a parametrization $\phi \in \Diff(\S_{\ug}, M)$ such that
$\phi|_{\S_{g_r} \times \{n_r\}} = \SS$, and let $\phi_\SS = \phi|_{\S_{\ug'}}$, where $\ug' = \ug \setminus \{(g_r, n_r)\}$.
Consider the map
\[
t_{\ug} \colon A_{\ug} \to A_{\ug'}
\]
defined on monomials by
\[
t_{\ug}(v_1 \otimes \dots \otimes v_k) = \tau(v_k) \cdot v_1 \otimes \dots \otimes v_{k-1},
\]
and extend it linearly.
For $v \in F(M)$, let
\[
F_{M,\SS}(v)(\phi_\SS) = t_{\ug}(v(\phi)).
\]
Again, this is well-defined; i.e., independent of the choice of~$\phi$.

Assume that~$\SS = \{s_-,s_+\}$ is a framed $0$-sphere. If~$s_-$ and~$s_+$ lie
in different components~$M_-$ and~$M_+$ of~$M$ of genera~$g_a$ and~$g_b$, respectively,
then let
\[
\begin{split}
q_- = (q_{g_a}, n_a) \in \S_- &:= \S_{g_a} \times \{n_a\}, \text{ and} \\
p_+ = (p_{g_b}, n_b) \in \S_+ &:= \S_{g_b} \times \{n_b\}. \\
\end{split}
\]
Choose a parametrization $\phi \in \Diff(\S_{\ug}, M)$ such that $\phi(q_-) = s_-$ and $\phi(p_+) = s_+$,
and such that $\phi$ preserves the framings.
Let $\S_{\ug}(q_-,p_+)$ be the result of surgery along the $0$-sphere $\{q_-,p_+\}$.
If~$n_{a,b}$ is the multiplicity of~$g_a + g_b$ in~$\ug$,
then we can identify $\S_{\ug}(q_-,p_+)$ with the canonical surface~$\S_{\ug'}$ for
\[
\ug' = \ug \setminus \{(g_a,n_a),(g_b,n_b)\} \cup \{(g_a + g_b, n_{a,b} + 1)\}.
\]
There is an induced parametrization
$\phi_\SS \colon \S_{\ug}(q_-,p_+) = \S_{\ug'} \to M(\SS)$ that is the connected sum
$(\phi|_{\S_-}) \# (\phi|_{\S_+})$ on $\S_- \# \S_+$,
and agrees with~$\phi$ on all the other components.
If~$v \in F(M)$ is an element such that $v(\phi)$ is a monomial
\[
\otimes_{i = 1}^r \otimes_{j = 1}^{n_i} v_{(i,j)},
\]
the integer $n_i'$ is the multiplicity of~$g_i$ in~$\ug'$ for $i \in \{1,\dots,r'\}$, and~$c$ is such that $g_c' = g_a + g_b$, then we
define $F_{M,\SS}(v)(\phi_\SS)$ as
\[
\left( \otimes_{i = 1}^{c - 1} \otimes_{j = 1}^{n_i'} v_{(i,j)} \right) \otimes
\left( \otimes_{j = 1}^{n_c} v_{(c,j)} \otimes \mu_{g_a, g_b}(v_{(a,n_a)}, v_{(b,n_b)}) \right) \otimes
\left( \otimes_{i = c + 1}^{r'} \otimes_{j = 1}^{n_i'} v_{(i,j)} \right).
\]
In other words, we omit~$v_{(a,n_a)}$ and~$v_{(b,n_b)}$ from~$v(\phi)$,
and insert their $\mu_{g_a, g_b}$-product in position $n_1' + \dots + n_c'$.
The element $F_{M,\SS}(v)$ defined above is independent of the choice of~$\phi$
since~$\mu_{g_a,g_b}$ is $\MCG(\S_{g_a} \sqcup \S_{g_b}, \PP_{g_a,g_b})$-equivariant.

If~$s_-$ and~$s_+$ lie in the same component~$M_s$ of~$M$, then let $g_a = g(M_s)$.
Consider the framed $0$-sphere~$\bP = \bP_{g_a} \times \{n_a\} \subset \S_{g_a} \times \{n_a\}$,
and choose a parametrization $\phi \in \Diff(\S_{\ug}, M)$ such that~$\phi \circ \bP = \SS$.
The surgered manifold~$M(\SS)$ is diffeomorphic to $\S_{\ug}(\bP)$, which in turn can be canonically
identified with $\S_{\ug'}$ for $\ug'$ obtained from~$\ug$ by removing a copy of~$g_a$ and inserting~$g_a + 1$.
The identification of $\S_{\ug}(\bP) \approx \S_{\ug'}$ is obtained by applying the identification
$\S_{g_\a}(\bP_{g_\a}) \approx \S_{g_\a+1}$ defined in Section~\ref{sec:surfaces}
to the $n_\a$-th component of $\S_{\ug}$ and the identity to all the other components.
By surgery, we obtain the parametrization
\[
\phi_\SS := \phi^\PP \,\colon\, \S_{\ug'} \approx \S_{\ug}(\bP)  \to M(\SS).
\]
Given an element $v \in F(M)$ such that $v(\phi) = \otimes_{i = 1}^r \otimes_{j = 1}^{n_i} v_{(i,j)}$, the element
$F_{M,\SS}(v)(\phi_\SS)$ is obtained by applying $\omega_{g_a}$ to~$v_{g_a, n_a}$.
The element $F_{M,\SS}(v)$ is independent of the choice of~$\phi$ since $\omega_{g_a}$
is $\MCG(\S_{g_a}, \PP_{g_a})$-equivariant by Lemma~\ref{lem:simplify}.

Now suppose that~$\SS$ is a framed $1$-sphere in~$M$, lying in a component~$M_s$ of genus~$g_a \in \ug$.
If~$\SS$ is non-separating, consider the curve $l = l_{g_a} \times \{n_a\} \subset \S_{\ug}$.
Then there is a diffeomorphism $\phi \colon \S_{\ug} \to M$
such that $\phi \circ l = \SS$. This is possible since any two non-separating simple closed curves
on a connected surface are ambient diffeomorphic (indeed, both $\S_{g_a} \setminus l_{g_a}$ and
$\S_s \setminus \SS$ are connected, twice punctured, genus $g_a - 1$ surfaces, hence they are diffeomorphic).
We obtain $\ug'$ by removing a copy of~$g_a$ and
replacing it by~$g_a - 1$. The surgered manifold~$M(\SS)$ is diffeomorphic to~$\S_{\ug}(l)$,
which is canonically identified with~$\S_{\ug'}$ by applying the identification $\S_{g_\a}(l_{g_\a}) \approx \S_{g_\a-1}$
defined in Section~\ref{sec:surfaces} to the $n_\a$-th component of $\S_{\ug}$ and the identity to all the other components.
Then let
\[
\phi_\SS := \phi^l \,\colon\, \S_{\ug'} \approx \S_{\ug}(l) \to M(\SS).
\]
If~$v \in F(M)$ is such that~$v(\phi)$ is of the form $\otimes_{i = 1}^r \otimes_{j = 1}^{n_i} v_{(i,j)}$,
then we obtain $F_{M,\SS}(v)(\phi_\SS)$ by applying $\alpha_{g_a}$ to the factor~$v_{g_a, n_a}$.
The map~$F_{M,\SS}$ is independent of the choice of~$\phi$ since $\a_{g_a}$ is $\MCG(\S_{g_a},l_{g_a})$-equivariant
by Lemma~\ref{lem:simplify}.

Finally, suppose that~$\SS$ separates $M_s$ into pieces of genera~$g_-$ on the negative side and~$g_+$ on the
positive side (in particular, $g_a = g_- + g_+$). Consider the framed circle
$c = s_{g_-} \times \{n_a\} \subset \S_{\ug} \times \{n_a\}$.
Then there is a diffeomorphism $\phi \colon \S_{\ug} \to M$ such that $\phi \circ c = \SS$.
Let~$\ug'$ be the vector obtained from~$\ug$ by removing~$g_a$ and inserting~$g_-$ and~$g_+$ to keep the sequence
of coordinates decreasing. There is a canonical diffeomorphism~$d_c \colon \S_{\ug}(c) \to \S_{\ug'}$ that
maps the components of $(\S_{g_a} \times \{n_a\})(c)$ to the last components of~$\S_{\ug}$ of genera~$g_-$
and~$g_+$, respectively. If $g_- = g_+$, then we map the part coming from the negative side of~$c$ as the second to last
such component, and the part coming from the positive side of~$c$ as the last component of the appropriate genus.
We define the map
\[
\phi_\SS := \phi^c \circ (d_c)^{-1} \colon \S_{\ug'} \to M(\SS).
\]
If $v(\phi)$ is of the form $\otimes_{i = 1}^r \otimes_{j = 1}^{n_i} v_{(i,j)}$,
then $F_{M,\SS}(v)(\phi_\SS)$ is obtained by applying the map~$\d_{g_-,g_+}$ to~$v_{g_a, n_a}$,
and then permuting the factors according to the diffeomorphism~$d_c$.
In this case, $F_{M,\SS}(v)$ is independent of the choice of~$\phi$ since~$\d_{g_-,g_+}$
is $\MCG(\S_{g_a}, s_{g_-})$-equivariant.

This concludes the construction of the vector spaces~$F(M)$ and maps~$F_{M,\SS}$.
By Theorem~\ref{thm:TQFT}, these completely determine the (2+1)-dimensional TQFT~$F$,
assuming they satisfy relations~\eqref{it:isot}--\eqref{it:0-sphere} of Definition~\ref{def:relations}.
We check these next.

\begin{prop} \label{prop:FA}
Let $\AA$ be a J-algebra.
Then the functor
\[
F = T(\AA) \colon \Man_2 \to \Vect
\]
and the maps $F_{M,\SS}$ constructed above
satisfy relations~\eqref{it:isot}--\eqref{it:0-sphere} of Definition~\ref{def:relations}
and diagram~\eqref{eqn:monoidal}.
\end{prop}

\begin{proof}
Relation~\eqref{it:isot} follows analogously to the (1+1)-dimensional case
and the fact that the $\Diff(\S_g)$-action on~$A_g$ factors through a~$\MCG(\S_g)$-action,
and it does not impose any additional algebraic restrictions.

Relation~\eqref{it:d-F} also follows analogously to the (1+1)-dimensional case, and requires no additional assumptions.
As an illustration, we check relation~\eqref{it:d-F} when~$M$ is a connected surface of genus~$g$,
and~$\SS$ is a non-separating $1$-sphere. In particular, $\ug = (g)$.
Choose a parametrization $\phi \in \Diff(\S_g, M)$ for which $\phi \circ l_g = \SS$,
and let $\phi_\SS \in \Diff(\S_{g-1}, M(\SS))$ be the induced parametrization.
Let $d \colon M \to M'$ be a diffeomorphism, $\SS' = d \circ \SS$, and choose an
element $v \in F(M)$. Then $v(\phi) \in A_g$, and, by equation~\eqref{eqn:diffaction} defining~$F(d^\SS)$, we have
\[
[F(d^\SS) \circ F_{M,\SS}(v)]\left(d^\SS \circ \phi_\SS \right) =  F_{M,\SS}(v)(\phi_\SS) = \a_g(v(\phi)) \in A_{g-1}.
\]
On the other hand,
\[
[F_{M',\SS'} \circ F(d)(v)]\left((d \circ \phi)^{\SS'} \right) = \a_g \left([F(d)(v)](d \circ \phi)\right) = \a_g(v(\phi)).
\]
The result follows once we observe that $d^\SS \circ \phi_\SS = (d \circ \phi)^{\SS'}$.

Now consider relation~\eqref{it:commut}. In particular, let~$\SS$ and~$\SS'$ be disjoint framed spheres in the surface~$M$.
The roles of~$\SS$ and~$\SS'$ are symmetric, and -- as in the (1+1)-dimensional case --
it is straightforward to check the relation when~$\SS = 0$ or $\SS$ is a framed $2$-sphere.
This leaves us with three cases depending on the dimensions of the two spheres.

First, suppose that both~$\SS$ and~$\SS'$ are framed $0$-spheres. Relation~\eqref{it:commut} is true if they
occupy distinct components of~$M$. There are four remaining subcases:
\begin{enumerate}
\item \label{it:s1} $\SS$ and~$\SS'$ occupy the same component~$M_s$ of~$M$,
\item \label{it:s2} $\SS$ intersects both $M_s$ and another component~$M_s'$, and $\SS'$ lies in~$M_s'$,
\item \label{it:s3} both~$\SS$ and~$\SS'$ intersect two components that coincide, namely~$M_s$ and~$M_s'$,
\item  \label{it:s4} $\SS$ intersects two components~$M_s$ and~$M_s'$, while~$\SS'$ intersects~$M_s'$ and~$M_s''$.
\end{enumerate}

Consider case~\eqref{it:s1}. Without loss of generality, we can assume that~$M$ is connected, as we can deal with multiple
components similarly to the (1+1)-dimensional case. Let~$C = b(\SS)$ and~$C' = b(\SS')$.
Choose parameterizations $\phi$, $\phi' \in \Diff(\S_{g+2}, M(\SS,\SS'))$ such that $\phi(m_{g+1}) = C$, $\phi(m_{g+2}) = C'$,
$\phi'(m_{g+1}) = C'$, $\phi'(m_{g+2}) = C$, and such that $\psi := \phi^{m_{g+1},m_{g+2}}$ and $\psi' := (\phi')^{m_{g+1},m_{g+2}}$ are isotopic
in $\Diff(\S_g, M)$. Furthermore, let $v \in F(M)$. Note that $\psi_{\SS,\SS'} = \phi$, hence
\[
F_{M(\SS), \SS'} \circ F_{M,\SS}(v)(\phi) = \omega_{g+1} \circ F_{M,\SS}(v)(\psi_\SS) =
\omega_{g+1} \circ \omega_g(v(\psi)).
\]
Similarly, $(\psi')_{\SS',\SS} = \phi'$, hence
\[
F_{M(\SS'), \SS} \circ F_{M,\SS'}(v)(\phi') = \omega_{g+1} \circ \omega_g(v(\psi')).
\]
Since~$\psi$ and~$\psi'$ are isotopic, $v(\psi) = v(\psi')$. Finally,
\[
F_{M(\SS'), \SS} \circ F_{M,\SS'}(v)(\phi') = \rho_{g+2}((\phi')^{-1} \circ \phi) \circ F_{M(\SS'), \SS} \circ F_{M,\SS'}(v)(\phi).
\]
As~$v$ is an arbitrary element of~$F(M)$, it follows that~$v(\phi)$ is an arbitrary element of~$A_g$.
Furthermore, $d = (\phi')^{-1} \circ \phi$ is an automorphism of~$\S_{g+2}$ that swaps~$m_{g+1}$ and~$m_{g+2}$,
and for which $d^{m_{g+1}, m_{g+2}}$ is isotopic to~$\Id_{\S_g}$.
Hence, relation~\eqref{it:commut} holds in case~\eqref{it:s1} if and only if for some diffeomorphism $d \in \Diff(\S_{g+2})$ that swaps~$m_{g+1}$
and~$m_{g+2}$, and for which $d^{m_{g+1}, m_{g+2}} \in \Diff(\S_g)$ is isotopic to~$\Id_{\S_g}$,
the automorphism~$\rho_{g+2}(d)$ of~$A_{g+2}$ is the identity on $\text{Im}(\omega_{g+1} \circ \omega_g)$; i.e.,
\[
\rho_{g+2}(d) \circ \omega_{g+1} \circ \omega_g = \omega_{g+1} \circ \omega_g.
\]
This holds by part~\eqref{it:omega-omega} of Lemma~\ref{lem:simplify}.


Now consider case~\eqref{it:s2}. Again, without loss of generality, assume that $M$ has only two components,
namely~$M_s$ of genus~$g$ and~$M_s'$ of genus~$g'$. Furthermore, by relation~\eqref{it:0-sphere}, which we will check later,
we can replace~$\SS$ by $\ol{\SS}$ if necessary to ensure that $\SS(-1,0) \in M_s$ and $\SS(1,0) \in M_s'$.
Similarly to the previous case, one can deduce that
commutativity of the two surgery maps holds if and only if
\[
\mu_{g,g'+1} \circ (\Id_{A_g} \otimes \omega_{g'}) = \omega_{g+g'} \circ \mu_{g,g'},
\]
which is true by equation~\eqref{eqn:6} of Lemma~\ref{lem:corresp}.

Case~\eqref{it:s3} is similar to case~\eqref{it:s1}. Without loss of generality, we can assume that~$M$
consists of only two components of genera~$g$ and~$g'$, respectively.
Let~$s$ be an arbitrary curve on~$\S_{g+g'+1}$ that becomes isotopic to~$s_g$
after doing surgery along~$m:= m_{g+g'+1}$; we can obtain~$s$ by taking the connected sum~$s_g \# m_{g+g'+1}$
along any path.
Let~$C = b(\SS)$ and~$C' = b(\SS')$. Then there is a diffeomorphism $\phi \in \Diff(\Sigma_{g+2},M(\SS,\SS'))$
such that $\phi(s) = C$ and $\phi(m) = C'$. As~$s$ is isotopic to~$s_g$ in~$M(m)$, we can canonically identify
$M(m,s)$ with $\S_g \sqcup \S_{g'}$, and we let
\[
\psi := \phi^{m,s} \in \Diff(\S_g \sqcup \S_{g'}, M).
\]
By construction, $\psi(\bP_{g,g'}) = \SS$ and $\phi^m(\bP_{g+g'}) = \SS'$, hence $\psi_{\SS,\SS'} = \phi$.
There exists a diffeomorphism $h \in \Diff(\S_{g+g'+1})$ such that $h(s) = m$ and $h(m) = s$, and
such that $h^{m,s}$ is isotopic to the identity. Then we set $\phi' := \phi \circ h^{-1}$;
this satisfies $\phi'(s) = C'$ and $\phi'(m) = C$. Again, if $\psi' = (\phi')^{m,s}$, then
$\phi' = (\psi')_{\SS',\SS}$. For any $v \in F(M)$, we have
\[
\begin{split}
F_{M(\SS), \SS'} \circ F_{M,\SS}(v)(\phi) &= \omega_{g+g'} \circ F_{M,\SS}(v)(\psi_\SS) =
\omega_{g+1} \circ \mu_{g+g'}(v(\psi)), \text{ and} \\
F_{M(\SS'), \SS} \circ F_{M,\SS'}(v)(\phi') &= \omega_{g+g'} \circ \mu_{g+g'}(v(\psi')).
\end{split}
\]
Since~$h^{m,s}$ is isotopic to the identity, $\psi$ and~$\psi'$ are isotopic, hence $v(\psi) = v(\psi')$. Furthermore,
\[
F_{M(\SS'), \SS} \circ F_{M,\SS'}(v)(\phi') = \rho_{g+g'+1}((\phi')^{-1} \circ \phi) \circ F_{M(\SS'), \SS} \circ F_{M,\SS'}(v)(\phi),
\]
and $(\phi')^{-1} \circ \phi = h$.
Hence, in this case, relation~\eqref{it:commut} translates to
\begin{equation} \label{eqn:s3}
\rho_{g+g'+1}(h) \circ \omega_{g+g'} \circ \mu_{g,g'} = \omega_{g+g'} \circ \mu_{g,g'}
\end{equation}
for some diffeomorphism $h \in \Diff(\S_{g+g'+1})$ that swaps $s_g \# m_{g+g'+1}$ and ~$m_{g+g'+1}$,
and such that $h^{m_{g+g'+1},s_g \# m_{g+g'+1}}$ is isotopic to the identity.
This holds by part~\eqref{it:omega-mu} of Lemma~\ref{lem:simplify}.

Finally, in case~\eqref{it:s4}, we obtain the associativity relation
\[
\mu_{g + g',g''} \circ (\mu_{g,g'} \otimes \Id_{A_{g''}}) = \mu_{g, g' + g''} \circ (\Id_{A_g} \otimes \mu_{g', g''}),
\]
which follows from equation~\eqref{eqn:1} of Lemma~\ref{lem:corresp}.

We now study relation~\eqref{it:commut} when both~$\SS$ and~$\SS'$ are framed $1$-spheres.
This is straightforward if~$\SS$ and~$\SS'$ occupy different components of~$M$.
Hence, without loss of generality, we can assume that~$M$ is connected of genus~$g$.
Then we have the following three cases:
\begin{enumerate}
\item \label{it:m1} Both~$\SS$ and~$\SS'$ are non-separating. There are two subcases depending on whether
    $\SS \cup \SS'$ is separating or not.
\item \label{it:m2} $\SS$ separates $M$ into components of genera~$j$ and $g-j$, and~$\SS'$ is non-separating.
    By relation~\eqref{it:0-sphere}, we can assume that~$\SS'$ lies on the positive side of $\SS$.
\item \label{it:m3} Both $\SS$ and $\SS'$ are separating.
By relation~\eqref{it:0-sphere}, we can assume that~$\SS'$ lies on the positive side of~$\SS$,
and that $\SS$ is on the negative side of~$\SS'$.
They divide~$M$ into pieces of genera~$i$, $j$, and~$k$.
\end{enumerate}

First, consider case~\eqref{it:m1}, and suppose that $\SS \cup \SS'$ is non-separating.
Then we can choose parameterizations~$\phi$, $\phi' \in \Diff(\S_g, M)$ for which $\phi(l_g) = \SS$,
$\phi(l_{g-1}) = \SS'$, $\phi'(l_g) = \SS'$, and $\phi'(l_{g-1}) = \SS$, and such that $\phi^{l_g,l_{g-1}}$
and $(\phi')^{l_g,l_{g-1}}$ are isotopic. Furthermore, let $v \in F(M)$. Then, by definition,
\[
F_{M(\SS),\SS'} \circ F_{M,\SS}(v)(\phi_{\SS,\SS'}) = \a_{g-1} \circ \a_g(v(\phi)),
\]
and, symmetrically,
\[
F_{M(\SS'),\SS} \circ F_{M,\SS'}(v)(\phi'_{\SS',\SS}) = \a_{g-1} \circ \a_g(v(\phi')).
\]
Since~$\phi_{\SS,\SS'} = \phi^{l_g,l_{g-1}}$ and $\phi'_{\SS',\SS} = (\phi')^{l_g,l_{g-1}}$ are isotopic,
we have
\[
F_{M(\SS),\SS'} \circ F_{M,\SS}(v)(\phi_{\SS,\SS'}) = F_{M(\SS'),\SS} \circ F_{M,\SS'}(v)(\phi'_{\SS',\SS}).
\]
Furthermore, $v(\phi') = \rho_g((\phi')^{-1} \circ \phi)(v(\phi))$.
Hence relation~\eqref{it:commut} holds in this case if and only if for some diffeomorphism $d \in \Diff(\S_g)$
that swaps~$l_g$ and~$l_{g-1}$ and for which~$d^{l_g, l_{g-1}} \in \Diff_0(\S_{g-2})$, we have
\[
\a_{g-1} \circ \a_g \circ \rho_g(d) = \a_{g-1} \circ \a_g.
\]
This is precisely part~\eqref{it:alpha-alpha} of Lemma~\ref{lem:simplify}.

If, in case~\eqref{it:m1}, the union~$\SS \cup \SS'$ separates~$M$ into pieces of genera~$i$ and~$j$, respectively,
then~$g = i+j+1$. The model case is when $M = \S_g$, $\SS = s_i \# l_g$, and~$\SS' = l_g$. Similarly
to equation~\eqref{eqn:s3}, we obtain the relation
\[
\d_{i,j} \circ \a_g \circ \rho_g(u) = \d_{i,j} \circ \a_g,
\]
where~$u \in \Diff(\S_g)$ swaps $s_i \# l_g$ and~$l_g$,
and such that $u^{s_i \# l_g, l_g}$ is isotopic to the identity.
This follows from part~\eqref{it:delta-alpha} of Lemma~\ref{lem:simplify}.

Now consider case~\eqref{it:m2}. This leads to the relation
\[
\delta_{j, g-j-1} \circ \a_g = (\Id_{A_j} \otimes \a_{g-j}) \circ \delta_{j,g-j},
\]
which is part of equation~\eqref{eqn:6} of Lemma~\ref{lem:corresp}.
Case~\eqref{it:m3} leads to the coassocitivity relation
\[
(\Id_{A_i} \otimes \delta_{j,k}) \circ \delta_{i,j+k} = (\delta_{i,j} \otimes \Id_{A_k}) \circ \delta_{i+j,k},
\]
which holds by equation~\eqref{eqn:2} of Lemma~\ref{lem:corresp}.

Finally, we consider relation~\eqref{it:commut} when~$\SS$ is a framed $0$-sphere and~$\SS'$ is a framed $1$-sphere.
Without loss of generality, we can assume that~$\SS$ intersects the component of~$M$ that~$\SS'$ occupies.
Here we distinguish the following cases:
\begin{enumerate}
\item \label{it:d1} $\SS$~lies in a single component~$M_s$ and~$\SS' \subset M_s$ is non-separating.
\item \label{it:d2} $\SS$~lies in a single component~$M_s$ and~$\SS'$ separates~$M_s$ into
pieces of genera~$i$ and~$g-i$. There are three subcases depending on whether~$\SS$ lies completely
to the left of~$\SS'$, on both sides, or completely to the right.
\item \label{it:d3} $\SS$ occupies the components~$M_s$ and~$M_s'$, and~$\SS' \subset M_s'$ is non-separating.
\item \label{it:d4} $\SS$ occupies the components~$M_s$ and~$M_s'$, and~$\SS'$ separates~$M_s'$
into components of genera~$i$ and~$g'-i$. There are two subcases depending on whether the point of~$\SS$
in~$M_s'$ lies to the left or to the right of~$\SS'$. By relation~\eqref{it:0-sphere}, we can assume it lies
to the left.
\end{enumerate}

In case~\eqref{it:d1}, without loss of generality, we can assume that~$M$ is connected.
Furthermore, by naturality, we can assume that $M = \S_g$, $\SS = \PP_g$, and~$\SS' = l_g$
(or, more precisely, we work with a parametrization $\phi \in \Diff(\S_g, M)$ such that $\phi(\PP_g) = \SS$
and $\phi(l_g) = \SS'$).
Let $d \in \Diff(\S_{g+1})$ be such that $d(l_g) = l_{g+1}$, and
$d^{l_g} = \Id_{\S_g}$ after the natural identifications of~$\S_{g+1}(l_g)$
and~$\S_{g+1}(l_{g+1})$ with~$\S_g$. As we already know the surgery maps are natural, the following diagram
is commutative:
\[
\xymatrix{
  F(\S_{g+1}) = A_{g+1} \ar[r]^-{\a_{g+1}}  & F(\S_g) = A_g  \\
  A_{g+1}  \ar[r]^-{F_{\S_{g+1},l_g}} \ar[u]^-{\rho_{g+1}(d)} & F(\S_{g+1}(l_g)) \cong A_g. \ar[u]_{F\left(d^{l_g}\right)}}
\]
By construction, $d^{l_g}$ is isotopic to~$\Id_{\S_g}$, so $F(d^{l_g}) = \Id_{A_g}$, and
\[
F_{\S_{g+1},l_g} = \a_{g+1} \circ \rho_{g+1}(d).
\]
Hence, from relation~\eqref{it:commut}, we obtain the condition
\[
\omega_{g-1} \circ \a_g = \a_{g+1} \circ \rho_{g+1}(d) \circ \omega_g,
\]
where $d \in \Diff(\S_{g+1})$ is such that $d(l_g) = l_{g+1}$, and
$d^{l_g} = \Diff_0(\S_g)$ after the natural identifications of~$\S_{g+1}(l_g)$
and~$\S_{g+1}(l_{g+1})$ with~$\S_g$. Notice that the diffeomorphism~$d$ coincides
with~$L_{g+1}$ acting on~$\S_{g+1}$ and interchanging~$l_g$ and~$l_{g+1}$.
Hence, this holds by part~\eqref{it:alpha-omega} of Lemma~\ref{lem:simplify} for $g = 1$,
and by property~\eqref{it:L} of Definition~\ref{def:repr} for $g > 1$.

In case~\eqref{it:d2}, when $\SS$ lies to the left of~$\SS'$, we replace~$\SS'$
by~$\ol{\SS}'$ and apply relation~\eqref{it:0-sphere}. The other two
cases lead to the relations
\begin{equation} \label{eqn:d2}
\begin{split}
\delta_{i,j+1} \circ \omega_g &= (\Id_{A_i} \otimes \omega_j) \circ \delta_{i,j}, \text{ and} \\
\a_{g+1} \circ \rho_{g+1}(\sigma_{g+1,i}) \circ \omega_g &= \mu_{i,g-i} \circ \delta_{i,g-i}.
\end{split}
\end{equation}
The first line of equation~\eqref{eqn:d2} follows from equation~\eqref{eqn:6} of Lemma~\ref{lem:corresp}.
The second equation, which is condition~\eqref{it:frobtwist} of Definition~\ref{def:repr},
can be derived by reversing the argument in the proof of Proposition~\ref{prop:AF} showing that
this holds for every $J$-algebra assigned to a TQFT.
Indeed, for the model case when $\SS = \bP$ and $\SS' = s_i$,
see the second surface on the top of Figure~\ref{fig:relations}.
In the proof of Proposition~\ref{prop:AF}, we derived the commutativity of the large pentagon
from the commutativity of the upper right square. We now use the reverse implication,
which follows from the commutativity of the other two small squares and the lower left triangle,
which in turn is a consequence of naturality.

In case~\eqref{it:d3}, the necessary and sufficient condition for relation~\eqref{it:commut} to hold is
\[
\a_{g+g'} \circ \mu_{g,g'} = \mu_{g,g'-1} \circ (\Id_{A_g} \otimes \a_{g'}),
\]
which follows from equation~\eqref{eqn:6} of Lemma~\ref{lem:corresp}.
There is a corresponding relation if $\SS'$ lies on the other side of~$\SS$, but that follows
from this one by relation~\eqref{it:0-sphere}.

Finally, in case~\eqref{it:d4}, we obtain
\[
\delta_{g+i,g'-i} \circ \mu_{g,g'} =  (\mu_{g,i} \otimes \Id_{A_{g'-i}})  \circ (\Id_{A_g} \otimes \delta_{i,g'-i}),
\]
which is the Frobenius condition~\eqref{eqn:3} in Lemma~\ref{lem:corresp}.

We now consider relation~\eqref{it:birth}; i.e., where $\SS' \subset M(\SS)$ intersects the belt sphere of~$\SS$ once.
If~$\SS = 0$ and~$\SS'$ is a $0$-sphere that has one point on the new~$S^2$ component and another point on a
component of~$M$ of genus~$g$, then we can assume $\SS'(-1,0) \in S^2$ by relation~\eqref{it:0-sphere}.
This leads to the relation
\[
\mu_{0,g} \circ (\eps \otimes \Id_{A_g}) = \Id_{A_g};
\]
i.e, that $1 = \eps(1)$ is a left unit for~$\mu$.
If~$\SS$ is a $0$-sphere, it has to lie in a single component of~$M$. Then we obtain the relation
\[
\a_{g+1}\circ \omega_g = \Id_{A_g},
\]
which is equation~\eqref{eqn:5} of Lemma~\ref{lem:corresp}.
If~$\SS$ is a $1$-sphere, then it has to be inessential, and $\SS'$ is the $2$-sphere split off by~$\SS$.
By relation~\eqref{it:0-sphere}, we can assume this $2$-sphere lies on the negative side of~$\SS$.
We obtain the relation
\[
(\tau \otimes \Id_{A_g}) \circ \delta_{0,g} = \Id_{A_g},
\]
which holds since~$\tau$ is a left counit for the coproduct~$\delta$.

Finally, consider relation~\eqref{it:0-sphere}. Think of~$\S_g$ as being standardly embedded in~$\R^3$ with
center lying at the origin, and such that the $x$-axis intersects it in the points~$p_g$ and~$q_g$.
Let $\iota_g \in \Diff(\S_g)$ be the involution
of~$\S_g$ that is a $\pi$-rotation about the $z$-axis
and swaps the $i$-th and $(g-i)$-th $S^1 \times S^2$ factor of~$\S_g$.
The $z$-axis passes through $s_{g/2}$ if~$g$ is even, and through the hole of the $(g+1)/2$-th $S^1 \times S^2$
summand when~$g$ is odd. This has the property that $\iota_g(s_i) = s_{g-i}$ for every $i \in \{\,0,\dots, g \,\}$.

First, suppose that~$\SS$ is a $0$-sphere that occupies two components of~$M$. Then the model scenario
is $M = \S_i \sqcup \S_j$ and~$\SS = \bP_{i,j}$. Let~$\sigma \colon \S_i \sqcup \S_j \to \S_j \sqcup \S_i$
be the diffeomorphism that swaps the two components of~$\S_i \sqcup \S_j$, then acts via $\iota_i \sqcup \iota_j$.
Then $\sigma(\ol{\bP}_{i,j}) = \bP_{j,i}$ and $\sigma^{\ol{\SS}} = \iota_{i+j}$. Hence,
using that $F_{M,\SS} = F_{M,\ol{\SS}}$ and the naturality of the surgery maps,
relation~\eqref{it:0-sphere} amounts to the relation
\[
\rho(\iota_{i+j}) \circ \mu_{i,j}(x \otimes y) = \mu_{j,i}\left(\rho_j(\iota_j)(y) \otimes \rho_i(\iota_i)(x)\right)
\]
for every $x \in A_i$ and $y \in A_j$. As $x^* = \rho(\iota_i)(x)$ for
every $i \in \Z_{\ge 0}$ and~$x \in A_i$, we can rewrite this relation as
\[
\mu_{i+j}(x,y)^* = \mu_{j,i}(y^* \otimes x^*),
\]
which is equation~\eqref{eqn:4} of Lemma~\ref{lem:corresp}.

Now consider the case when~$\SS$ is a $0$-sphere in a single component of~$M$. Then the model case is
$M = \S_g$ and~$\SS = \bP_g$. Let $t_g \in \Diff(\S_g)$ be the diffeomorphism that is characterized
by $t_g(m_g) = -m_g$ and $t_g^{m_g} \in \Diff_0(\S_{g-1})$. Then relation~\eqref{it:0-sphere} in this
case is equivalent to the relation
\[
\rho_{g+1}(t_{g+1}) \circ \omega_g = \omega_g,
\]
which is part~\eqref{it:t-omega} of Lemma~\ref{lem:simplify}.

Applied to separating $1$-spheres, we obtain the relation
\[
T_{i,j} \circ \delta_{i,j}(x) = \delta_{j,i}(x^*),
\]
where $T_{i,j} \colon A_i \otimes A_j \to A_j \otimes A_i$ is given by $T_{i,j}(v \otimes w) = w^* \otimes v^*$.
This is part of equation~\eqref{eqn:4} of Lemma~\ref{lem:corresp}

When~$\SS$ is a non-separating $1$-sphere, we obtain that
\[
\alpha_g = \alpha_g \circ \rho(r_g),
\]
where $r_g \in \Diff(\S_g)$ is characterized by $r_g(l_g) = -l_g$ and $(r_g)^{l_g} \in \Diff_0(\S_{g-1})$.
This is precisely part~\eqref{it:alpha-rho} of Lemma~\ref{lem:simplify}.
\end{proof}

\begin{proof}[Proof of Theorem~\ref{thm:2+1}]
By Proposition~\ref{prop:AF}, for every (2+1)-dimensional TQFT~$F$, the tuple $J(F)$
defined in Section~\ref{sec:assignment} is a J-algebra.
Conversely, Proposition~\ref{prop:FA} ensures that, given a J-algebra~$\AA$,
the associated functor $T(\AA)$ is a TQFT. Both of these assignments are functorial:
Given a natural transformation $\eta \colon F \Rightarrow F'$ of TQFTs,
the maps $\eta_{\S_i} \colon F(\S_i) \to F'(\S_i)$
form a J-algebra homomorphism
\[
J(\eta) \colon J(F) \to J(F').
\]
Conversely, a J-algebra homomorphism $h \colon \AA \to \AA'$ extends to a natural isomorphism
$T(h) \colon T(\AA) \Rightarrow T(\AA')$ in a straightforward manner.
Indeed, for a surface $M$ of diffeomorphism type $\S_{\ug}$, we obtain
$T(h)_M \colon T(\AA)(M) \to T(\AA')(M)$ by mapping the factor $A_{\ug}$ of $T(\AA)(M)$ corresponding
to $\phi \in \Diff(\S_{\ug},M)$ to the factor $A'_{\ug}$ of $T(\AA')(M)$ corresponding to the same
parametrization~$\phi$ via
\[
h_{\ug} :=  h_1^{\otimes n_1} \otimes \dots \otimes h_r^{\otimes n_r},
\]
where $h_i = h|_{A_i} \colon A_i \to A_i'$.

We finally show that the functors $J \circ T$ and $T \circ J$ are naturally isomorphic to the identity.
Let $\AA = (\A,\a,\omega,\{\rho_i \,\colon\, i \in \N\})$ be a J-algebra. Then
the J-algebra $J \circ T(\AA)$ in grading $g$ is given by $T(\AA)(\S_g)$.
This is a subset of $\prod_{d \in \Diff(\S_g,\S_g)} A_g$, and projecting it onto the $\Id_{\S_g}$ factor gives a
natural isomorphism to $A_g$.

Now consider $T \circ J$. Let $F \colon \Cob_2 \to \Vect$ be a TQFT, and let $\AA = J(F)$ be
the corresponding J-algebra. We are going to construct
a monoidal natural isomorphism $\eta \colon T \circ J \Rightarrow \Id$.
In particular, we define
\[
\eta_F \colon T \circ J(F) = T(\AA) \to F,
\]
which is itself is a monoidal natural isomorphism. If $M$ is a surface,
then we need to give an isomorphism
\[
\eta_{F,M} \colon T(\AA)(M) \to F(M).
\]
Pick a parametrization $\phi \in \Diff(\S_{\ug}, M)$; then this induces an isomorphism
\[
F(\phi) \colon F(\S_{\ug}) \to F(M).
\]
The monoidal structure of $F$ gives an isomorphism
$\Phi_{\ug} \colon A_{\ug} \to F(\S_{\ug})$, as $\AA = J(F)$ and hence $A_i = F(\S_i)$ for every $i \in \N$.
If
\[
p_\phi \colon \prod_{\psi \in \Diff(\S_{\ug},M)} A_{\ug} \to A_{\ug}
\]
is the projection onto the $\phi$ factor, then it restricts to an isomorphism
\[
p_\phi|_{T(\AA)(M)} \colon T(\AA)(M) \to A_{\ug}.
\]
Finally, we set
\[
\eta_{F,M} := F(\phi) \circ \Phi_{\ug} \circ p_\phi|_{T(\AA)(M)}.
\]
We leave it to the reader to check that this is independent of the choice of~$\phi$,
that $\eta_F$ is indeed a monoidal natural isomorphism, and that $\eta$
is a monoidal natural isomorphism from $T \circ J$ to the identity.
\end{proof}

\section{Examples and applications} \label{sec:applications}

Dijkgraaf~\cite{Dij} noted that if~$F$ is an $(n+1)$-dimensional TQFT, then~$F(S^n)$ carries the structure
of a commutative Frobenius algebra that acts on $F(M)$ for every connected $n$-manifold~$M$.
We say that~$F$ is \emph{based} on the Frobenius algebra~$F(S^n)$.
Sawin~\cite[Theorem~1]{Sawin2} proved the following result about direct sum decompositions of TQFTs.

\begin{prop} \label{prop:decomp}
Suppose the TQFT~$F$ is based on a direct sum $A = A_1 \oplus A_2$ of Frobenius algebras.
Then there exist TQFTs~$F_1$ and~$F_2$, based on~$A_1$ and~$A_2$, respectively, such that
$F \cong F_1 \oplus F_2$. Conversely, if~$F$ decomposes as a direct sum of TQFTs, then the associated
Frobenius algebra decomposes as a corresponding direct sum of Frobenius algebras.
\end{prop}

He also gave a classification of indecomposable commutative Frobenius algebras
over an algebraically closed field~$\F$. For each $\lambda \in \F^\times$,
let~$S_\lambda$ be the algebra~$\F$ with counit~$\tau(x) = \lambda^{-1} x$. Also, let~$A$ be a commutative
algebra spanned by the identity and at least one nilpotent, and suppose the socle,
the space of all~$x \in A$ such that $ax = 0$ for all nilpotent $a \in A$, is one-dimensional.
Let~$\tau$ be any linear functional on~$A$ which is non-zero on the socle. We denote by~$N_{A,\tau}$ the
algebra~$A$ together with the functional~$\tau$. The following is~\cite[Proposition~2]{Sawin2}.

\begin{prop} \label{prop:base}
$S_\lambda$ and $N_{A,\tau}$ are indecomposable Frobenius algebras. Further, every
commutative indecomposable Frobenius algebra is isomorphic to one of these, and
these are nonisomorphic up to algebra isomorphism.
\end{prop}

We now turn our attention to (2+1)-dimensional TQFTs. By Proposition~\ref{prop:decomp}, it suffices to
focus on irreducible theories as every TQFT is a direct sum of these.
An important class of (2+1)-dimensional TQFTs  are ones that extend to one-manifolds, these
are called (1+1+1)-dimensional TQFTs.
Bartlett et.~al~\cite{BDSV, BDSV2} showed that (1+1+1)-dimensional TQFTs correspond to
anomaly free modular tensor categories. Given an anomaly free modular tensor category, they also
describe how to obtain $F(\S_g)$ by taking the vector space generated by string diagrams inside the
handlebody bounded by~$\S_g$ in~$\R^3$ and labeled by simple objects, modulo
equivalence relations in the category. They give the action of elementary cobordisms as well.
This is essentially the construction of Reshetikhin and Turaev~\cite{RT}.
It is a fundamental open question whether every (2+1)-dimensional TQFT~$F$ comes from a (1+1+1)-dimensional
theory. If it does and if~$F$ is irreducible, then $\dim F(S^2) = 1$, so it has to be based on
one of the Frobenius algebras~$S_\lambda$ according to Proposition~\ref{prop:base}.

Consider condition~\eqref{it:frobtwist} of Definition~\ref{def:repr} for~$n = 0$ and $i = 0$:
\begin{equation} \label{eqn:spec}
\a_1 \circ \rho_1(\sigma_{1,0}) \circ \omega_0 = \mu_{0,0} \circ \delta_{0,0}.
\end{equation}
The diffeomorphism $\sigma_{1,0} = a_1 \circ H_{1,0} \in \Diff(T^2)$ induces the $S$-matrix
\[\begin{pmatrix}
0 & -1 \\
1 & 0
\end{pmatrix}\]
on $H_1(T^2)$ in the basis~$\langle m , l \rangle$ since $h_{1,0}$ is isotopic to the identity.
If $\rho_1(\sigma_{1,0}) = \Id_{A_1}$, then the left-hand side of equation~\eqref{eqn:spec}
becomes $\a_1 \circ \omega_0 = \Id_{A_0}$. This means that $\mu_{0,0} \circ \delta_{0,0} = \Id_{A_0}$; i.e.,
that the Frobenius algebra~$A_0$ is \emph{special}. The only special Frobenius algebra among~$S_\lambda$
and~$N_{A,\tau}$ is~$S_0$. So, if the (2+1)-dimensional TQFT satisfies $\rho_1(\sigma_{1,0}) = \Id_{A_1}$,
then it is based on the direct sum of finitely many copies of~$S_0$, and is based on~$S_0$ if it is indecomposable.

\begin{ex} \label{ex:simple}
Consider the \gnf-algebra $\A = (A,\mu,\d,\eps,\tau,*)$, where $(A,\mu)$ is the polynomial algebra $\F[x]$
with grading $A_i = \F \langle x^i \rangle$, coproduct
\[
\d(x^n) = \sum_{i = 0}^n x^i \otimes x^{n-i},
\]
unit $\eps = \Id_\F \colon \F \to A_0$, partial counit $\tau = \Id_\F \colon A_0 \to \F$, and involution $* = \Id_A$.
We define the modular splitting $(\a,\omega)$ by taking $\a(x^i) = x^{i-1}$ for $i > 0$ and $\a(1) = 0$,
and $\omega$ is multiplication by~$x$. If we define each $\rho_i \colon \M_i \to \End(A_i)$ to be trivial,
then this satisfies all the properties of a mapping class group representation. Hence, this data gives rise to
a (2+1)-dimensional TQFT~$F_1$. This assigns~$\F$ to any surface, and the identity morphism to any cobordism
between two surfaces, under the identifications $\F^{\otimes k} \cong \F$.
\end{ex}

\begin{lem} \label{lem:triv}
Let $(\A,\a,\omega)$ be a split \gnf algebra with a mapping class group representation
such that $\rho_i \colon \M_i \to A_i$ is trivial for some~$i \in \N$.
Then~$\rho_j$ is also trivial for every~$j < i$.
\end{lem}

\begin{proof}
It suffices to show that~$\rho_{i-1}$ is also trivial. Pick an arbitrary diffeomorphism $d \in \Diff(\S_{i-1})$.
We isotope~$d$ such that it fixes the disk bounded by the curve $s_{i-1} \subset \S_{i-1}$, and let $d_i \in \Diff(\S_i)$
be the diffeomorphism of~$\S_i$ that agrees with~$d$ to the left of the curve~$s_{i-1} \subset \S_i$, and
is the identity to the right of~$s_{i-1}$. Then, by the $\MCG(\S_{i-1} \sqcup \S_i, \bP_{i-1,1})$-equivariance of $\mu_{i-1,1}$,
and since $\rho_i$ is trivial, we have
\[
\mu_{i-1,1}(\rho_{i-1}(d)(x) \otimes w) = \rho_i(d_i)(\mu_{i-1,1}(x \otimes w)) = \mu_{i-1,1}(x \otimes w)
\]
for every $x \in A_{i-1}$. It follows that
\[
\omega_{i-1}\left((\rho_{i-1}(d) - \Id_{A_{i-1}})(x)\right) = \mu_{i-1,1}\left((\rho_{i-1}(d) - \Id_{A_{i-1}})(x) \otimes w\right) = 0
\]
for every $x \in A_{i-1}$. As $\omega_{i-1}$ is injective, this implies that $\rho_{i-1}(d) = \Id_{A_{i-1}}$.
\end{proof}

\begin{prop} \label{prop:trivmcg}
Let $(\A,\a,\omega)$ be a split \gnf-algebra  over $\C$ such that
\[
\dim A_i < 2i
\]
for some $i > 2$.
Then~$\rho_j$ is trivial for every~$j \le i$. Hence, if $\dim A_i < 2i$ for infinitely many $i \in \N$,
then every mapping class group representation on~$\A$ is trivial.
\end{prop}

\begin{proof}
Franks and Handel~\cite{FH} proved that any representation of $\M_i$ in $\text{GL}(n,\C)$ is trivial assuming that $i > 2$
and $n < 2i$. The result now follows from Lemma~\ref{lem:triv}.
\end{proof}

\begin{prop} \label{prop:onedim}
Let $F \colon \Cob_2 \to \Vect_\C$ be a TQFT such that $F(\S) \cong \C$ for every surface~$\S$.
Then there is a natural isomorphism between~$F$ and the TQFT~$F_1$ constructed in Example~\ref{ex:simple}.
\end{prop}

\begin{proof}
Let $(\A,\a,\omega)$ be the split \gnf-algebra associated to the TQFT~$F$. By Proposition~\ref{prop:trivmcg},
the mapping class group action is trivial. Since $\dim A_i = 1$ for every $i \in \N$, the map $\omega_i$ is a bijection
for every~$i \in \N$.
As $\omega$ is given by right-multiplication with an element $w \in A_1$, it follows that $\A \cong \C[x]$,
where the isomorphism maps $w^n \in A_n$ to $x^n$. From the formula $\a_{i+1} \circ \omega_i = \Id_{A_i}$,
we obtain that $\a_{i+1} = \omega_i^{-1}$; i.e., $\a_{i+1}(w^{i+1}) = w^i$.
Since $\mu$ is associative, $\mu_{i,j}(w^i \otimes w^j) = w^{i+j}$.
By condition~\eqref{it:frobtwist} of Definition~\ref{def:repr}, and since $\rho_{n+1}$ is trivial,
\[
\mu_{i,n-i} \circ \delta_{i,n-i} = \a_{n+1} \circ \rho_{n+1}(\sigma_{n+1,i}) \circ \omega_n  = \Id_{A_n}.
\]
It follows that $\d_{i,n-i} = (\mu_{i,n-i})^{-1} \colon A_n \to A_i \otimes A_{n-i}$; hence, $\d_{i,n-i}(w^n) = w^i \otimes w^{n-i}$.
Finally, since $(\tau \otimes \Id_{A_0}) \circ \d_{0,0}(1) = \tau(1) = 1$, we have $\tau = \Id_{\C}$.
So the \gnf-algebra $(\A,\a,\omega)$ is isomorphic to the \gnf-algebra $\C[x]$ of Example~\ref{ex:simple}.
It follows that~$F$ is isomorphic to~$F_1$.
\end{proof}

\begin{prop} \label{prop:n}
Let $F \colon \Cob_2 \to \Vect_\C$ be a TQFT, and suppose that there is a number $n \in \N$ such that
$\dim F(\S) = n$ for every connected surface~$\S$.
Then there is a natural isomorphism between~$F$ and~$(F_1)^{\oplus n}$, where~$F_1$ is the TQFT
constructed in Example~\ref{ex:simple}.
\end{prop}

\begin{proof}
By Proposition~\ref{prop:trivmcg}, the mapping class group representation corresponding to~$F$ is trivial in every genus.
In particular, $\rho_1(\sigma_{1,0}) = \Id_{A_1}$, and hence, by equation~\eqref{eqn:spec}, the commutative Frobenius algebra~$A_0$ is
special, and so it is a direct sum of finitely many copies of~$S_0$. By Proposition~\ref{prop:decomp}, the TQFT $F$ splits as
a direct sum $Z_1 \oplus \dots \oplus Z_n$ of TQFTs, each based on~$S_0$. In particular, $\dim Z_i(S^2) = 1$, and so, by
the injectivity of the map~$\omega$, we have $\dim Z_i(\S_g) \ge 1$ for every $i \in \{\,1,\dots,n\,\}$.
Since
\[
\sum_{i=1}^n \dim Z(\S_g) = F(\S_g) = n,
\]
we must have~$\dim Z_i(\S_g) = 1$ for every $i \in \{\,1,\dots,n\,\}$. So Proposition~\ref{prop:onedim}
implies that $Z_i \cong F_1$ for every $i \in \{\,1,\dots,n\,\}$, hence $F \cong (F_1)^{\oplus n}$.
\end{proof}

\begin{ex} \label{ex:complicated}
This is an extension of Example~\ref{ex:simple}, and gives an explicit
description of the split \gnf-algebra associated to the TQFT $(F_1)^{\oplus n}$ appearing in Proposition~\ref{prop:n}.
Let $(A_0, \mu, \d, \eps, \tau)$ be a commutative special Frobenius algebra over a field~$\F$,
where by \emph{special} we mean that $\mu \circ \d = \Id_{A_0}$. We know from above that this is a
direct sum of copies of~$S_0$ if~$\F$ is algebraically closed.
Then we can associate to~$A_0$ a split \gnf-algebra
\[
\A = \A(A_0) = (A,\mu_{i,j},\d_{i,j},\eps,\tau,*,\a_i,\omega_i)
\]
with a trivial mapping class group action, as follows. Let $A = A_0 \otimes \F[x]$ with the grading
$A_i = A_0 \otimes \F\langle x^i \rangle$ for $i \in \N$, where we identify~$A_0$ with
$A_0 \otimes \F \langle 1 \rangle$. For elements $a$, $b \in A_0$, we define
\[
\mu_{i,j}(ax^i \otimes bx^j) = ab x^{i+j},
\]
where~$ab$ stands for~$\mu(a \otimes b)$. This product is clearly associative.

For $a \in A$, let $\d(a) = a_{(1)} \otimes a_{(2)}$
in sumless Sweedler notation. Then we define
\[
\d_{i,j}(ax^{i+j}) = a_{(1)} x^i \otimes a_{(2)}x^j.
\]
We now show $\d_{i,j}$ is coassociative. The coassociativity of~$\d$ in Sweedler notation can be written as
\[
a_{(1)} \otimes a_{(2)(1)} \otimes a_{(2)(2)} = a_{(1)(1)} \otimes a_{(1)(2)} \otimes a_{(2)}.
\]
Then we have
\begin{equation*}
\begin{split}
(\Id_{A_i} \otimes \d_{j,k}) \circ \d_{i,j+k}(ax^{i+j+k}) &= (\Id_{A_i} \otimes \d_{j,k})(a_{(1)}x^i \otimes a_{(2)}x^{j+k}) = \\
a_{(1)}x^i \otimes a_{(2)(1)} x^j \otimes a_{(2)(2)} x^k &= a_{(1)(1)} x^i \otimes a_{(1)(2)} x^j \otimes a_{(2)} x^k = \\
(\d_{i,j} \otimes \Id_{A_k})(a_{(1)} x^{i+j} \otimes a_{(2)} x^k) &= (\d_{i,j} \otimes \Id_{A_k}) \circ \d_{i+j,k} (ax^{i+j+k}).
\end{split}
\end{equation*}

The unit $\eps \colon \F \to A_0$ of~$\A$ is defined to be the unit of the Frobenius algebra~$A_0$.
Indeed, if $\eps(1) = 1 \in A_0$, then $1 \cdot x^0$ is a unit of~$\A$ as $\mu_{0,i}(1 \cdot x^0 \otimes ax^i) = ax^i$.
The counit~$\tau \colon A_0 \to \F$ of the Frobenius algebra~$A_0$ will be the partial left counit of
our \gnf-algebra~$\A$. More precisely, we set~$\tau(ax^0) = \tau(a)$. Indeed, we have
\[
\tau(a_{(1)} \otimes a_{(2)}) = \tau(a_{(1)}) a_{(2)} = a, \text{ hence}
\]
\[
(\tau \otimes \Id_{A_j}) \circ \d_{0,j}(ax^j) = (\tau \otimes \Id_{A_j})(a_{(1)}x^0 \otimes a_{(2)}x^j)
= \tau(a_{(1)}) a_{(2)} x^j = ax^j.
\]

The Frobenius condition for~$A_0$ can be written as
\[
(ab)_{(1)} \otimes (ab)_{(2)} = ab_{(1)} \otimes b_{(2)}.
\]
This implies the Frobenius condition for~$\A$, as
\begin{equation*}
\begin{split}
(\mu_{i,j} \otimes \Id_{A_k}) \circ (\Id_{A_i} \otimes \d_{j,k})(ax^i \otimes bx^{j+k}) &=
(\mu_{i,j} \otimes \Id_{A_k})(ax^i \otimes b_{(1)} x^j \otimes b_{(2)} x^k) = \\
ab_{(1)} x^{i+j} \otimes b_{(2)} x^k  &= (ab)_{(1)} x^{i+j} \otimes (ab)_{(2)} x^k = \\
\d_{i+j,k}(abx^{i+j+k}) &= \d_{i+j,k} \circ \mu_{i,j+k}(ax^i \otimes bx^{j+k}).
\end{split}
\end{equation*}

We define the involution~$*$ to be the identity; then this is an anti-automorphism since
$A_0$ is commutative. Indeed,
\[
(ax^i \cdot bx^j)^* = abx^{i+j} = bax^{j+i} = (bx^j)^* \cdot (ax^i)^*,
\]
and similarly for the coproduct.

The modular splitting is defined by the formulas $\omega(ax^i) = ax^{i+1}$ for $i \in \N$
and $\a(ax^i) = ax^{i-1}$ for $i > 0$ and $\a(ax^0) = 0$. These satisfy the necessary conditions as
\[
\d_{i,j-1} \circ \a_{i+j}(ax^{i+j}) = \d_{i,j-1}(ax^{i+j-1}) = a_{(1)}x^i \otimes a_{(2)}x^{j-1} =
(\Id_{A_i} \otimes \a_j) \circ \d_{i,j}(ax^{i+j}),
\]
and similarly,
\[
\d_{i,j+1} \circ \omega_{i+j}(ax^{i+j}) = a_{(1)}x^i \otimes a_{(2)}x^{j+1} =
(\Id_{A_i} \otimes \omega_j) \circ \d_{i,j}(ax^{i+j}).
\]

Finally, we show that the trivial mapping class group representation satisfies the required properties.
Except for the following two, all properties of a mapping class group representation trivially hold.
Condition~\eqref{it:L} of Definition~\ref{def:repr} translates to
\[
\a_{i+1} \circ \omega_i = \omega_{i-1} \circ \a_i.
\]
Indeed,
\[
\a_{i+1} \circ \omega_i(ax^i) = \a_{i+1}(ax^{i+1}) = ax^i = \omega_{i-1}(ax^{i-1}) = \omega_{i-1} \circ \a_i (ax^i).
\]
To check condition~\eqref{it:frobtwist},
observe that the left-hand side equals $\a_{n+1} \circ \omega_n = \Id_{A_n}$.
Furthermore,
\[
\mu_{i,n-i} \circ \d_{i,n-i}(ax^n) = \mu_{i,n-i}(a_{(1)}x^i \otimes a_{(2)}x^{n-i}) = a_{(1)}a_{(2)}x^n = \mu \circ \d(a) x^n = a x^n,
\]
where in the last step we used that the Frobenius algebra~$A_0$ is special.
\end{ex}

In summary, if one would like to find a (2+1)-dimensional TQFT~$F$ over~$\C$ that does not extend to a (1+1+1)-dimensional one
by constructing an irreducible TQFT that is based on one of the nilpotent Frobenius algebras~$H_{A,\tau}$,
then one must have $F(\S_g) \ge 2g$ with a non-trivial mapping class group action for each $g > 0$. Otherwise,
the commutative Frobenius algebra~$F(S^2)$ will be irreducible and special, and hence isomorphic to~$S_1$.

In contrast, as stated in Proposition~\ref{prop:nonextending},
there are $2^{2^\omega}$ pairwise non-equivalent (2+1)-dimensional
oriented \emph{lax-monoidal} TQFTs over~$\C$ that do not extend to (1+1+1)-dimensional TQFTs.
We now prove this claim.

\begin{proof}[{Proof of Proposition~\ref{prop:nonextending}}]
According to Funar~\cite[p.~410]{Funar}, a $\C$-valued homeomorphism invariant~$f$ of
oriented 3-manifolds is \emph{multiplicative} if
\[
f(M \# N) = f(M)f(N)
\]
for any pair of oriented 3-manifolds $(M,N)$, where $\#$ denotes the connected sum, $f(-M) = \ol{f(M)}$, and $f(S^3) = 1$.
By~\cite[Corollary~2.9]{Funar} (cf.~\cite[Theorem~4.4]{TuTQFT}),
any multiplicative invariant canonically extends to a (2+1)-di\-men\-sional lax monoidal TQFT.
On the other hand, Funar~\cite[Corollary~1.1]{Funar2} constructed manifolds~$N$ and~$N'$ such that,
for any modular tensor category~$C$, their Reshetikhin-Turaev invariants agree:
\[
RT_C(N) = RT_C(N').
\]
By the work of Bartlett,  Douglas, Schommer-Pries, and Vicary~\cite{BDSV2}, every (1+1+1)-dimensional TQFT
is of the form $RT_C$ for some anomaly free modular tensor category~$C$.
It follows that~$N$ and~$N'$ cannot be distinguished by (1+1+1)-dimensional TQFTs.

Let $\{M_i \,\colon\, i \in \N \}$ be an enumeration of all prime oriented 3-manifolds such that
$M_0, \dots, M_n$ are the prime components of $N \sqcup N'$ without multiplicity, and let $p_0,\dots,p_n$
be distinct prime numbers. Then we define $f(S^3) = 1$ and $f(M_i) = p_i$ for every $i \in \{0,\dots,n\}$,
and $f(M_i)$ is an arbitrary complex number for $i > n$. As 3-manifolds have unique
prime decompositions, $f$ uniquely extends to a multiplicative invariant of 3-manifolds.
Since~$N$ and~$N'$ are not homeomorphic, they have distinct prime components, and so $f(N) \neq f(N')$
as they have different prime factorizations. It follows that the TQFT arising from such an~$f$
is not (1+1+1)-dimensional.
We have $2^\omega$ different choices for $f(M_i)$ for every $i>n$, giving rise to
$2^{2^\omega}$ different multiplicative invariants~$f$, each distinguishing $N$ and $N'$.

Alternatively, by the work of Bruillard, Ng, Rowell, and Wang~\cite[Theorem~3.1]{modular}, there are only countably many
modular tensor categories up to equivalence, while there are $2^{2^\omega}$ multiplicative 3-manifold
invariants, so, with countably many exceptions, a (2+1)-dimensional lax monoidal TQFTs is not (1+1+1)-dimensional.
\end{proof}

\bibliographystyle{amsplain}
\bibliography{topology}
\end{document}

%% file: gradients.pdf_tex
\begingroup%
  \makeatletter%
  \providecommand\color[2][]{%
    \errmessage{(Inkscape) Color is used for the text in Inkscape, but the package 'color.sty' is not loaded}%
    \renewcommand\color[2][]{}%
  }%
  \providecommand\transparent[1]{%
    \errmessage{(Inkscape) Transparency is used (non-zero) for the text in Inkscape, but the package 'transparent.sty' is not loaded}%
    \renewcommand\transparent[1]{}%
  }%
  \providecommand\rotatebox[2]{#2}%
  \ifx\svgwidth\undefined%
    \setlength{\unitlength}{263.39065164bp}%
    \ifx\svgscale\undefined%
      \relax%
    \else%
      \setlength{\unitlength}{\unitlength * \real{\svgscale}}%
    \fi%
  \else%
    \setlength{\unitlength}{\svgwidth}%
  \fi%
  \global\let\svgwidth\undefined%
  \global\let\svgscale\undefined%
  \makeatother%
  \begin{picture}(1,0.79833773)%
    \put(0,0){\includegraphics[width=\unitlength,page=1]{gradients.pdf}}%
    \put(0.55375674,0.38732793){\color[rgb]{0,0,0}\makebox(0,0)[lb]{\smash{$p'$}}}%
    \put(0.43867854,0.44330745){\color[rgb]{0,0,0}\makebox(0,0)[lb]{\smash{$p$}}}%
    \put(0.55861303,0.3112947){\color[rgb]{0,0,0}\makebox(0,0)[lb]{\smash{$\mathcal{E'}$}}}%
    \put(0.37061094,0.37626288){\color[rgb]{0,0,0}\makebox(0,0)[lb]{\smash{$\partial_-\mathcal{E}$}}}%
    \put(0.42593856,0.00611022){\color[rgb]{0,0,0}\makebox(0,0)[lb]{\smash{$a(\mathbb{S})$}}}%
    \put(0.93650839,0.13497257){\color[rgb]{0,0,0}\makebox(0,0)[lb]{\smash{$M$}}}%
    \put(0.92730395,0.75013664){\color[rgb]{0,0,0}\makebox(0,0)[lb]{\smash{$M'$}}}%
    \put(0.93344012,0.44025345){\color[rgb]{0,0,0}\makebox(0,0)[lb]{\smash{$W$}}}%
    \put(0.38730949,0.15031337){\color[rgb]{0,0,0}\makebox(0,0)[lb]{\smash{$K_p$}}}%
    \put(0.52077406,0.15184733){\color[rgb]{0,0,0}\makebox(0,0)[lb]{\smash{$K_{p'}$}}}%
    \put(0.48828932,0.44964013){\color[rgb]{0,0,0}\makebox(0,0)[lb]{\smash{$\partial_+\mathcal{E}$}}}%
    \put(0.64911733,0.73947893){\color[rgb]{0,0,0}\makebox(0,0)[lb]{\smash{$d(b(\mathbb{S}))$}}}%
  \end{picture}%
\endgroup%

%% file: operations.pdf_tex
\begingroup%
  \makeatletter%
  \providecommand\color[2][]{%
    \errmessage{(Inkscape) Color is used for the text in Inkscape, but the package 'color.sty' is not loaded}%
    \renewcommand\color[2][]{}%
  }%
  \providecommand\transparent[1]{%
    \errmessage{(Inkscape) Transparency is used (non-zero) for the text in Inkscape, but the package 'transparent.sty' is not loaded}%
    \renewcommand\transparent[1]{}%
  }%
  \providecommand\rotatebox[2]{#2}%
  \ifx\svgwidth\undefined%
    \setlength{\unitlength}{355.41378485bp}%
    \ifx\svgscale\undefined%
      \relax%
    \else%
      \setlength{\unitlength}{\unitlength * \real{\svgscale}}%
    \fi%
  \else%
    \setlength{\unitlength}{\svgwidth}%
  \fi%
  \global\let\svgwidth\undefined%
  \global\let\svgscale\undefined%
  \makeatother%
  \begin{picture}(1,0.68900572)%
    \put(0,0){\includegraphics[width=\unitlength,page=1]{operations.pdf}}%
    \put(0.48431018,0.64274685){\color[rgb]{0,0,0}\makebox(0,0)[lb]{\smash{$\alpha_4$}}}%
    \put(0,0){\includegraphics[width=\unitlength,page=2]{operations.pdf}}%
    \put(0.48431018,0.49701473){\color[rgb]{0,0,0}\makebox(0,0)[lb]{\smash{$\omega_4$}}}%
    \put(0,0){\includegraphics[width=\unitlength,page=3]{operations.pdf}}%
    \put(0.37024751,0.47392804){\color[rgb]{0,0,0}\makebox(0,0)[lb]{\smash{$\mathbb{P}_4$}}}%
    \put(0,0){\includegraphics[width=\unitlength,page=4]{operations.pdf}}%
    \put(0.4779383,0.33834972){\color[rgb]{0,0,0}\makebox(0,0)[lb]{\smash{$\delta_{3,1}$}}}%
    \put(0.31632823,0.36536934){\color[rgb]{0,0,0}\makebox(0,0)[lb]{\smash{$s_3$}}}%
    \put(0,0){\includegraphics[width=\unitlength,page=5]{operations.pdf}}%
    \put(0.47793834,0.18797811){\color[rgb]{0,0,0}\makebox(0,0)[lb]{\smash{$\mu_{3,1}$}}}%
    \put(0.3705433,0.67277634){\color[rgb]{0,0,0}\makebox(0,0)[lb]{\smash{$l_4$}}}%
    \put(0,0){\includegraphics[width=\unitlength,page=6]{operations.pdf}}%
    \put(0.30428364,0.21193188){\color[rgb]{0,0,0}\makebox(0,0)[lb]{\smash{$\mathbb{P}_{3,1}$}}}%
    \put(0.33824951,0.03566107){\color[rgb]{0,0,0}\makebox(0,0)[lb]{\smash{$\emptyset$}}}%
    \put(0,0){\includegraphics[width=\unitlength,page=7]{operations.pdf}}%
    \put(0.49774584,0.06259613){\color[rgb]{0,0,0}\makebox(0,0)[lb]{\smash{$\varepsilon$}}}%
    \put(0,0){\includegraphics[width=\unitlength,page=8]{operations.pdf}}%
    \put(0.49774615,0.01155969){\color[rgb]{0,0,0}\makebox(0,0)[lb]{\smash{$\tau$}}}%
    \put(0,0){\includegraphics[width=\unitlength,page=9]{operations.pdf}}%
  \end{picture}%
\endgroup%

%% file: relation.pdf_tex
\begingroup%
  \makeatletter%
  \providecommand\color[2][]{%
    \errmessage{(Inkscape) Color is used for the text in Inkscape, but the package 'color.sty' is not loaded}%
    \renewcommand\color[2][]{}%
  }%
  \providecommand\transparent[1]{%
    \errmessage{(Inkscape) Transparency is used (non-zero) for the text in Inkscape, but the package 'transparent.sty' is not loaded}%
    \renewcommand\transparent[1]{}%
  }%
  \providecommand\rotatebox[2]{#2}%
  \ifx\svgwidth\undefined%
    \setlength{\unitlength}{354.05633274bp}%
    \ifx\svgscale\undefined%
      \relax%
    \else%
      \setlength{\unitlength}{\unitlength * \real{\svgscale}}%
    \fi%
  \else%
    \setlength{\unitlength}{\svgwidth}%
  \fi%
  \global\let\svgwidth\undefined%
  \global\let\svgscale\undefined%
  \makeatother%
  \begin{picture}(1,0.7200383)%
    \put(0,0){\includegraphics[width=\unitlength,page=1]{relation.pdf}}%
    \put(0.63508744,0.67952664){\color[rgb]{0,0,0}\makebox(0,0)[lb]{\smash{$\delta_{i,n-i}$}}}%
    \put(0.86371255,0.51030963){\color[rgb]{0,0,0}\makebox(0,0)[lb]{\smash{$\mu_{i,n-i}$}}}%
    \put(0.80222021,0.71169783){\color[rgb]{0,0,0}\makebox(0,0)[lb]{\smash{$\mathbb{P}_{i,n-i}$}}}%
    \put(0.43843718,0.65085196){\color[rgb]{0,0,0}\makebox(0,0)[lb]{\smash{$s_i$}}}%
    \put(0.12857804,0.50861368){\color[rgb]{0,0,0}\makebox(0,0)[lb]{\smash{$\omega_n$}}}%
    \put(0.30637164,0.6754957){\color[rgb]{0,0,0}\makebox(0,0)[lb]{\smash{$\psi^{-1}$}}}%
    \put(0.508115,0.20231776){\color[rgb]{0,0,0}\makebox(0,0)[lb]{\smash{$\nu$}}}%
    \put(0.85753484,0.19828284){\color[rgb]{0,0,0}\makebox(0,0)[lb]{\smash{$(\nu^{s_i})^{-1}$}}}%
    \put(0.21440826,0.20473865){\color[rgb]{0,0,0}\makebox(0,0)[lb]{\smash{$\sigma_{n+1,i}$}}}%
    \put(0.30421148,0.37785684){\color[rgb]{0,0,0}\makebox(0,0)[lb]{\smash{$\psi^{\bP_n}$}}}%
    \put(0.63661887,0.09351126){\color[rgb]{0,0,0}\makebox(0,0)[lb]{\smash{$\alpha_{n+1}$}}}%
    \put(0.46176394,0.70468169){\color[rgb]{0,0,0}\makebox(0,0)[lb]{\smash{$\bP$}}}%
    \put(0.4570183,0.00761667){\color[rgb]{0,0,0}\makebox(0,0)[lb]{\smash{$\S_{n+1}$}}}%
    \put(0.83278299,0.28209409){\color[rgb]{0,0,0}\makebox(0,0)[lb]{\smash{$\S_n$}}}%
    \put(0.84129214,0.00191048){\color[rgb]{0,0,0}\makebox(0,0)[lb]{\smash{$\S_n$}}}%
    \put(0.79805165,0.59318974){\color[rgb]{0,0,0}\makebox(0,0)[lb]{\smash{$\S_i \sqcup \S_{n-i}$}}}%
    \put(0.46199508,0.28077008){\color[rgb]{0,0,0}\makebox(0,0)[lb]{\smash{$\S_n(\bP)$}}}%
    \put(0.10260737,0.59258776){\color[rgb]{0,0,0}\makebox(0,0)[lb]{\smash{$\S_n$}}}%
    \put(0.10525411,0.27842215){\color[rgb]{0,0,0}\makebox(0,0)[lb]{\smash{$\S_{n+1}$}}}%
    \put(0.23317003,0.66038934){\color[rgb]{0,0,0}\makebox(0,0)[lb]{\smash{$\bP_n$}}}%
    \put(0.48644697,0.59145351){\color[rgb]{0,0,0}\makebox(0,0)[lb]{\smash{$\S_n$}}}%
    \put(0.54514931,0.12479462){\color[rgb]{0,0,0}\makebox(0,0)[lb]{\smash{$l_{n+1}$}}}%
  \end{picture}%
\endgroup%